\RequirePackage{fix-cm}
\documentclass[smallextended]{svjour3}
\smartqed

\usepackage{graphics,graphicx,epstopdf}
\usepackage[leqno]{amsmath}
\usepackage{amsxtra,amsfonts,amscd,amssymb,bm,epsf}
\usepackage{cases}
\usepackage{subfig}
\usepackage{xcolor}
\usepackage{stmaryrd}
\usepackage{dsfont}
\usepackage{enumitem}
\usepackage{mathtools}

\usepackage{setspace,url}
\usepackage{tabularx}
\usepackage{longtable}
\usepackage{comment}
\usepackage[algo2e,linesnumbered,vlined,ruled]{algorithm2e}
\usepackage[normalem]{ulem}
\usepackage{times}

\ifpdf
  \DeclareGraphicsExtensions{.eps,.pdf,.png,.jpg}
\else
  \DeclareGraphicsExtensions{.eps}
\fi

\numberwithin{definition}{section}
\numberwithin{corollary}{section}
\numberwithin{lemma}{section}
\numberwithin{theorem}{section}
\numberwithin{equation}{section}

\usepackage{amsopn}

\definecolor{tumb}{RGB}{0,101,189}

\newcommand{\R}{\mathbb{R}}
\newcommand{\Rn}{\mathbb{R}^n}
\newcommand{\C}{\mathbb{C}}
\newcommand{\Cn}{\mathbb{C}^n}
\newcommand{\CSn}{\mathbb{CS}^{n-1}}

\newcommand{\N}{\mathbb{N}}

\newcommand{\Sn}{\mathbb{S}^{n-1}}

\newcommand{\cA}{\mathcal{A}}  
\newcommand{\cB}{\mathcal{B}}

\newcommand{\cI}{\mathcal{I}}

\newcommand{\Hess}{\mathrm{Hess\!\;}}
\newcommand{\grad}{\mathrm{grad\!\;}}

\newcommand{\half}{\frac{1}{2}}

\newcommand{\st}{\mathrm{s.\,t.}}

\DeclareMathOperator*{\argmin}{arg\,min}

\DeclareMathOperator*{\diag}{diag}

\newcommand{\be}{\begin{equation}}
\newcommand{\ee}{\end{equation}}
\newcommand{\bee}{\begin{equation*}}
\newcommand{\eee}{\end{equation*}}

\newcommand{\bea}{\begin{eqnarray}}
\newcommand{\eea}{\end{eqnarray}}
\newcommand{\beaa}{\begin{eqnarray*}}
\newcommand{\eeaa}{\end{eqnarray*}}

\begin{document}

\title{On the geometric analysis of a quartic-quadratic optimization problem under a spherical constraint
\thanks{H$.$ Zhang is partly supported by the elite undergraduate training program of School of Mathematical Sciences in Peking University.  Z$.$ Wen is supported in part by the NSFC grants 11831002 and 11421101. A$.$ Milzarek is partly supported by the Fundamental Research Fund -- Shenzhen Research Institute for Big Data (SRIBD) Startup Fund JCYJ-AM20190601.}
}

\titlerunning{Geometric analysis of a quartic-quadratic optimization problem}        

\author{Haixiang Zhang         \and
        Andre Milzarek         \and
        Zaiwen Wen        	   \and
        Wotao Yin
}


\institute{Haixiang Zhang \at
              School of Mathematical Sciences, Peking University, CHINA \\
              \email{1500010620@pku.edu.cn}           
           \and
           Andre Milzarek \at
              Institute for Data and Decision Analytics, Chinese University of Hong Kong, Shenzhen, CHINA \\
              \email{andremilzarek@cuhk.edu.cn}
           \and
           Zaiwen Wen \at
              Beijing International Center for Mathematical Research, Peking University, CHINA \\
              \email{wenzw@pku.edu.cn}
           \and
           Wotao Yin \at
              Department of Mathematics, University of California, Los Angeles, CA \\
              \email{wotaoyin@math.ucla.edu}
}

\date{Received: date / Accepted: date}

\maketitle
\begin{abstract}

This paper considers the problem of solving a special
quartic-quadratic optimization problem with a single sphere constraint, namely,
finding a global and local minimizer of
$\frac{1}{2}\mathbf{z}^{*}A\mathbf{z}+\frac{\beta}{2}\sum_{k=1}^{n}\lvert
z_{k}\rvert^{4}$ such that $\lVert\mathbf{z}\rVert_{2}=1$. This problem spans
multiple domains including quantum mechanics and chemistry sciences and we
investigate the geometric properties of this optimization problem. Fourth-order
optimality conditions are derived for characterizing local and global minima.
When the matrix in the quadratic term is diagonal, the problem has no spurious local minima and global solutions can be represented explicitly and calculated in $O(n\log{n})$ operations. When $A$ is a rank one matrix, the global minima of the problem are unique under certain phase shift schemes. The strict-saddle property, which can imply polynomial time convergence of
second-order-type algorithms, is established when the coefficient $\beta$ of the quartic
term is either at least $O(n^{3/2})$ or not larger than $O(1)$. Finally, the Kurdyka-\L ojasiewicz exponent of quartic-quadratic problem is estimated  and it is shown that the exponent is ${1}/{4}$ for a broad class of stationary points.
\end{abstract}

\section{Introduction}

In this paper, we analyze the geometric properties of the following nonconvex quartic-quadratic problem under a single spherical constraint,
\begin{align}
\min_{\mathbf{z}\in\mathbb{C}^{n}}~f(\mathbf{z})=\frac{1}{2}\mathbf{z}^{*}A\mathbf{z}+\frac{\beta}{2}\sum_{k\in[n]}\lvert z_{k}\rvert^{4}\quad\mathrm{s.t.}\quad\lVert\mathbf{z}\rVert_{2}=1,
\label{eqn:obj}
\end{align}
where $\beta > 0$ is a fixed interaction coefficient and $A \in {\mathbb C}^{n \times n}$ is a given Hermitian matrix. An important class of applications of this type is the so-called Bose-Einstein condensation (BEC) problem, which has attracted great interests in the atomic, molecule and optical physics community and in the condense matter community. Utilizing a proper non-dimensionalization and discretization, the BEC problem can be rewritten as a quartic-quadratic minimization problem of the form \eqref{eqn:obj}, where the matrix $A$ corresponds to the sum of the discretized Laplace operator and a diagonal matrix. If a non-rotating BEC problem is considered, then the variable $\mathbf{z}$ can be restricted to the real space $\mathbb{R}^{n}$ and problem \eqref{eqn:obj} becomes a real optimization problem. For a more detailed setup of the BEC problem and its specific mathematical formulation, we refer to \cite{griffin1996bose,bao2012mathematical,pethick2002bose}. 

Our interest in problem \eqref{eqn:obj} and its geometric properties is primarily triggered by related numerical results and observations with Bose-Einstein condensates and Kohn-Sham density functional calculations, see, e.g., \cite{WenYin13,hu2017adaptive,wu2017regularized,GaoLiuCheYua18}, and is motivated by recent landscape results for matrix completion \cite{ge2016matrix,sun2016guaranteed,ge2017no}, phase retrieval \cite{sun2016geometric,chen2018gradient}, phase synchronization \cite{bandeira2016low,boumal2016nonconvex,liu2017estimation}, and quadratic programs with spherical constraints  \cite{gao2016ojasiewicz,liu2017quadratic}. Understanding the geometric landscape of the nonconvex optimization problem \eqref{eqn:obj} is a fundamental step towards understanding and explaining the global and local behavior of the problem and the performance of associated algorithms. Despite recent progress on the geometric properties of nonconvex minimization problems and due to the complex interaction of the quadratic and quartic terms, the landscape of \eqref{eqn:obj} is still elusive. We further note that in \cite{hu2016note}, Hu et al. have shown that the minimization problem \eqref{eqn:obj} can be interpreted as a special instance of the partition problem and thus, it is generally NP-hard to solve \eqref{eqn:obj}.

\subsection{Related Work and Geometric Concepts}

Although nonconvex optimization problems are generally NP-hard, \cite{murty1987some}, direct and traditional minimization approaches, such as basic gradient and trust region schemes, can still be applied to solve certain and important classes of nonconvex problems -- with astonishing success -- and they remain the methods of choice for the practitioner \cite{boumal2016nonconvex}. A recent and steadily growing area of research concentrates on the identification of such classes of problems and tries to close the discrepancy between theoretical results and numerical performances, see, e.g., \cite{sun2016nonconvex,jain2017non,chi2018nonconvex} for an overview. Herein geometric observations and techniques play a major role in understanding the landscape and the global and local behavior of a nonconvex problem and of associated algorithms. 
Specifically, we are interested in the following geometric properties: 
\setlength{\leftmargini}{7ex}
\begin{itemize}
\item[$(\mathcal P_1)$] All local minimizers are also global solutions, i.e., there are no spurious local minimizers. 
\item[$(\mathcal P_2)$] The objective function possesses negative curvature directions at all saddle points and local maximizers which allows to effectively escape those points. 
\end{itemize} 
Condition $(\mathcal P_2)$ is the basis of the so-called \textit{strict-saddle property} and was introduced in \cite{ge2015escaping,sun2016geometric,ge2017no}. The strict-saddle and other related conditions can be used in the convergence analysis and in the design of algorithms to efficiently avoid saddle points. For instance, Sun, Qu, and Wright \cite{SunQuWri18} established a polynomial-time convergence rate of a Riemannian trust region method that is tailored to solve phase retrieval problems which satisfy the strict-saddle property. Furthermore, in \cite{lee2016gradient,panageas2016gradient,lee2017first-order} it is shown that certain randomly initialized first-order methods can converge to local minimizers and escape saddle points almost surely if the strict-saddle property holds. In the following, we briefly review recent classes of nonconvex optimization problems for which the conditions $(\mathcal P_1)$, $(\mathcal P_2)$, or other desirable geometric properties are satisfied. 

The generalized phase retrieval (GPR) problem is a popular nonconvex problem which has seen remarkable progress these years, see, e.g., \cite{jaganathan2015phase,shechtman2015phase} for an overview. Classical methods that transform the GPR problem into a convex program include convex relaxation techniques \cite{candes2014solving,candes2013phaselift,chen2015exact} and Wirtinger flow algorithms with carefully-designed initialization \cite{candes2015phase}. Phase retrieval problems are typically formulated as a quartic and unconstrained least squares problem depending on $m$ measurements ${\bf y}_k = |{\bf a}_k^*{\bf z}|$, $k = 1,...,m$. Traditional GPR methods can recover the true signal ${\bf z}$ from the measurements as long as the sample size $m$ satisfies $m \gtrapprox n$ or $m \gtrapprox n \log n$ where $n$ is the dimension of the signal. A provable convergence rate for a randomly initialized trust region-type algorithm is given in \cite{sun2016geometric} as long as $m \gtrapprox n\log^{3}{n}$ via showing that all the local minima are global and the strict-saddle property holds. When the signal and observations are real, the convergence rate of the vanilla gradient descent method is established by Chen et al. \cite{chen2018gradient} under the assumption $m \gtrapprox n\log^{13}{n}$.   
Another interesting class of amenable nonconvex problems are low-rank matrix factorization problems. Classical methods for matrix factorization are based on nuclear norm minimization \cite{candes2010power,recht2011simpler} and are usually memory intensive or require long running times.
In \cite{keshavan2010matrix1,keshavan2010matrix2}, Keshavan, Montanari, and Oh showed that the well-initialized gradient descent method can recover the ground truth of those problems. A strong convexity-type property is proved to hold around the optimal solution by Sun and Luo in \cite{sun2016guaranteed} and the objective function is shown to be sharp and weakly convex in nonsmooth settings by Li et al., \cite{li2018nonconvex}. Further, the strict-saddle property for the low-rank matrix factorization problem is established in \cite{ge2017no,ge2016matrix}, as well as for other low rank problems such as robust PCA and matrix sensing. Other classes of nonconvex optimization problems with provable convergence or geometric properties comprise orthogonal tensor decomposition \cite{ge2015escaping,ge2017optimization}, complete dictionary learning \cite{arora2015simple,sun2017complete1,sun2017complete2}, phase synchronization and community detection \cite{bandeira2016low,boumal2016nonconvex,liu2017estimation} and shallow neural networks \cite{liang2018understanding}. There are also several numerical methods that work well in practice for solving the BEC problem (e.g., tools for numerical partial differential equation \cite{adhikari2000numerical,edwards1995numerical} or optimization methods \cite{garcia2001optimizing,hu2016note,wu2017regularized}), but their geometric properties are not known.

So far the mentioned concepts allow to cover global structures and landscapes. Instead, local properties and the local behavior of \eqref{eqn:obj} can be captured by the so-called Kurdyka-{\L}ojasiewicz (KL), \cite{Kur98}, or {\L}ojasiewicz inequality, \cite{lojasiewicz1963propriete}. The {\L}ojasiewicz inequality is a useful tool to estimate the convergence rate of first-order iterative methods in the nonconvex setting \cite{absil2005convergence,merlet2013convergence,schneider2015convergence}. 
Moreover, the convergence rate of first-order methods satisfying a certain line-search criterion and descent condition can be derived via the KL inequality, \cite{attouch2009convergence,bolte2014proximal,schneider2015convergence}, where the rate depends on the KL exponent $\theta$. However, there is no general method to determine or estimate the KL exponent of specific optimization problems, though  the existence of the KL exponent is guaranteed in many situations. For optimizing a real analytic function over a compact real analytic manifold (such as problem \eqref{eqn:obj}), the existence of the KL exponent is established by {\L}ojasiewicz in \cite{lojasiewicz1963propriete}. There are also several few works that derive explicit estimates of the KL exponent for certain structured problems, such as general polynomials \cite{Gwo99,d2005explicit,Yan08}, convex problems \cite{li2018calculus}, non-convex quadratic optimization problems with simple convex constraints \cite{forti2006convergence,li2018calculus,luo1994error,luo2000error}, and quadratic optimization problems with single spherical constraint \cite{gao2016ojasiewicz,liu2017quadratic}. Obviously, the above four cases do not cover our constrained quartic-quadratic optimization problem \eqref{eqn:obj}. 

\subsection{Contributions}

In this work, we investigate different geometric concepts for the quadratic-quartic optimizations problem \eqref{eqn:obj} and give theoretical explanations why first- and second-order methods can perform well on it. In section \ref{sec:notation}, we first derive several new second- and fourth-order optimality conditions for problem \eqref{eqn:obj} that can be utilized to characterize local and global solutions. These conditions capture fundamental geometric properties of stationary points and local minima and form the basis of our geometric analysis. We then investigate problem (\ref{eqn:obj}) in the special case where $A$ is a diagonal matrix. In this situation, we show that a complete characterization of the landscape can be obtained and that problem (\ref{eqn:obj}) does not possess any spurious local minima. Furthermore, global solutions can be computed explicitly using a closed-form expression that involves the projection onto an $n$-simplex which requires $O(n \log{n})$ operations. These results can be partially extended to the case where $A$ is a rank-one matrix and we can prove uniqueness of global minima up to a certain phase shift. In general, the complex interplay between the quartic and quadratic terms impedes the derivation of explicit expressions for stationary points and local minima and complicates the landscape analysis of $f$ significantly. However, if either the quartic or the quadratic term dominates the objective function, we can establish the strict-saddle property $(\mathcal P_2)$ and identify and calculate the location and number of local minima. Our methodology is based on a careful discussion of the quartic and quadratic terms for large and small interaction coefficients that is applicable for general deterministic and arbitrary choices of $A$. We note that previous works and results only cover fourth-order unconstrained optimization problems (e.g., phase retrieval), quadratic constrained optimization problems (e.g., matrix completion and phase synchronization), or fourth-order constrained optimization problems without quadratic terms (e.g., fourth-order tensor decomposition). In particular, there is no interaction between quartic and quadratic terms
 and between their Riemannian derivatives. Different from most nonconvex problems discussed in the literature, our problem does not have a natural probabilistic framework and thus, probabilistic techniques such as concentration inequalities can not be directly applied. 



In addition, we estimate the KL exponent and establish a Riemannian \L ojasiewicz-type inequality for problem (\ref{eqn:obj}). Again, the presence of the quartic term 
considerably complicates the theoretical analysis. In order to deal with the high-order terms appearing in the Taylor expansion, we first separate the nonzero and zero components of a stationary point in order to facilitate the discussion of the leading terms. Then we divide the proof into several cases corresponding to different leading terms. The appearance of the quartic term requires the third-order and the fourth-order terms in the Taylor expansion to fully describe the local behavior, rather than merely the second-order terms. Due to the additional terms, the number of possible leading terms is significantly increased and we carefully analyze the relationship between those different terms. If the matrix $A$ is diagonal, we show that the \L ojasiewicz inequality holds at every stationary point of \eqref{eqn:obj} with exponent $\theta = \frac14$. Moreover, this result can be extended to more general choices of $A$, if the problem is restricted to the real space and positive semi-definiteness of the stationary certification matrix is assumed. The proof is based on the diagonal case and on estimates of the local behavior of the objective function and the Riemannian gradient in different subspaces. The positive semi-definiteness assumption is utilized throughout the proof to handle the non-isolated case and can not be easily removed. Although this additional condition represents a stronger notion of global optimality, a wide range of global minima in the real case satisfy this condition.  
To the best of the authors' knowledge, our work is the first to estimate and analyze these properties for quadratic-quartic optimization problems over a single sphere.

\subsection{Organization and notations}

This paper is organized as follows. In section \ref{sec:notation}, we present second- and fourth-order optimality conditions and characterize global minimizer of problem \eqref{eqn:obj}. Next, in section \ref{sec:diagonal} and section \ref{sec:rank1}, we consider two special cases and investigate geometric properties of problem \eqref{eqn:obj} when $A$ is either diagonal or has rank one. General landscape results for the real case are discussed in section \ref{sec:strict_saddle}. Finally, in section \ref{sec:KLExp}, we estimate the KL exponent of problem \eqref{eqn:obj}. 

For $n \in \N$, we define $[n] := \{1,...,n\}$ and for ${\bf z} \in \Cn$, we set $\|{\bf z}\| = \|{\bf z}\|_2 = \sqrt{{\bf z}^*{\bf z}}$. 
Let $\Sn$ and $\mathbb{CS}^{n-1}$ denote the $n$-dimensional real and complex sphere, respectively. In the following sections, we will use the notation $\mathcal{M} = \mathbb{S}^{n-1}$ or $\mathcal{M} = \mathbb{CS}^{n-1}$ depending on whether we consider the real or the complex case. The tangent space of $\mathbb{CS}^{n-1}$ at a point $\mathbf{z} \in \mathbb{CS}^{n-1}$ is given by $\mathcal{T}_{\mathbf{z}}\mathcal{M}:=\{\mathbf{v}\in\mathbb{C}^{n}:\Re({\bf v}^{*}{\bf z})=0\}$.  
For ${\bf z} \in \Cn$, $\diag(\mathbf{z})$ is a diagonal matrix with diagonal entries $z_{1},...,z_{n}$ and we use $|{\bf z}|^2$ to denote the component-wise absolute value, $|{\bf z}|^2 = {\bf z} \odot \bar {\bf z}$, of ${\bf z}$. We use $I$ to denote the ($n \times n$) identity matrix.
%
%
%
The Euclidean and corresponding Riemannian gradient of $f$ at $\mathbf{z}$ on $\mathcal M$ are denoted by $\nabla f(\mathbf{z})$ and $\grad{f(\mathbf{z})}$. Similarly, $\nabla^{2}f(\mathbf{z})$ and $\Hess{f(\mathbf{z})}$ represent the Euclidean and Riemannian Hessian, respectively. 

Throughout this paper and without loss of generality we will assume that the matrix $A$ is positive definite. Furthermore, $A = P \Lambda P^*$ is an associated eigenvalue decomposition of the Hermitian matrix $A$ with $\lambda_1 \geq \lambda_2 \geq ... \geq \lambda_n > 0$, $\Lambda = \diag(\lambda_1,...,\lambda_n)$, $P = (\mathbf{p}_{1},...,\mathbf{p}_{n}) \in \C^{n \times n}$, and $P^*P = I$.

\section{Wirtinger Calculus and Optimality Conditions}
\label{sec:notation}

Since the real-valued objective function $f$ is nonanalytic in ${\bf z}$, we utilize the Wirtinger calculus \cite{kreutz2009complex,sorber2012unconstrained} to express the complex derivatives of $f$. Specifically, the Wirtinger gradient and Hessian of $f$ are defined as 
\begin{align*}
&\nabla f(\mathbf{z}):=
\begin{bmatrix}
\nabla_{\mathbf{z}}f\\
\nabla_{\bar{\mathbf{z}}f}\\
\end{bmatrix},\quad\nabla^{2}f(\mathbf{z}):=
\begin{bmatrix}
\frac{\partial}{\partial\mathbf{z}}(\frac{\partial f}{\partial\mathbf{z}})^{*}
&\frac{\partial}{\partial\bar{\mathbf{z}}}(\frac{\partial f}{\partial \mathbf{z}})^{*}\\
 & \\
\frac{\partial}{\partial\bar{\mathbf{z}}}(\frac{\partial f}{\partial\bar{\mathbf{z}}})^{*}
&\frac{\partial}{\partial\bar{\mathbf{z}}}(\frac{\partial f}{\partial\bar{\mathbf{z}}})^{*} 
\end{bmatrix},
\end{align*}
where $\nabla_{\mathbf{z}}f(\mathbf{z}):=({\partial f}/{\partial \mathbf{z}})^{*},\nabla_{\bar{\mathbf{z}}}f(\bar{\mathbf{z}}):=({\partial f}/{\partial \bar{\mathbf{z}}})^{*}$ and following \cite{wu2017regularized}, we obtain
%
$\nabla f_{\bf z}({\bf z}) = \frac12 A{\bf z} + \beta \diag(|{\bf z}|^2){\bf z}$, $\nabla_{\bar{\bf z}} f(\bar {\bf z}) = \overline{\nabla f_{\bf z}({\bf z})}$, and 
\[ \nabla^{2}f(\mathbf{z}) = \begin{bmatrix} \frac{1}{2}A+2\beta\diag(|{\bf z}|^2) & \beta \diag(z_{1}^{2},...,z_{n}^{2}) \\ \beta \diag(\bar{z}_{1}^{2},...,\bar{z}_{n}^{2}) & \frac{1}{2}\bar{A}+2\beta \diag(|{\bf z}|^2) \end{bmatrix}. \]
%
Furthermore, using the identification $\mathcal{T}_{\mathbf{z}}\mathcal{M} \equiv \{ {\bf v}, \bar {\bf v} \in \Cn: {\bf z}^*{\bf v} + \bar {\bf z}^T \bar {\bf v} = 0 \}$,  the Riemannian gradient and Hessian of $f$ are given by
\begin{align} \label{eq:lambda}\grad{f(\mathbf{z})}= \nabla f({\bf z}) -\lambda \begin{bmatrix} \mathbf{z} \\ \bar{\mathbf{z}} \end{bmatrix} \quad \text{and} \quad \Hess{f(\mathbf{z})} = \nabla^2 f({\bf z}) -\lambda I_{2n}, \end{align}
%
where $\lambda = {\bf z}^* \nabla_{\bf z} f({\bf z}) =  \frac{1}{2}\mathbf{z}^{*}A\mathbf{z}+\beta\lVert\mathbf{z}\rVert_{4}^{4} \in \R$, see, e.g., \cite[Section 3.6 and 5.5]{Absil2009Optimization}. Let us notice that $\grad{f}$ and $\Hess{f}$ coincide with the standard gradient and Hessian of the Lagrangian $L({\bf z},\mu) = f({\bf z},\bar{\bf z}) - \frac{\mu}{2} (\bar {\bf z}^T {\bf z} - 1)$ when choosing $\mu = \lambda$. Exploiting the symmetry in $\grad{f}$, the associated first-order optimality conditions for \eqref{eqn:obj} now take the form: 
\begin{equation} \label{eq:stat} \grad{f({\bf z})} = 0 \quad \iff \quad [A + 2\beta \diag(|{\bf z}|^2)] {\bf z} = 2\lambda {\bf z}. \end{equation}
A point ${\bf z} \in \Cn$ satisfying the conditions \eqref{eq:stat} will be called \textit{stationary point} of problem \eqref{eqn:obj}. We define the curvature of $f$ at ${\bf z}$ along a direction $\mathbf{v} \in \Cn$ via 
\begin{align*}
H_{f}(\mathbf{z})[\mathbf{v}] & := 
\begin{bmatrix}
\mathbf{v}^* & {\mathbf{v}}^T \end{bmatrix} \Hess{f(\mathbf{z})}
\begin{bmatrix}
\mathbf{v}\\
\bar{\mathbf{v}}
\end{bmatrix} \\
& =  \mathbf{v}^{*} [ A+4\beta \diag(|{\bf z}|^2)]{\bf v} +2\beta \cdot {\sum}_{k = 1}^n \Re(v_{k}^{2}{\bar z_{k}}^{2})-2\lambda\lVert\mathbf{v}\rVert^{2} \\
& =  \mathbf{v}^{*}[ A + 2\beta \diag(|{\bf z}|^2) - 2\lambda I] \mathbf{v} + 4\beta \cdot {\sum}_{k=1}^n \Re(v_k\bar z_k)^2
\end{align*}
In the real case, the latter formulae reduce to $\grad{f(\mathbf{z})} = [A+2\beta\diag(|{\bf z}|^2)]{\bf z}-2\lambda \mathbf{z}$, $\Hess{f(\mathbf{z})}=A+6\beta\diag(|{\bf z}|^2)-2\lambda I_{n}$, and $H_{f}(\mathbf{z})[\mathbf{v}] :=\mathbf{v}^{T}\Hess{f(\mathbf{z})}\mathbf{v}$.

\subsection{Second-Order Optimality Conditions}

Due to the analogy of the Riemannian expressions and the Lagrangian formalism, we can apply classical optimality results to describe the second-order optimality conditions of problem \eqref{eqn:obj}. In particular, by \cite[Theorem 12.5 and 12.6]{NoceWrig06} we have:

%
\begin{lemma}[Second-Order Necessary and Sufficient Conditions] \label{lemma:son-sos}
Suppose that $\mathbf{z}$ is a local solution of problem \eqref{eqn:obj}. Then, it holds that $\grad{f(\mathbf{z})}=0$ and we have $H_{f}(\mathbf{z})[\mathbf{v}]\geq0$ for all $\mathbf{v}\in\mathcal{T}_{\mathbf{z}}\mathcal{M}$. Conversely, if $\mathbf{z}$ is a stationary point satisfying $\grad f(\mathbf{z})=0$ and $H_{f}(\mathbf{z})[\mathbf{v}]>0$ for all $\mathbf{v}\in\mathcal{T}_{\mathbf{z}}\mathcal{M}\backslash\{0\}$, then $\mathbf{z}$ is an isolated local minimum of problem \eqref{eqn:obj}.
\end{lemma}

Next, for some ${\bf z} \in \Cn$ we define the equivalence class
\begin{align} \label{eq:equi-class} \llbracket {\bf z} \rrbracket := \{ {\bf y} \in \Cn: |y_k| = |z_k|, \;\; \forall~k \in [n] \}. \end{align}
The following theorem gives a general sufficient condition for a stationary point to be a global minimum of problem \eqref{eqn:obj}.  

\begin{theorem} \label{theorem:glob-suff} Let ${\bf z} \in \Cn$ be a stationary point of problem \eqref{eqn:obj} with corresponding multiplier $\lambda$ and suppose that the matrix
\begin{align} \label{eq:glob-suff} H := A + 2\beta \diag(|{\bf z}|^2) - 2\lambda I \succeq 0 \end{align}
is positive semidefinite. Then ${\bf z}$ is a global minimum and all global minima of problem \eqref{eqn:obj} belong to the equivalence class $\llbracket {\bf z} \rrbracket$.
\end{theorem}
\begin{proof} Let ${\bf y} \in \Cn$ be an arbitrary point with $\|{\bf y}\| = 1$ and let us introduce the polar coordinates $z_i =  r_i e^{i \theta_i}$, $y_i = t_i e^{i \phi_i}$ for $r_i, t_i \geq 0$, $\theta_i, \phi_i \in [0,2\pi]$ and all $i \in[n]$. Using the stationarity condition $\grad{f({\bf z})} = 0$ and $\|{\bf z}\| = \|{\bf y} \| = 1$, it holds that
%
\begin{align}
\nonumber f({\bf y}) - f({\bf z}) &  = \half {\bf y}^* A {\bf y} - \half {\bf z}^* [ 2\lambda {\bf z} - 2\beta \diag(|{\bf z}|^2) {\bf z}] +  \frac{\beta}{2} (\|{\bf y}\|_4^4 - \|{\bf z}\|_4^4) \\ \nonumber & \hspace{-0ex}  = \half {\bf y}^* H {\bf y} - \beta {\bf y}^* \diag(|{\bf z}|^2){\bf y} + \frac{\beta}{2}(\|{\bf y}\|_4^4 + \|{\bf z}\|_4^4) \\ & \hspace{-0ex}  = \half {\bf y}^* H {\bf y} + \frac{\beta}{2} \; {\sum}_{k=1}^n [t_k^2 - r_k^2]^2. 
\label{eq:objective_value}
\end{align}
%
Consequently, the positive semidefiniteness of $H$ yields $f({\bf y}) - f({\bf z}) \geq 0$ for all ${\bf y} \in \Cn$ with $\|{\bf y}\| = 1$. Suppose now that ${\bf y}$ is a global minimum with ${\bf y} \notin \llbracket {\bf z} \rrbracket$. In this case the last sum in the above expression is strictly positive which, together with the positive semi-definiteness of $H$ yields a contradiction.  
\end{proof}

%
%
If problem \eqref{eqn:obj} has two different global minimizers ${\bf y}$ and ${\bf z}$ with $\llbracket {\bf y} \rrbracket \cap \llbracket {\bf z} \rrbracket = \emptyset$, Theorem \ref{theorem:glob-suff} implies that $H$ can not be positive semidefinite. Moreover, if condition \eqref{eq:glob-suff} holds at a stationary point ${\bf z}$, it automatically has to hold at all global minimizers in $\llbracket {\bf z} \rrbracket$. 

The definiteness condition in Theorem \ref{theorem:glob-suff} can be equivalently rephrased as follows: The multiplier $\lambda$ associated with ${\bf z}$ is the minimum eigenvalue of the matrix $A + 2\beta \diag(|{\bf z}|^2)$ and ${\bf z}$ is the corresponding eigenvector. Characterizations of this type  are also known for (quadratic) trust-region subproblems and for general quadratic programs with quadratic constraints, see \cite{Mor93,VanThoai2005}. Furthermore, utilizing \cite[Theorem 3.1]{cai2018on}, it can be shown that such an eigenvector ${\bf z}$ with the stated properties exists under the assumption $0 < \beta \leq (\lambda_{n-1}-\lambda_n)/8$. In this case, the condition \eqref{eq:glob-suff} is necessary \text{and} sufficient for global optimality. 

\subsection{Fourth-order optimality conditions} \label{section:foc}
In the following section, we derive several fourth-order optimality conditions based on a special and finer expansion of the objective function $f$. In contrast to the sufficient conditions in Theorem \ref{theorem:glob-suff}, this allows us to fully characterize global optima. Let ${\bf z}\in\CSn$ be an arbitrary stationary point. For ${\bf v} \in \mathcal T_{\bf z}\mathcal M \cap \CSn$ and $\theta \in \R$, we consider the point ${\bf y}=\cos(\theta){\bf z}+\sin(\theta){\bf v} \in \CSn$. 
Using this decomposition in \eqref{eq:objective_value}, we obtain 
\begin{align*}
f({\bf y}) - f({\bf z}) & = \half {\bf y}^{*}H{\bf y}+\frac{\beta}{2}\sum_{k\in[n]}\left[|y_{k}|^{2}-|z_{k}|^{2}\right]^{2} \\ & = 
\frac{\sin^2(\theta)}{2} \left[ H_f({\bf z})[{\bf v}]  + 2 \beta \sin(2\theta) \cdot {\sum}_{k \in [n]} (|v_k|^2 - |z_k|^2) \Re(\bar z_k v_k) \right]  \\ & \hspace{4ex} +  \frac{\beta\sin^{4}(\theta)}{2} \left[ {\sum}_{k \in [n]} [(|v_{k}|^{2}-|z_{k}|^{2})^2 - 4\Re(\bar{z}_{k}v_{k})^2] \right]. 
\end{align*}
Defining $H_{3}({\bf v}) := \beta \sum_{k \in [n]}(|v_k|^2 - |z_k|^2)\Re(\bar z_k v_k)$, $H_4({\bf v}) := \beta {\sum}_{k \in [n]} [(|v_{k}|^{2}-|z_{k}|^{2})^2 - 4\Re(\bar{z}_{k}v_{k})^2]$ and dividing the latter equation by $\sin^4(\theta)$, this yields
\begin{align*} G({\bf v},\cot(\theta)) & =  \frac{2[f({\bf y}) - f({\bf z})]}{\sin^4(\theta)} \\ & = H_f({\bf z})[{\bf v}] \cdot \cot^2(\theta) + 4 H_3({\bf v}) \cdot \cot(\theta) + 2[f({\bf v}) - f({\bf z})] ,
\end{align*}
for $\theta \neq k\pi$, $k \in {\mathbb Z}$. 
We first propose necessary and sufficient optimality conditions for ${\bf z}$ being a local minimum.
\begin{theorem}[Characterization of Local Optimality]\label{thm:4order-local}
Let ${\bf z}$ be a stationary point of problem \eqref{eqn:obj}. Then ${\bf z}$ is locally optimal if and only if there exists a constant $M_{{\bf z}}>0$ such that $G({\bf v},t)\geq0$ for all ${\bf v}\in\mathcal{T}_{{\bf z}}\mathcal{M}\cap\CSn$ and all $t$ such that $|t|\geq M_{{\bf z}}$.
\end{theorem}
\begin{proof}
By definition, ${\bf z}$ is a local minimum if and only if there exists a constant $\delta_{{\bf z}}>0$ such that $f({\bf y})\geq f({\bf z})$ for all ${\bf y}\in B({\bf z},\delta_{{\bf z}})\cap\CSn$. Due to $\|{\bf y}-{\bf z}\|^{2}=2-2\cos(\theta)$, the condition ${\bf y}\in B({\bf z},\delta_{{\bf z}})\cap\CSn$ is equivalent to ${\bf y} = \cos(\theta) {\bf z} + \sin(\theta) {\bf v}$ with ${\bf v}\in\mathcal{T}_{{\bf z}}\mathcal{M}\cap\CSn$ and $\cos(\theta)\geq1-{\delta_{{\bf z}}^{2}}/{2}$. Moreover, $\cos(\theta)\geq1-{\delta_{{\bf z}}^{2}}/{2}$ implies
\begin{align*}
|\sin(\theta)|\leq\delta_{{\bf z}}\sqrt{1-{\delta_{{\bf z}}^{2}}/{4}},\quad|\cot(\theta)|\geq\frac{1-\delta_{{\bf z}}^{2}}{\delta_{{\bf z}}\sqrt{1-\delta_{{\bf z}}^{2}/4}} =: M_{\bf z}.
\end{align*}
Consequently, it holds that ${\bf y}\in B({\bf z},\delta_{{\bf z}})\cap\CSn$ if and only if ${\bf v}\in\mathcal{T}_{{\bf z}}\mathcal{M}\cap\CSn$ and $|\cot(\theta)|\geq M_{{\bf z}}$. Further, by definition of $G$, the necessary and sufficient conditions for ${\bf z}$ being a local minima are equivalent to: there exists a constant $M_{{\bf z}}>0$ such that $G({\bf v},t)\geq0$ for all ${\bf v}\in\mathcal{T}_{{\bf z}}\mathcal{M}\cap\CSn$ and all $t$ such that $|t|\geq M_{{\bf z}}$.
\end{proof}

Similarly, we can derive the fourth-order global optimality conditions.
\begin{theorem}[Characterization of Global Optimality] \label{theorem:glob-opt-full}
Let ${\bf z}$ be a stationary point of problem \eqref{eqn:obj}. Then ${\bf z}$ is a global solution if and only if
\begin{align*}
& H_f({\bf z})[{\bf v}] \geq 0,~2 H_{3}({\bf v})^2 \leq H_f({\bf z})[{\bf v}] \cdot [f({\bf v}) - f({\bf z})], \quad \forall~{\bf v}\in\mathcal{T}_{{\bf z}}\mathcal{M}\cap\CSn,\\
&H_{3}({\bf v}) = 0,~H_{4}({\bf v})\geq0, \quad \forall~{\bf v}\in\mathcal{T}_{{\bf z}}\mathcal{M}\cap\CSn~\text{with}~H_f({\bf z})[{\bf v}] =0.
\end{align*}
\end{theorem}
\begin{proof}
A stationary point ${\bf z}$ is a global minimum if and only if $f({\bf y})\geq f({\bf z})$ for all ${\bf y}\in\CSn$, which is equivalent to $G({\bf v},t)\geq0$ for all $t\in\mathbb{R}$ and all ${\bf v}\in\mathcal{T}_{{\bf z}}\mathcal{M}\cap\CSn$. However, nonnegativity of the (degenerated) quadratic function $t \mapsto G({\bf v},t)=H_f({\bf z})[{\bf v}] t^{2}+ 4 H_{3}({\bf v})t+ 2 [f({\bf v}) - f({\bf z})]$ on $\R$ is equivalent to the conditions stated in Theorem \ref{theorem:glob-opt-full}. 
\end{proof}

Finally, we establish fourth-order necessary conditions for local optimality.

\begin{theorem}[Fourth-Order Necessary Optimality Conditions] \label{theorem:foc-nes}
Let ${\bf z}$ be a local minima of problem \eqref{eqn:obj}. Then it holds that
\[ \left[ \begin{array}{l} H_f({\bf z})[{\bf v}]\geq0, \quad \forall~{\bf v}\in\mathcal{T}_{{\bf z}}\mathcal{M}\cap\CSn,\\[0.5ex]
H_{3}({\bf v})=0,~H_{4}({\bf v})\geq0, \quad \forall~{\bf v}\in\mathcal{T}_{{\bf z}}\mathcal{M}\cap\CSn~\text{with}~H_f({\bf z})[{\bf v}]=0,\\[0.5ex]
4H_{3}^{2}({\bf v},{\bf w})\leq H_f({\bf z})[{\bf w}]H_{4}({\bf v}),\\[0.5ex]
 \quad \forall~{\bf v},{\bf w}\in\mathcal{T}_{{\bf z}}\mathcal{M}\cap\CSn~\text{with}~H_f({\bf z})[{\bf v}]=0,~H_f({\bf z})[{\bf w}]>0,~{\bf v}^{*}{\bf w}=0,
 \end{array}  \right. \]
where $H_{3}({\bf v},{\bf w})=H_3({\bf v}) + 2\beta\sum_{k\in[n]}\Re(\bar{v}_{k}w_{k})\Re(\bar{v}_{k}z_{k})$.
\end{theorem}

\begin{proof}
Theorem \ref{thm:4order-local} implies that there is $M_{{\bf z}} > 0$ such that $G({\bf v},t)\geq0$ for all ${\bf v}\in\mathcal{T}_{{\bf z}}\mathcal{M}\cap\CSn$ and $|t|\geq M_{{\bf z}}$. Hence, for fixed ${\bf v}$, we have $G({\bf v},t) = H_f({\bf z})[{\bf v}] t^{2}+ 4 H_{3}({\bf v})t+ 2 [f({\bf v}) - f({\bf z})] \geq 0$ for $t\rightarrow\pm\infty$. Thus, it follows $H_f({\bf z})[{\bf v}] \geq 0$ and if $H_f({\bf z})[{\bf v}] = 0$, it must hold $H_{3}({\bf v}) = 0$ and $2[f({\bf v}) - f({\bf z})] = H_f({\bf z})[{\bf v}] + H_4({\bf v}) = H_4({\bf v})\geq0$.
%
%
Now we prove the last condition. Suppose ${\bf v}$ and ${\bf w}$ are two vectors in $\mathcal{T}_{{\bf z}}\mathcal{M}\cap\CSn$ satisfying $H_f({\bf z})[{\bf v}]=0$, $H_f({\bf z})[{\bf w}]>0$ and ${\bf v}^{*}{\bf w}=0$. Let ${\bf y}=\cos(\theta){\bf v}+\sin(\theta){\bf w}\in\mathcal{T}_{{\bf z}}\mathcal{M}\cap\CSn$, where $\theta\in[-\pi,\pi]\backslash\{0\}$. We consider the limiting process $\theta\rightarrow0$. Then, due to $H_f({\bf z})[{\bf v}] \geq 0$ for all ${\bf v} \in \mathcal{T}_{{\bf z}}\mathcal{M}$, it follows that $H_f({\bf z})[{\bf y}]=H_f({\bf z})[{\bf w}]\cdot\sin^{2}(\theta)$ and 
\begin{equation}
H_{3}({\bf y})=H_{3}({\bf v},{\bf w})\cdot\sin(\theta)+O(\sin^{2}(\theta)), \quad H_{4}({\bf y})=H_{4}({\bf v})+O(\sin(\theta)).
\label{eqn:foc-nes}
\end{equation}
We first discuss the case $H_{4}({\bf v})=0$. 
The discriminant of $G({\bf y},t)$ -- as a quadratic function of $t$ -- is given by 
\begin{align*}
16H_{3}^{2}({\bf y})-4H_f({\bf z})[{\bf y}](H_f({\bf z})[{\bf y}]+H_{4}({\bf y})) & =16H_{3}^{2}({\bf v},{\bf w})\sin^{2}(\theta)+O(\sin^{3}(\theta)).
\end{align*}
If $H_{3}({\bf v},{\bf w})\neq0$, this term is positive for all sufficiently small $\theta$ and hence, $G({\bf y},t)$ has two real roots. 
The absolute value of the larger root is
\begin{align*}
\frac{4|H_{3}({\bf y})|+\sqrt{16H_{3}^{2}({\bf y})-4H_f({\bf z})[{\bf y}](H_f({\bf z})[{\bf y}]+H_{4}({\bf y}))}}{2H_f({\bf z})[{\bf y}]}=\Theta\left(|\sin(\theta)|^{-1}\right),
\end{align*}
which implies that there does not exist a constant $M_{{\bf z}}>0$ such that $G({\bf y},t)\geq0$ for all ${\bf y}$ and $t$ such that ${\bf y}\in\mathcal{T}_{{\bf z}}\mathcal{M}\cap\CSn$ and $|t|\geq M_{{\bf z}}$. Thus, we have $H_{3}({\bf v},{\bf w})=0$ in this case. Next, we consider the case $H_{4}({\bf v})>0$ and let us suppose $4H_{3}^{2}({\bf v},{\bf w})>H_f({\bf z})[{\bf y}]H_{4}({\bf v})=\Theta(\sin^{2}(\theta))$. 
The discriminant of $G({\bf y},t)$ now satisfies
\begin{align*}
&16H_{3}^{2}({\bf y})-4H_f({\bf z})[{\bf y}](H_f({\bf z})[{\bf y}]+H_{4}({\bf y}))\\
& \hspace{6ex}=4\left[4H_{3}^{2}({\bf v},{\bf w})\sin^{2}(\theta)-H_f({\bf z})[{\bf w}]H_{4}({\bf v})\right]\sin^{2}(\theta)+O(\sin^{3}(\theta))>0,
\end{align*}
for all sufficiently small $\theta \neq 0$. As in the last case, the absolute value of the larger root of $t \mapsto G({\bf y},t)$ converges to $+\infty$ as $\theta \to 0$ which yields the same contradiction. Consequently, we have $4H_{3}^{2}({\bf v},{\bf w})\leq H_f({\bf z})[{\bf w}]H_{4}({\bf v})$ by combining the two cases.
\end{proof}

The fourth-order optimality conditions in Theorem \ref{theorem:glob-opt-full} and \ref{theorem:foc-nes} resemble other known fourth order conditions, see, e.g., \cite{Ded95,Pen17,CarGouToi18}, and might be hard to verify in practice. However, in the real case, the inequality $H_4({\bf v}) \geq 0$ is equivalent to checking $\||{\bf v}|^2 - |{\bf z}|\| \geq 2\sqrt{2} \|{\bf z}\|_4^2$. In this situation, the framework presented in \cite{AnaGe16} can be used to verify the first two conditions in Theorem \ref{theorem:foc-nes} in polynomial time.

\begin{theorem}[Fourth-Order Sufficient Optimality Conditions] \label{theorem:foc-suf}
Suppose that the point ${\bf z}\in\CSn$ satisfies the conditions
\[ \left[ \begin{array}{l}  H_f({\bf z})[{\bf v}]\geq0, \quad \forall~{\bf v}\in\mathcal{T}_{{\bf z}}\mathcal{M}\cap\CSn,\\[0.5ex]
H_{3}({\bf v})=0,~H_{4}({\bf v})\geq0, \quad \forall~{\bf v}\in\mathcal{T}_{{\bf z}}\mathcal{M}\cap\CSn~\text{with}~H_f({\bf z})[{\bf v}]=0,\\[0.5ex]
4H_{3}^{2}({\bf v},{\bf w})\leq H_f({\bf z})[{\bf w}]H_{4}({\bf v}),\\[0.5ex]
\quad \forall~{\bf v},{\bf w}\in\mathcal{T}_{{\bf z}}\mathcal{M}\cap\CSn~\text{with}~H_f({\bf z})[{\bf v}]=0,~H_f({\bf z})[{\bf w}]>0,~{\bf v}^{*}{\bf w}=0,
 \end{array}  \right. \]
where $H_{3}({\bf v},{\bf w})=H_3({\bf v}) + 2\beta\sum_{k\in[n]} \Re(\bar{v}_{k}w_{k})\Re(\bar{v}_{k}z_{k})$ and equality in the last inequality holds if and only if $H_{4}({\bf v})=0$. Then ${\bf z}$ is a local minimum of \eqref{eqn:obj}.
\end{theorem}

\begin{proof}
We prove that there exists a constant $M_{{\bf z}}>0$ such that $G({\bf v},t) \geq 0$ for all ${\bf v}\in\mathcal{T}_{{\bf z}}\mathcal{M}\cap\CSn$ and $t$ with $|t|\geq M_{{\bf z}}$. If this condition holds, then by Theorem \ref{thm:4order-local}, we know that ${\bf z}$ is a local minima of problem \eqref{eqn:obj}. If $H_{f}({\bf z})[{\bf v}]=0$, then it follows $G({\bf v},t)=H_{4}({\bf v})\geq0$ for all $t\in\mathbb{R}$.

Next, let $\epsilon>0$ be a small constant. If $H_{f}({\bf z})[{\bf v}]>\epsilon$, then the roots of the quadratic function $t \mapsto G({\bf v},t)$ are bounded by
\begin{align*}
\frac{4|H_{3}({\bf v})|+\sqrt{|16H_{3}^{2}({\bf v})-4H_{f}({\bf z})[{\bf v}](H_{f}({\bf z})[{\bf v}]+H_{4}({\bf v}))|}}{2H_{f}({\bf z})[{\bf v}]}\leq\frac{(4+2\sqrt{6})M}{2\epsilon},
\end{align*}
where $M :=\max_{{\bf v}\in\mathcal{T}_{{\bf z}}\mathcal{M}\cap\CSn}\max\left\{H_{f}({\bf z})[{\bf v}],|H_{3}({\bf v})|,|H_{4}({\bf v})|\right\}$. The continuity of the functions $H_{f}({\bf z})[\cdot]$, $|H_{3}(\cdot)|$, and $|H_{4}(\cdot)|$ implies $M<+\infty$. Hence, we have $G({\bf v},t)\geq0$ for all $t\in\mathbb{R}$ such that $|t|\geq (2+\sqrt{6})M\epsilon^{-1}$. We now consider the case $H_{f}({\bf z})[{\bf v}]\in(0,\epsilon]$. Let us define the decomposition
\begin{align*}
{\bf v}=\cos(\theta)\cdot{\bf u}+\sin(\theta)\cdot{\bf w},~H_{f}({\bf z})[{\bf u}]=0,~H_{f}({\bf z})[{\bf w}]>0,~{\bf u}^{*}{\bf w}=0,
\end{align*}
where $\theta\in[-\pi,\pi]$ and ${\bf u},{\bf w}\in\mathcal{T}_{{\bf z}}\mathcal{M}\cap\CSn$. Specifically, introducing the sets $\mathcal N := \{{\bf v} \in \Cn: H_f({\bf z})[{\bf v}] = 0\}$ and $\mathcal W := \mathcal{T}_{{\bf z}}\mathcal{M}\cap\CSn \cap \mathcal N^\bot$, there exists $\bar \sigma > 0$ such that $H_f({\bf z})[{\bf w}] \geq \bar \sigma$ for all ${\bf w} \in \mathcal W$. 
As before, we have $H_{f}({\bf z})[{\bf v}]=H_{f}({\bf z})[{\bf w}]\cdot\sin^{2}(\theta)$ and in the case $H_{4}({\bf u})=0$,  the condition $4H_{3}^{2}({\bf u},{\bf w})\leq H_f({\bf z})[{\bf w}]H_{4}({\bf u})$ implies $H_{3}({\bf u},{\bf w})=0$. Utilizing \eqref{eqn:foc-nes}, this yields $|H_{3}({\bf v})|\leq\eta\sin^{2}(\theta)$ for some universal constant $\eta>0$ and $H_{4}({\bf v})=O(\sin(\theta))$. 
If the discriminant of the quadratic function $t \mapsto G({\bf v},t)$ is negative, it follows $G({\bf v},t)\geq0$ for all $t\in\mathbb{R}$. Otherwise, if the discriminant is non-negative, then the absolute values of the roots are bounded by
\begin{align*}
&\frac{4|H_{3}({\bf v})|+\sqrt{16H_{3}^{2}({\bf v})-4H_{f}({\bf z})[{\bf v}](H_{f}({\bf z})[{\bf v}]+H_{4}({\bf v}))}}{2H_{f}({\bf z})[{\bf w}]\cdot\sin^{2}(\theta)}\\
& \hspace{10ex}\leq\frac{4\eta\sin^{2}(\theta)+\sqrt{16\eta^{2}-4H_{f}^{2}({\bf z})[{\bf w}]+O(\sin(\theta))}\sin^{2}(\theta)}{2H_{f}({\bf z})[{\bf w}]\cdot\sin^{2}(\theta)}\\
& \hspace{10ex}\leq\frac{4\eta+\sqrt{16\eta^{2}+1}}{2H_{f}({\bf z})[{\bf w}]}\leq\frac{4\eta+1}{\bar \sigma}.
\end{align*}
Consequently, it holds $G({\bf v},t)\geq0$ for all $t\in\mathbb{R}$ such that $|t|\geq {(4\eta+1)}{\bar\sigma}^{-1}$. Otherwise, if $H_{4}({\bf u})>0$, then the last condition of this theorem implies $4H_{3}^{2}({\bf u},{\bf w})<H_{f}({\bf z})[{\bf w}]H_{4}({\bf u})$. In this case, the discriminant of $G({\bf v},\cdot)$ satisfies
\begin{align*}
&16H_{3}^{2}({\bf v})-4H_{f}({\bf z})[{\bf v}](H_{f}({\bf z})[{\bf v}]+H_{4}({\bf v}))\\
& \hspace{10ex}=4[4H_{3}^{2}({\bf u},{\bf w})-H_{f}({\bf z})[{\bf w}]H_{4}({\bf u})]\cdot\sin^{2}(\theta)+O(\sin^{3}(\theta))<0,
\end{align*}
if $\theta$ is chosen sufficiently small and thus, we obtain $G({\bf v},t)\geq0$ for all $t\in\mathbb{R}$. Overall, we can set $M_{{\bf z}} :=\max\{(2+\sqrt{6})M\epsilon^{-1},(4\eta+1){\bar\sigma}^{-1}\}$ and ${\bf z}$ is a local minima of problem \eqref{eqn:obj}.

\end{proof}

\section{Geometric Analysis of the Diagonal Case}
\label{sec:diagonal}

In this section, we investigate the geometric properties of problem (\ref{eqn:obj}) under the assumption that $A$ is a diagonal matrix, i.e., $A=\diag(\mathbf{a}) = \diag(a_{1},a_{2},...,a_{n})$. 

By setting $u_{k}=\lvert z_{k}\rvert^{2}$, we can reformulate problem \eqref{eqn:obj} as a convex problem
\begin{align}
\min_{\mathbf{u}\in\mathbb{R}^{n}}~\frac{1}{2}\mathbf{a}^{T}\mathbf{u}+\frac{\beta}{2}\lVert\mathbf{u}\rVert^{2} \quad\mathrm{s.t.}\quad \mathbf{u}\in\Delta_{n}.
\label{eqn:obj_1}
\end{align}
where $\Delta_{n}=\{\mathbf{u}\in\mathbb{R}^{n}:u _{k}\in[0,1], k\in[n], \sum_{k\in[n]} u_{k}=1\}$ is the $n$-simplex. We will use this connection later to show that there are no spurious local minima in the diagonal case and that the global solutions can be characterized via the unique solution of the strongly convex problem \eqref{eqn:obj_1}.
%

We first derive an explicit representation of critical points of problem \eqref{eqn:obj}.
%
\begin{lemma}[Characterizing Stationary Points] Suppose that $A$ is a diagonal and let $\mathbf{z}\in\mathbb{CS}^{n-1}$ be given. Let us set $ \mathcal I :=  \{k \in [n]: z_k \neq 0\}$ and 
\[ u_k := 0, \quad \forall~k \in \mathcal I^{\sf C}, \quad u_k := \frac{1}{|\mathcal I|} + \frac{1}{2\beta}\left[ \frac{1}{|\mathcal I|} {\sum}_{i \in \mathcal I} a_i - a_k \right], \quad \forall~k \in \mathcal I. \]
Then, ${\bf z}$ is a stationary point if and only if there exist $\theta_{k}\in[0,2\pi)$, $k \in [n]$, such that $u_k \in (0,1]$ for all $k \in \mathcal I$ and $z_{k} = \sqrt{u_k}e^{i\theta_{k}}$ for all $k$.
\label{tho:diag_1}
\end{lemma}
\begin{proof}
In the diagonal case, introducing the polar form $\mathbf{z}=(r_{1}e^{i\theta_{1}},\dots,r_{n}e^{i\theta_{n}})^{T}$, the first-order optimality conditions reduce to
\begin{align*}
(a_{k}+2\beta r_k^{2} - 2\lambda) r_k e^{i \theta_k} = 0, \quad \forall~k\in[n]
\end{align*}
where $\lambda$ is the associated Lagrange multiplier. Specifically, for all $k \in \mathcal I$, we have $a_{k}+2\beta r_k^{2} = 2\lambda$ and summing these equations, we obtain
\[ \lambda = \frac{2 \beta + \sum_{k \in \mathcal I} a_k}{2|\mathcal I|}, \quad r_k^2 = \frac{2\lambda - a_k}{2\beta}, \]
and $2\lambda - a_k \in (0,2\beta]$ for all $k \in \mathcal I$. The claimed result in Lemma \ref{tho:diag_1} now follows immediately by setting $u_k = r_k^2$, $k \in [n]$.
%
%
\end{proof}

Next, we discuss the local minimizer of problem \eqref{eqn:obj}. By combining Theorems \ref{theorem:glob-suff} and \ref{theorem:char-glob-loc}, we see that there are no spurious local minimizer in the diagonal case, i.e., all local solutions are automatically global solutions of problem \eqref{eqn:obj}.

\begin{theorem}[Characterization of Local Minimizer] \label{theorem:char-glob-loc}
Let $A$ be a diagonal matrix. A point ${\bf z} \in \Cn$ is a local minimizer of problem \eqref{eqn:obj} if and only if
\begin{align}
\grad{f(\bf z)} = 0 \quad \text{and} \quad H = A + 2\beta \diag(|{\bf z}|^2) - 2\lambda I \succeq 0,
\label{eqn:3-7}
\end{align}
where $\lambda \in \R$ is the associated Lagrange multiplier. In addition, every local solution ${\bf z}$ can be represented explicitly and has to satisfy
\[ z_k = \sqrt{u_k} e^{i \theta_k}, \quad \theta_k \in [0,2\pi), \quad {\bf u} = \mathcal P_{\Delta_n}(-{\bf a}/2\beta), \quad \forall~k\in[n], \]
where $\mathcal P_{\Delta_n}$ denotes the Euclidean projection onto the $n$-simplex $\Delta_n$. 
\end{theorem}
\begin{proof}
According to Theorem \ref{theorem:glob-suff}, a point satisfying the conditions \eqref{eqn:3-7} is a global minimum of problem \eqref{eqn:obj} and hence, it also a local minimum. Let ${\bf z} \in \CSn$ now be an arbitrary local minimum. Then, the first- and second-order necessary optimality conditions hold at ${\bf z}$, i.e., we have $\grad{f({\bf z})} = 0$ and 
\begin{equation} \label{eq:soc-in-theo} H_f({\bf z})[{\bf v}] = {\bf v}^*H{\bf v} + 4\beta \cdot {\sum}_{k=1}^n r_k^2 t_k^2 \cos^2(\theta_k - \phi_k) \geq 0 \end{equation}
for all ${\bf v} \in \mathcal T_{\bf z}\mathcal M$, where $(r_{1}e^{i\theta_{1}},\dots,r_{n}e^{i\theta_{n}})^{T}$ and $(t_{1}e^{i\phi_{1}},\dots,t_{n}e^{i\phi_{n}})^{T}$ are the corresponding polar coordinates of ${\bf z}$ and ${\bf v}$, respectively.

As shown in Lemma \ref{tho:diag_1} and using the stationarity condition $\grad{f({\bf z})} = 0$, it follows $H_{kk} = a_k + 2\beta |z_k|^2 - 2\lambda = 0$ for all $k \in \mathcal I = \{k: z_k \neq 0\}$. Next, for $k \in \mathcal I^{\sf C}$, we define ${\bf v} := e_k$, where $e_k$ denotes the $k$-th unit vector. This choice of ${\bf v}$ obviously fulfills $\Re({\bf v}^* {\bf z}) = 0$ and thus, the optimality condition \eqref{eq:soc-in-theo} implies $H_{kk} = H_f({\bf z})[{\bf v}] \geq 0$. Since $H$ is diagonal, this yields $H \succeq 0$. 

In order to verify the explicit characterization of local minimizers, we notice that ${\bf u} = \mathcal P_{\Delta_n}(-{\bf a}/2\beta)$ is the unique solution of the strongly convex problem \eqref{eqn:obj_1}.
Moreover, using the identity $u_{k} \equiv \lvert z_{k}\rvert^{2}$, $k \in [n]$, every global solution of \eqref{eqn:obj} corresponds to a global minimizer of the problem \eqref{eqn:obj_1} and vice versa. Since problem \eqref{eqn:obj} does not possess spurious local minimizers, this finishes the proof of Theorem \ref{theorem:char-glob-loc}.
\end{proof}

The latter theorem shows that we can identify and explicitly compute the unique equivalence class $\llbracket {\bf z} \rrbracket$ of global minimizer by a projection onto the $n$-simplex. This can be realized numerically in $O(n\log{n})$ operations, see \cite{finlayson1987numerical}. 

Inspired by the analysis of phase synchronization problems in \cite{bandeira2017tightness}, we now study the behavior of global minimizer when the diagonal matrix $A$ is perturbed by a random noise matrix $W$. 

\begin{theorem} \label{theorem:perturb} Let $A$ be a given diagonal matrix and let $W \in \C^{n\times n}$ be a Hermitian noise matrix with noise level $\sigma > 0$. Suppose that ${\bf z}_0$ is a global minimizer of \eqref{eqn:obj} and that the point ${\bf y} \in \CSn$  
satisfies $f_{\sigma}(\mathbf{y})\leq \min_{{\bf z} \in \llbracket {\bf z}_0 \rrbracket} f_{\sigma}(\mathbf{z})$, where $f_\sigma({\bf z}) := f({\bf z}) + \frac{\sigma}{2} {\bf z}^*W{\bf z}$. Then, it holds that
\begin{align*}
\min_{{\bf z} \in \llbracket {\bf z}_0 \rrbracket}~\lVert\mathbf{y}-\mathbf{z}\rVert_{4} \leq \sqrt[3]{{2\sigma\beta^{-1}\lVert W\rVert_{2}n^{1/4}}}.
\end{align*}
\end{theorem}
\begin{proof}
As usual, we introduce the polar coordinates $\mathbf{z}_{0}=(r_{1}e^{i\theta_{1}},...,r_{n}e^{i\theta_{n}})^{T}$ and $\mathbf{y}=(t_{1}e^{i\phi_{1}},...,t_{n}e^{i\phi_{n}})^{T}$. Due to Theorem \ref{theorem:char-glob-loc}, we can assume $\theta_{k}=\phi_{k}$ and we have $a_k + 2\beta r_k^2 - 2\lambda \geq 0$ for all $k \in [n]$, where $a_k = A_{kk}$ and $\lambda$ is the associated multiplier of ${\bf z}_0$.
%
%
Thus, using $f_{\sigma}(\mathbf{y})\leq f_{\sigma}(\mathbf{z}_{0})$, this implies
\begin{align*}
\frac{\sigma}{2}(\mathbf{z}_{0}^{*}W\mathbf{z}_{0}-\mathbf{y}^{*}W\mathbf{y}) & \geq \frac{1}{2}(\mathbf{y}^{*}A\mathbf{y}-\mathbf{z}_{0}^{*}A\mathbf{z}_{0})+\frac{\beta}{2} \cdot {\sum}_{k\in[n]}(t_{k}^{4}-r_{k}^{4})\\
&=\frac{1}{2}\sum_{k\in[n]}(a_{k}+\beta(t_{k}^{2}+r_{k}^{2}))(t_{k}^{2}-r_{k}^{2})\\
& \geq \frac{1}{2}\sum_{k\in[n]}(2\lambda +\beta(t_{k}^{2}-r_{k}^{2}))(t_{k}^{2}-r_{k}^{2}) \\
&=\frac{\beta}{2}\sum_{k\in[n]}(t_{k}^{2}-r_{k}^{2})^{2}\geq\frac{\beta}{2}\sum_{k\in[n]}(t_{k}-r_{k})^{4}=\frac{\beta}{2}\lVert\mathbf{y}-\mathbf{z}_{0}\rVert_{4}^{4}.
\end{align*}
%
Note that the last inequality follows from $(t_{k}+r_{k})^{2}\geq(t_{k}-r_{k})^{2}$. Furthermore, by H\"older's inequality and by $\|{\bf x}\|_{4/3} \leq n^{1/4} \|{\bf x}\|$, we have
\begin{align*}
\mathbf{z}_{0}^{*}W\mathbf{z}_{0}-\mathbf{y}^{*}W\mathbf{y} & =\Re((\mathbf{z}_{0}-\mathbf{y})^{*}W(\mathbf{z}_{0}+\mathbf{y})) \leq \lVert\mathbf{z}_{0}-\mathbf{y}\rVert_{4}\lVert W(\mathbf{z}_{0}+\mathbf{y})\rVert_{4/3} \\ & \leq n^{1/4} \lVert\mathbf{z}_{0}-\mathbf{y}\rVert_{4}\lVert W(\mathbf{z}_{0}+\mathbf{y})\rVert \leq 2n^{1/4} \lVert W\rVert_{2}\lVert\mathbf{z}_{0}-\mathbf{y}\rVert_{4}.
\end{align*}
Combining the above two inequalities, concludes the proof.
\end{proof}
\begin{remark}
If $W \in \C^{n \times n}$ is a Hermitian random matrix with i.i.d. off-diagonal entries following a standard complex normal distribution and with zero diagonal entries, then Bandeira, Boumal, and Singer, \cite{bandeira2017tightness}, have shown that the bound $\lVert W\rVert_{2}\leq3\sqrt{n}$ holds with probability at least $1-2n^{-5/4}-e^{-n/2}$. Combing this observation with Theorem \ref{theorem:perturb}, we can obtain 
\begin{align*}
\min_{{\bf z} \in \llbracket {\bf z}_0 \rrbracket}~\lVert\mathbf{y}-\mathbf{z}_{0}\rVert_{4}\leq\sqrt[3]{6\sigma\beta^{-1}}\cdot n^{1/4}
\end{align*}
with probability at least $1-2n^{-5/4}-e^{-n/2}$.
\end{remark}


\section{Geometric Analysis of the Rank-One Case}
\label{sec:rank1}

In this section, we investigate the case when $A$ is rank-one and positive semidefinite, i.e., we can write $A=\mathbf{a}\mathbf{a}^{*}$ for some $\mathbf{a}\in\mathbb{C}^{n}$ and the quartic-quadratic problem \eqref{eqn:obj} reduces to
%
\begin{align}
\min_{\mathbf{z}\in\mathbb{C}^{n}}~f(\mathbf{z})=\frac{1}{2}\lvert\mathbf{a}^{*}\mathbf{z}\rvert^{2}+\frac{\beta}{2}\lVert\mathbf{z}\rVert_{4}^{4} \quad \mathrm{s.t.} \quad \lVert\mathbf{z}\rVert_{2}=1.
\label{eqn:obj-rank1}
\end{align}
The associated first- and second-order necessary optimality conditions are given by
\[ {\bf a}^*{\bf z} \cdot {\bf a} + 2\beta \diag(|{\bf z}|^2) {\bf z} = 2\lambda {\bf z}, \quad 2\lambda = \lvert\mathbf{a}^{*}\mathbf{z}\rvert^{2}+2\beta\lVert\mathbf{z}\rVert_{4}^{4} \]
and $H_f({\bf z})[{\bf v}] = {\bf v}^*[{\bf a}{\bf a}^* + 2\beta \diag(|{\bf z}|^2) - 2\lambda I]{\bf v} + 4\beta \sum_{k=1}^n \Re(v_k \bar z_k)^2 \geq 0$ 
%
for all $\mathbf{v} \in \Cn$ with $\Re(\mathbf{v}^{*}\mathbf{z})=0$. 
We now present a first structural and preparatory property of local and global minima.
\begin{lemma}
Suppose that $\mathbf{z}$ is a local minimizer of \eqref{eqn:obj-rank1}. Then, for all $k\in[n]$ with $a_{k}=0$ it holds that $\lvert z_{k}\rvert^{2}=\frac{\lambda}{\beta}$.
\label{lem:4-1}
\end{lemma}
\begin{proof}
If $a_{k}=0$, the first-order optimality conditions imply $\beta\lvert z_{k}\rvert^{2}z_{k}=\lambda z_{k}$.
Let us assume $z_{k}=0$ and let us choose ${\bf v} \in \Cn$ with $v_{k}=1$ and $v_{j}=0$ for all $j\neq k$. Due to $\lambda \geq \beta \|{\bf z}\|_4^4 \geq \frac{\beta}{n} > 0$, we obtain 
\begin{align*}
H_{f}(\mathbf{z})[\mathbf{v}]=-2\lambda<0,
\end{align*}
which contradicts the second-order necessary optimality conditions. Hence, we have $\lvert z_{k}\rvert^{2}=\frac{\lambda}{\beta}$.
\end{proof}

In the following sections, we discuss two different classes of local minima, which are characterized by the orthogonality to the vector $\mathbf{a}$.

\subsection{Orthogonal local minima}

We first analyze the case where the local minimizer $\mathbf{z}$ satisfies $\mathbf{a}^{*}\mathbf{z}=0$.
\begin{theorem}
Suppose that $\mathbf{z}$ is a local minimizer satisfying $\mathbf{a}^{*}\mathbf{z}=0$. Then, $\mathbf{z}$ has at most one zero component and all of its nonzero components must have the same modulus.
\label{thm:4-3}
\end{theorem}
\begin{proof}
By the first-order optimality conditions, it follows $z_{k}=0$ or $\lvert z_{k}\rvert^{2}=\frac{\lambda}{\beta}$ for all $k$. Hence, all nonzero components of $\mathbf{z}$ have the same modulus.

Without loss of generality we now assume that $\lvert z_{k}\rvert^{2}=\frac{\lambda}{\beta}$ for $1\leq k\leq \tau$ and $z_{k}=0$ for $\tau+1\leq k\leq n$. Due to Lemma \ref{lem:4-1}, we have $a_{n-1},a_{n}\neq0$ if $\tau \leq n-2$. Let us set
\begin{align*}
v_{k} = 0, \quad  k\in[n-2], \quad v_{n-1}=\frac{a_{n}}{\sqrt{\lvert a_{n-1}\rvert^{2}+\lvert a_{n}\rvert^{2}}},\quad v_{n}=\frac{-a_{n-1}}{\sqrt{\lvert a_{n-1}\rvert^{2}+\lvert a_{n}\rvert^{2}}}.
\end{align*}
Then, it holds that $H_{f}(\mathbf{z})[\mathbf{v}]=-2\lambda<0$, which is a contradiction. Thus, we have $\tau = n-1$ or $\tau = n$ (which means that all components $z_k$ are nonzero).
\end{proof}

Next, we derive conditions under which the existence of such local minima can be ensured. Before we present the formal statement and proof of the main theorem, we discuss a result that is used later in Theorem \ref{theorem:exi-rank-1}. 

\begin{lemma}
If $\|{\bf a}\|_\infty \leq\frac{1}{2}\|{\bf a}\|_1$, there exist phases $\{\theta_{k}\}_{k\in[n]}$, $\theta_k \in [0,2\pi]$, such that $\sum_{k\in[n]}e^{i\theta_{k}}a_{k}=0$.
\label{lem:4-3}
\end{lemma}
\begin{proof}
Let us assume $a_{k}\neq0$ for all $k$. If $n=1$, the condition $\|{\bf a}\|_\infty \leq\frac{1}{2}\|{\bf a}\|_1$ is never satisfied and hence, the statement in Lemma \ref{lem:4-3} holds automatically. In the case $n=2$, we have $2\max \{ \lvert a_{1}\rvert, |a_2|\} \leq \lvert a_{1}\rvert+\lvert a_{2}\rvert$. This implies $\lvert a_{1}\rvert=\lvert a_{2}\rvert$ and thus, we can choose $\theta_1$ and $\theta_2$ such that $e^{i\theta_{1}}a_{1}=\lvert a_{1}\rvert$ and $e^{i\theta_{2}}a_{2}=-\lvert a_{1}\rvert$.

Otherwise assume $n\geq3$ and $\lvert a_{1}\rvert=\|{\bf a}\|_\infty$. Let $b_{1}=\sum_{k\geq 3}\lvert a_{k}\rvert$, $b_{2}=\lvert a_{2}\rvert$ and $b_{3}=\lvert a_{1}\rvert$. It holds that $b_{1}+b_{2}=\|{\bf a}\|_1-\lvert a_{1}\rvert\geq b_{3}$, which means that the numbers $b_{1},b_{2},b_{3}$ can be interpreted as sides of a (degenerated) triangle. Let $ABC$ be such a triangle embedded into the complex space, where $A,B,C\in\mathbb{C}$ denote the nodes of $ABC$ with $\lvert B-C\rvert=b_{1},\lvert C-A\rvert=b_{2},\lvert A-B\rvert=b_{3}$. Consequently, there exist $\phi_{1},\phi_{2},\phi_{3}\in[0,2\pi)$ such that $B-C=e^{i\phi_{1}}b_{1},C-A=e^{i\phi_{2}}b_{2},A-B=e^{i\phi_{3}}b_{3}$. But then we have $\sum_{k=1}^{3}e^{i\phi_{k}}b_{k}=0$, 
which completes the proof of the lemma.
\end{proof}

\begin{theorem}[Existence of Orthogonal Local Minima] \label{theorem:exi-rank-1}
There exists a local minimizer $\mathbf{z}$ of \eqref{eqn:obj-rank1} such that $\mathbf{a}^{*}\mathbf{z}=0$ if and only if $\|{\bf a}\|_\infty \leq\frac{1}{2}\lVert\mathbf{a}\rVert_{1}$, or $\mathbf{a}$ has only one nonzero component and we have $\lVert\mathbf{a}\rVert^{2}\geq2\beta/(n-1)$. Further, if $\mathbf{a}$ satisfies such conditions, all local minima with $\mathbf{a}^{*}\mathbf{z}=0$ are the only global minima of \eqref{eqn:obj-rank1}.
\end{theorem}
\begin{proof}
Let $\mathbf{z}=(r_{1}e^{i\theta_{1}},\dots,r_{n}e^{i\theta_{n}})^{T}$ be a local minimizer of problem (\ref{eqn:obj-rank1}) such that $\mathbf{a}^{*}\mathbf{z}=0$. By Theorem \ref{thm:4-3}, we only need to consider the cases when the local minimizer has no zero component or exactly one zero component.

\textit{Case 1.} If $\mathbf{z}$ does not have any zero component, then it follows $\lvert z_{k}\rvert^{2}=\frac{1}{n}$ for all $k$ and we have
\begin{align*}
\mathbf{a}^{*}\mathbf{z}=\frac{1}{\sqrt{n}}\sum_{k\in[n]}e^{i\theta_{k}}a_{k}=0,
\end{align*}
which implies $\lvert a_{j}\rvert=\lvert\sum_{k\neq j}e^{i\theta_{k}}a_{k}\rvert\leq\sum_{k\neq j}\lvert a_{k}\rvert=\lVert\mathbf{a}\rVert_{1}-\lvert a_{j}\rvert$ for all $j \in[n]$. Choosing $a_{j}$ to be the element with maximal modulus, we get $ \|{\bf a}\|_\infty \leq\frac{1}{2}\lVert\mathbf{a}\rVert_{1}$.

\textit{Case 2.} Let us suppose $z_{n} = 0$. Then, due to Lemma \ref{lem:4-1}, we obtain $a_{n}\neq0$. Let us assume that there exists another component $a_{k}\neq0$ for some $k \in [n-1]$. Setting
\begin{align*}
v_{j}=iz_{j}, \quad j\neq k,n,\quad v_{k}=(1-\lvert a_{n}\rvert)iz_{k},\quad v_{n}= \frac{a_{n}}{\lvert a_{n}\rvert}\bar{a}_{k}\cdot iz_{k},
\end{align*}
and normalizing $\mathbf{v}$, we have $\mathbf{a}^{*}\mathbf{v} = \Re({\bf z}^*{\bf v})=0$ and the curvature is given by
\begin{align*}
H_{f}(\mathbf{z})[\mathbf{v}]=2\beta\sum_{k\in[n-1]} |z_k|^2 \lvert v_{k}\rvert^{2}-2\lambda=-2\lambda\lvert v_{n}\rvert^{2}<0,
\end{align*}
which contradicts with the second-order optimality conditions. Hence, we can infer $a_{k}=0$ for all $k\in[n-1]$, which implies that $\mathbf{a}$ has only one nonzero component. By Theorem \ref{thm:4-3}, we have $\lvert z_{k}\rvert^{2}=\frac{1}{n-1}$ for all $k\in[n-1]$. The second-order necessary optimality conditions yields
\begin{align*}
H_{f}(\mathbf{z})[\mathbf{v}]=\lvert a_{n}v_{n}\rvert^{2}+\frac{2\beta}{n-1}(1-\lvert v_{n}\rvert^{2})-\frac{2\beta}{n-1}=\left(\lvert a_{n}\rvert^{2}-\frac{2\beta}{n-1}\right)\lvert v_{n}\rvert^{2}\geq0,
\end{align*}
for $|v_n| \in [0,1]$ and it follows $\lVert\mathbf{a}\rVert^{2}=\lvert a_{n}\rvert^{2}\geq\frac{2\beta}{n-1}$.

We continue with the proof of the second direction. In particular, suppose that $\mathbf{a}$ satisfies the conditions stated in Theorem \ref{theorem:exi-rank-1}. We again discuss two cases.

\textit{Case 1.} By Lemma \ref{lem:4-3}, if $\|{\bf a}\|_\infty \leq\frac{1}{2}\lVert\mathbf{a}\rVert_{1}$, we can choose phases $\{\theta_{k}\}_{k\in[n]}$ such that $\sum_{k\in[n]}e^{i\theta_{k}}a_{k}=0$. Let us set $z_{k}= e^{i\theta_{k}} / {\sqrt{n}}$ for all $k\in[n]$. Then, we have $\mathbf{a}^{*}\mathbf{z}=0$ and $f(\mathbf{z})= {\beta}/{(2n)}$, which is the lower bound of the objective function $f$. Thus, in this case $\mathbf{z}$ is a global minimizer of (\ref{eqn:obj-rank1}). Moreover, the objective function attains its optimal value if and only if $\mathbf{a}^{*}\mathbf{z}=0$ and $\|{\bf z}\|_{4}^4={1}/{n}$.

\textit{Case 2.} If $\mathbf{a}$ has only one nonzero component $a_{n}$ with $\lVert\mathbf{a}\rVert^{2}\geq\frac{2\beta}{n-1}$, the matrix $A$ is diagonal and we can apply the results derived in section 3. Specifically, by Theorem \ref{theorem:char-glob-loc}, it can be shown that all local minimizer are global minimizer and satisfy
\[ \lvert z_{k}\rvert=\frac{1}{\sqrt{n-1}}, \quad \forall~k\in[n-1], \quad \text{and} \quad z_{n}=0. \]
This finishes the proof of Theorem \ref{theorem:exi-rank-1}.
\end{proof}

\subsection{Non-orthogonal local minima}

We discuss the case when there is no local minimizer $\mathbf{z}$ such that $\mathbf{a}^{*}\mathbf{z}=0$, or, equivalently, $\mathbf{a}$ does not satisfy the conditions in Theorem \ref{thm:4-3}. By the first-order optimality conditions, all $z_{k}$ with $a_{k}\neq0$ satisfy
\begin{align*}
\mathrm{arg}(z_{k})-\mathrm{arg}(a_{k})=\mathrm{arg}(\mathbf{a}^{*}\mathbf{z})=\mathrm{const},
\end{align*}
where $\mathrm{arg}(z)$ is the principal angle of the complex number $z$ modulo $\pi$. Since a global shift of the phase will not change the objective function value and the first-order optimality conditions, we can shift $\mathbf{z}$ by a global phase such that the principal angles of the nonzero components are the same as $a_{k}$. In the case $a_{k}=0$, the phase of $z_{k}$ does not influence the objective function value and the first-order optimality conditions and we can adjust ${\bf z}$ to be a real number. Consequently, for every stationary point of problem \eqref{eqn:obj-rank1}, we can find a corresponding stationary point which has the same objective function value and satisfies the following  `consistency' property.
\begin{definition}
A stationary point $\mathbf{z}$ of problem \eqref{eqn:obj-rank1} is called consistent, if it satisfies
\begin{align*}
\bar a_{k} z_{k}\in\mathbb{R}, \quad \forall~k \; \text{with} \; a_{k}\neq0 \quad \text{and} \quad z_{k} \in\mathbb{R}, \quad \forall~k \; \text{with} \; a_{k}=0.
\end{align*}
\end{definition}

Note that the corresponding consistent stationary point of a local minimizer of problem \eqref{eqn:obj-rank1} does not need to be a local minimizer. On the other hand, shifted consistent stationary points of global minimizer remain global minimizer. In this subsection, we focus on structural properties of consistent stationary points of \eqref{eqn:obj-rank1}.

\begin{remark}
Suppose $\mathbf{a}$ does not satisfy the conditions in Theorem \ref{thm:4-3} and there exists a local minimizer $\mathbf{z}$ such that $\mathbf{a}^{*}\mathbf{z}\neq0$. If we have $z_{k}=0$ for some $k\in[n]$, then the first-order optimality conditions imply  $a_{k}=0$ which contradicts Lemma \ref{lem:4-1}. Hence, for any local minima $\mathbf{z}$ with $\mathbf{a}^{*}\mathbf{z}\neq0$, we have $z_{k}\neq0$ for all $k\in[n]$.
\end{remark}

In the following result, we show that consistent local minima must belong to the same equivalence class defined in \eqref{eq:equi-class}.
\begin{theorem}
Suppose that $\mathbf{z},\mathbf{y}\in\mathbb{CS}^{n-1}$ are two consistent local minima of problem \eqref{eqn:obj-rank1} with $\mathbf{a}^{*}\mathbf{z}\neq0$ and $\mathbf{a}^{*}\mathbf{y}\neq0$. Then, we have ${\bf y} \in \llbracket {\bf z} \rrbracket$.
\end{theorem}
\begin{proof} The consistency of the stationary points ${\bf y}$ and ${\bf z}$ implies ${\bf y}^*{\bf z} \in \R$ and thus, it holds that $i\mathbf{y} \in\mathcal{T}_{\mathbf{z}}\mathcal{M}\cap\mathbb{CS}^{n-1}$ and $i\mathbf{z}\in\mathcal{T}_{\mathbf{y}}\mathcal{M}\cap\mathbb{CS}^{n-1}$. By the second-order necessary optimality conditions, we have
\begin{align}
\nonumber H_{f}(\mathbf{z})[i\mathbf{y}]&=\lvert\mathbf{a}^{*}\mathbf{y}\rvert^{2}+2\beta\sum_{k\in[n]}\lvert z_{k} y_{k}\rvert^{2}-2\lambda_\mathbf{z} \geq 0,\\
H_{f}(\mathbf{y})[i\mathbf{z}]&=\lvert\mathbf{a}^{*}\mathbf{z}\rvert^{2}+2\beta\sum_{k\in[n]}\lvert z_{k} y_{k}\rvert^{2}-2\lambda_\mathbf{y} \geq0,
\label{eqn:4-10}
\end{align}
where $2\lambda_{\bf z} = |{\bf a}^*{\bf z}|^2 + 2\beta \|{\bf z}\|_4^4$ and $2\lambda_{\bf y} = |{\bf a}^*{\bf y}|^2 + 2\beta \|{\bf y}\|_4^4$. Summing those two inequalities yields
\begin{align*}
 0 & \leq - 2\beta \left[ \|{\bf z}\|_4^4 - 2 \sum_{k\in[n]}\lvert z_{k} y_{k}\rvert^{2} + \|{\bf y}\|_4^4 \right] = -2\beta \sum_{k \in [n]} [|z_k|^2 - |y_k|^2]^2.
 \end{align*}
%
Hence, we have $\lvert z_{k}\rvert =\lvert y_{k}\rvert$ for all $k\in[n]$.
\end{proof}
\begin{remark}
Similarly, if $\mathbf{z},\mathbf{y}\in\mathbb{CS}^{n-1}$ are two local minima of problem \eqref{eqn:obj-rank1} such that $z_{k}$ and $y_{k}$ have the same phases for all $k\in[n]$, then ${\bf y} \in \llbracket {\bf z} \rrbracket$.
\end{remark}

We now prove that there are no spurious consistent local minima.
\begin{theorem}
If $\mathbf{a}$ does not satisfy the conditions in Theorem \ref{thm:4-3}, then all the consistent local minima of problem \eqref{eqn:obj-rank1} are global minima.
\end{theorem}
\begin{proof}
Suppose $\mathbf{z},\mathbf{y}\in\mathbb{CS}^{n-1}$ are two consistent local minima. Using the inequalities in \eqref{eqn:4-10} and $\lvert z_{k}\rvert = \lvert y_{k}\rvert$, we obtain $\lvert\mathbf{a}^{*}\mathbf{z}\rvert^{2}=\lvert\mathbf{a}^{*}\mathbf{y}\rvert^{2}$. Hence, all consistent local minima have the same objective function value. Since $\mathbf{a}$ does not satisfy the conditions in Theorem $\ref{thm:4-3}$, there exists a consistent global minimizer satisfying $\mathbf{a}^{*}\mathbf{z}\neq0$. This shows that all consistent local minima are global minima. 
\end{proof}
\begin{remark}
Combining the results of the last two subsections, we see that global minima of problem \eqref{eqn:obj-rank1} are unique up to certain shifts in the phase. In particular, we can shift the phases of components with $(\mathbf{a}^{*}\mathbf{z})a_{k}=0$ arbitrarily and shift all the other components by the same angle.
\end{remark}

%

\section{Analyzing the Geometric Landscape -- the Real Case}
\label{sec:strict_saddle}

We now investigate a variant of the so-called \textit{strict-saddle property} introduced by Ge, Jin, and Zheng in \cite[Definition 2]{ge2017no}. More specifically, as in \cite[Theorem 2.2]{sun2016geometric}, we strengthen the first condition in \cite[Definition 2]{ge2017no} to uniform positive definiteness of the Riemannian Hessian.
\begin{definition} \label{def:ssp}
Let $\xi,\epsilon,\zeta>0$ be given constants. A function $f$ is called $(\xi,\epsilon,\zeta)$-\textbf{strict-saddle} if for all $\mathbf{z}\in\mathcal{M}$ one of the following conditions holds:
\begin{itemize}
  \item[1.](Strong convexity). For all $\mathbf{v}\in\mathcal{T}_{\mathbf{z}}\mathcal{M}\cap\mathbb{S}^{n-1}$ we have $H_{f}(\mathbf{z})[\mathbf{v}]\geq\xi$.
  \item[2.](Large gradient). It holds that $\lVert\grad{f(\mathbf{z})}\rVert\geq\epsilon$.
  \item[3.](Negative curvature). There exists $\mathbf{v}\in\mathcal{T}_{\mathbf{z}}\mathcal{M}\cap\mathbb{S}^{n-1}$ with $H_{f}(\mathbf{z})[\mathbf{v}]\leq-\zeta$.
\end{itemize}
\end{definition}

The strict-saddle property can be utilized to establish polynomial time convergence rates of second-order optimization algorithms, such as the  Riemannian trust region method \cite{sun2016geometric}, and almost sure convergence to local minimizer of Riemannian gradient descent methods with random initialization, see, e.g., \cite{lee2016gradient,lee2017first-order}.

In the following sections, we analyze the geometric landscape of problem \eqref{eqn:obj} and show that the strict-saddle property is satisfied in the real case when the interaction coefficient $\beta$ is sufficiently small or large. In general, due to the intricate relation between the quadratic and quartic terms, we can not expect that the conditions in Definition \ref{def:ssp} do hold for all choices of $\beta$ and $A$. 

For instance, let us consider the example $A := \alpha {\mathds 1}{\mathds 1}^T$, $\alpha \geq 0$, and ${\bf z} := {\mathds 1}/\sqrt{n} \in \Sn$. Then, the associated multiplier $\lambda$ is given by $2 \lambda = \alpha n + 2\beta /n$ and it can be shown that ${\bf z}$ is a stationary point of \eqref{eqn:obj} for all $\alpha$. Furthermore, for all ${\bf v} \in \mathcal{T}_{\mathbf{z}}\mathcal{M}$, it follows
\[ H_f({\bf z})[{\bf v}] = {\bf v}^T[A + 6\beta \diag(|{\bf z}|^2) - 2\lambda I]{\bf v} = [ 4\beta n^{-1} - \alpha n ] \|{\bf v}\|^2. \]
Hence, in the case $\alpha = 4\beta / n^2$, the strict-saddle property can not hold. Let us further note that the strong convexity condition in Definition \ref{def:ssp} is never satisfied at stationary points in the complex case. In particular, if ${\bf z} \in \CSn$ is a critical point, then we have $i{\bf z} \in \mathcal{T}_{\mathbf{z}}\mathcal{M}\cap\CSn$ and $H_f({\bf z})[i {\bf z}] = 0$ which contradicts condition 1. 

In Figure \ref{fig:land}, we illustrate different landscapes of the mapping $f$ when the problem is real and three-dimensional and the parameter $\beta$ changes. We consider the setup 
\begin{equation} \label{eq:example-A} A = \begin{pmatrix} 1 & 0 & 1 \\ 0 & 1 & 0 \\ 1 & 0 & 1 \end{pmatrix}, \quad \beta \in \{0.25,0.75,1.25,3.25\}, \end{equation}
and the eigenvalues of $A$ are given by $2$, $1$, and $0$, respectively. 
Moreover, we use spherical coordinates $(\phi,\theta) \mapsto (\cos(\phi),\sin(\phi)\cos(\theta),\sin(\phi)\sin(\theta))^{T}$, $\phi \in [0,2\pi]$, $\theta \in [0,\pi]$, to plot the objective function $f$ on the sphere. 

\begin{figure}[t]
\centering
\setlength{\belowcaptionskip}{-6pt}
\begin{tabular}{cc}
\subfloat[$\beta = 0.25$]{\includegraphics[trim={0 1.5cm 0 2cm},clip,width=5.5cm]{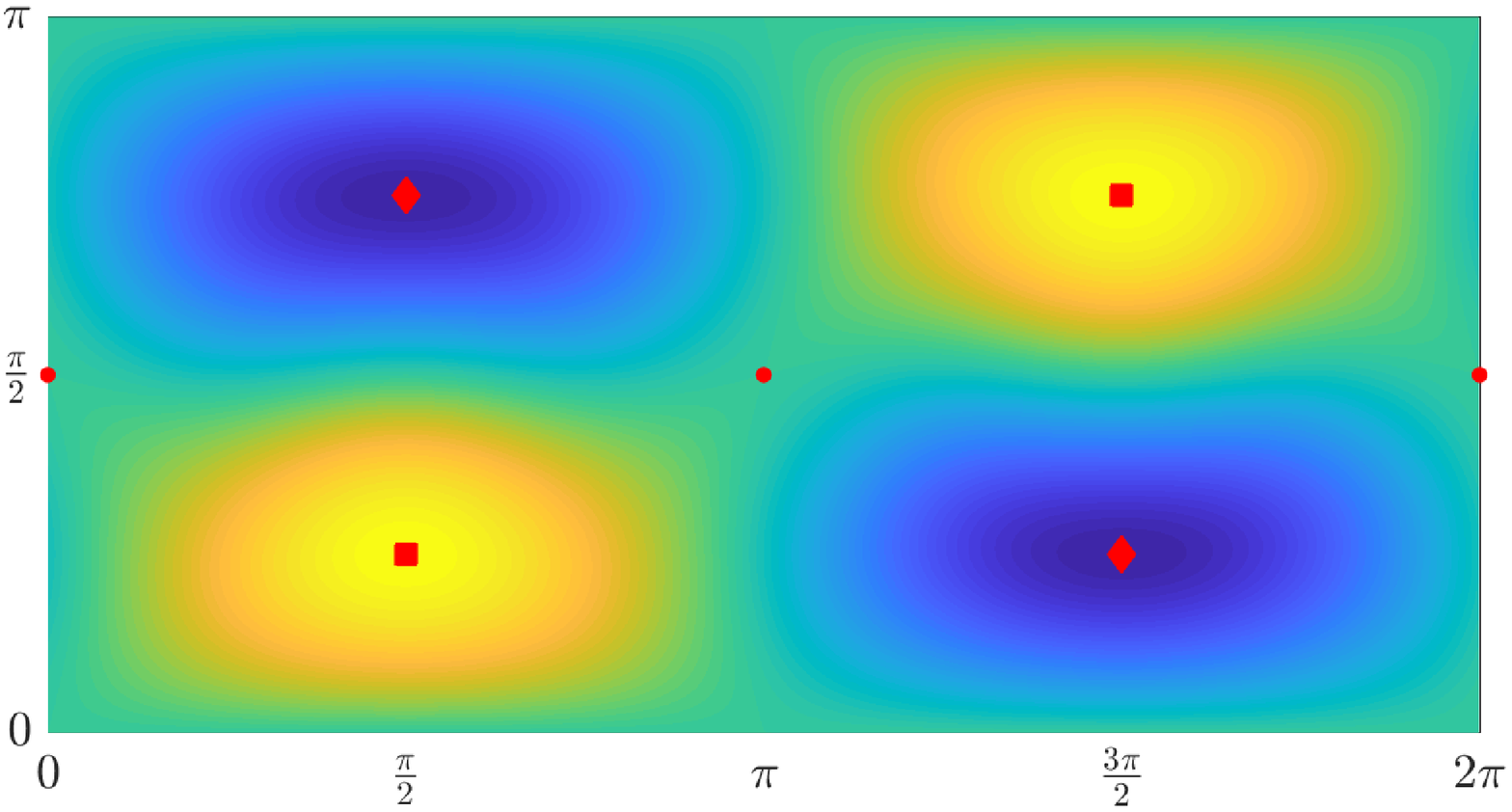}} &
\subfloat[$\beta = 0.75$]{\includegraphics[trim={0 1.5cm 0 2cm},clip,width=5.5cm]{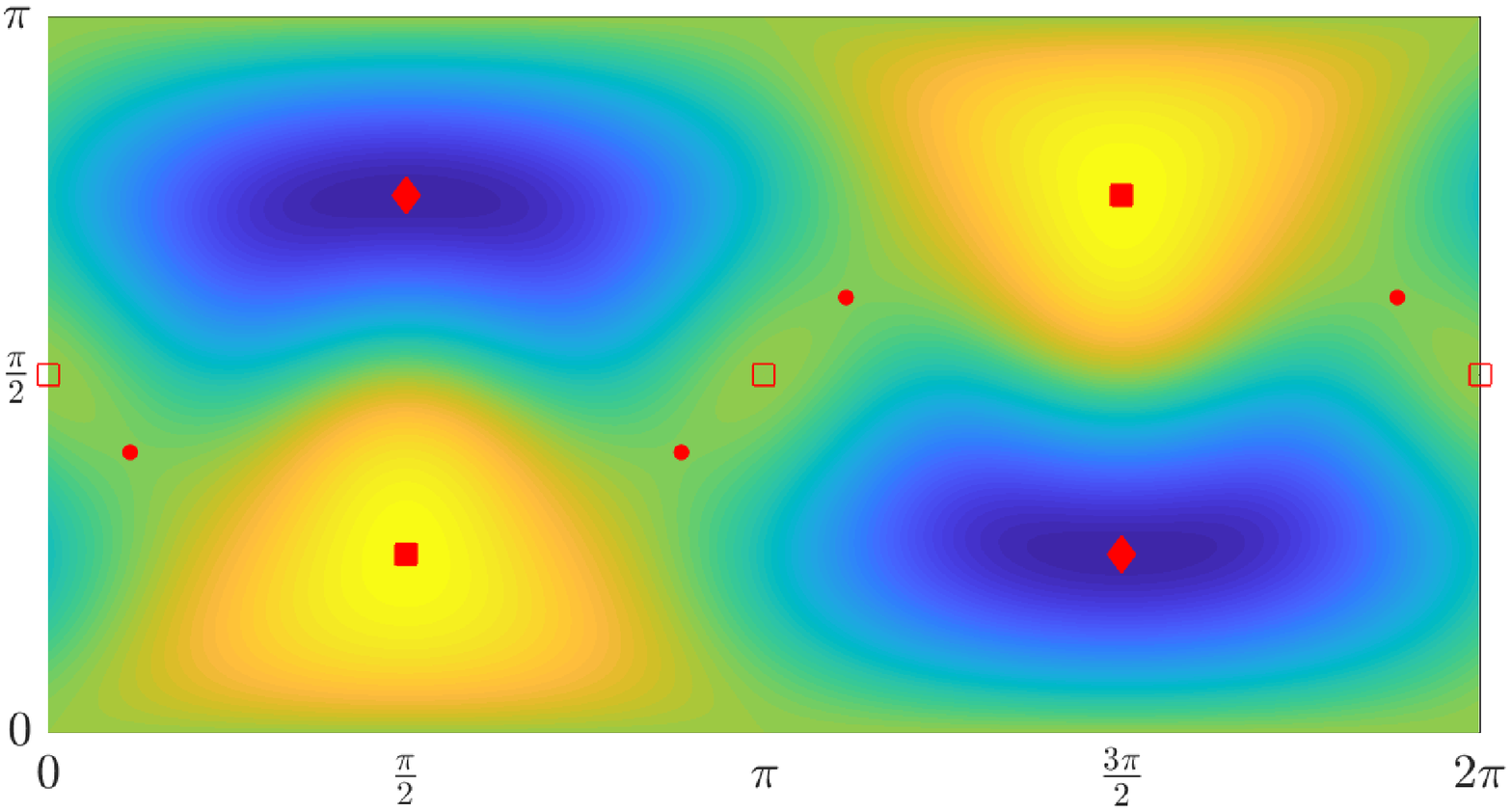}} \\
\subfloat[$\beta = 1.25$]{\includegraphics[trim={0 1.5cm 0 2cm},clip,width=5.5cm]{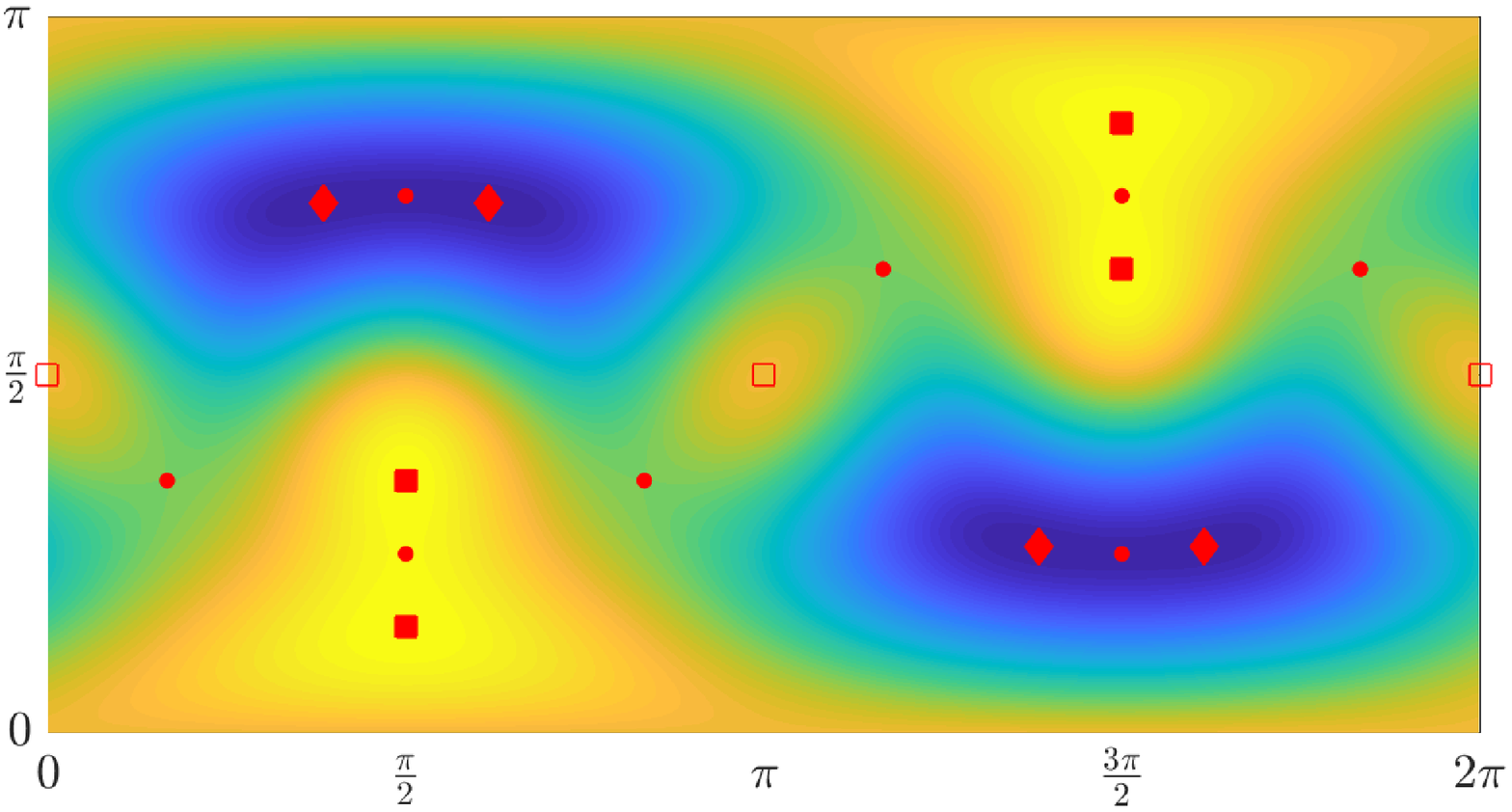}} &
\subfloat[$\beta = 3.25$]{\includegraphics[trim={0 1.5cm 0 2cm},clip,width=5.5cm]{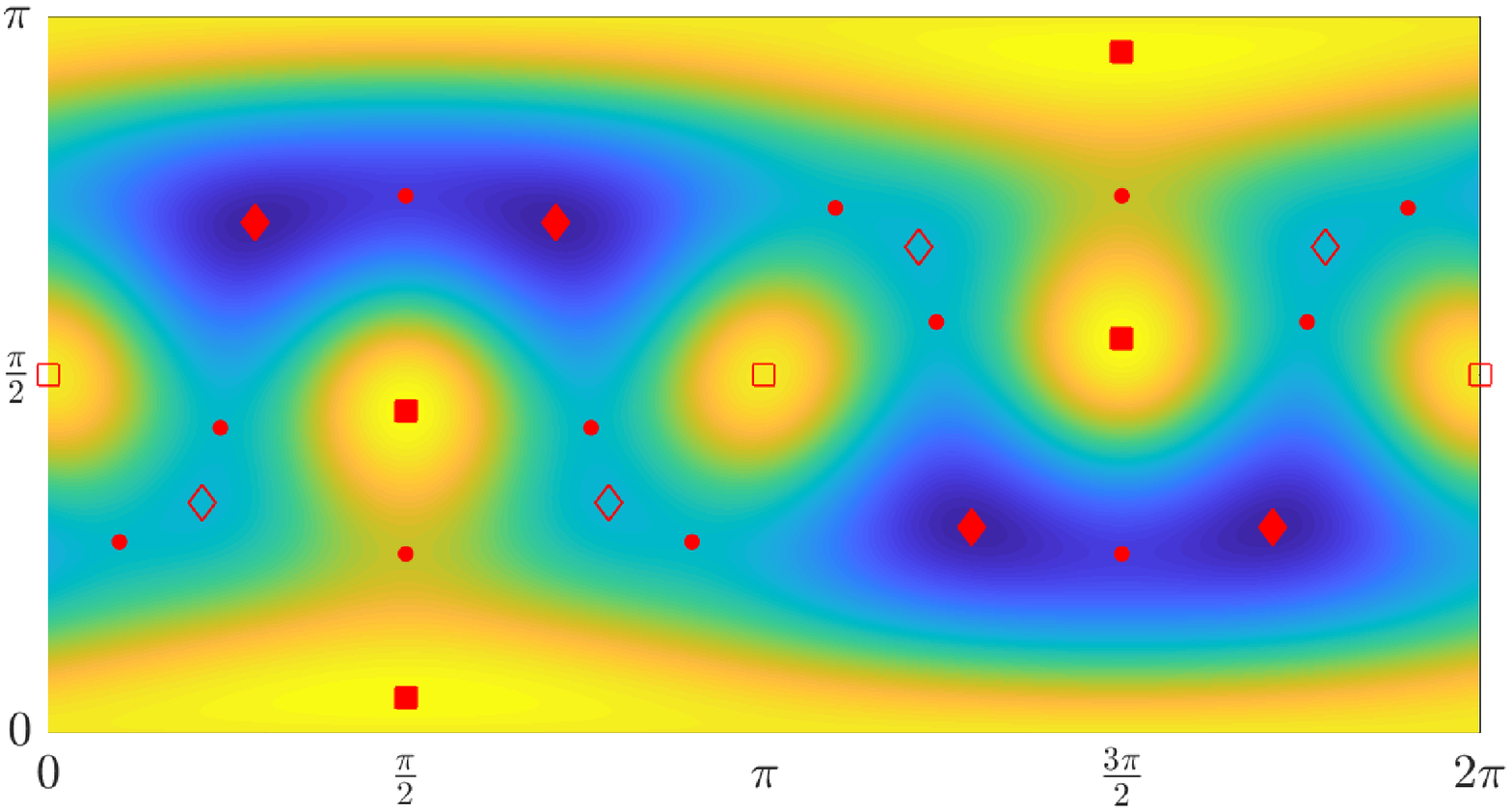}} 
\end{tabular}
\caption{Plot of the landscape of the objective function for fixed $A$ and different values of $\beta$. The red point marker depicts the location of saddle points. Local and global minima are indicated by non-filled and filled diamond markers. The location of local and global maxima is marked by non-filled and filled squares.}
\label{fig:land}
\end{figure}

Figure \ref{fig:land} demonstrates that the landscape of the objective function varies a lot when the interaction coefficient $\beta$ changes. Specifically, it shows that the number of stationary points and local minima increases from $6$ to $26$ and from $2$ to $8$ as $\beta$ increases from $0.25$ to $3.25$. 
In section \ref{sec:large} and section \ref{sec:small}, we investigate basic geometric features and the strict-saddle property for large and small choices of $\beta$ while keeping the matrix $A$ fixed. In particular, our results will allow us to characterize and describe the geometric landscape of $f$ in the sub-figures (a) and (d) of Figure \ref{fig:land}.  
In the following, we assume $n\geq2$ since the landscape of $f$ is trivial in the case $n=1$.



\subsection{Large interaction coefficient}
\label{sec:large}

In this section, we prove that the objective function possesses the strict-saddle property if the coefficient $\beta$ is chosen sufficiently large and satisfies $\beta \approx C\rho n^{3/2}$ for some constant $C > 0$. Here, $\rho := \lambda_1 - \lambda_n$ denotes the difference between the largest and the smallest eigenvalue of $A$. We analyze the geometric properties and behavior of $f$ on the following sets:
\begin{itemize}
  \item[1.](Strong convexity) $\mathcal{R}_{1} :=\{\mathbf{z}\in\Sn: \| |{\bf z}|^2-{\mathds 1}/n \|_\infty \leq 1/2n\}$,
  \item[2.](Large gradient) $\mathcal{R}_{2} := \{\mathbf{z}\in\Sn: \| |{\bf z}|^2-{\mathds 1}/n \|_\infty \geq 1/2n,\\\min_{k\in[n]} z_{k}^{2}\geq(n-1)/4n^{2}\}$,
  \item[3.](Negative curvature) $\mathcal{R}_{3} :=\{\mathbf{z}\in\Sn:\min_{k\in[n]}z_{k}^{2}\leq(n-1)/4n^{2}\}$.
\end{itemize}
Obviously, these three regions cover the sphere $\mathbb{S}^{n-1}$. An illustration of the regions $\mathcal R_1$--$\mathcal R_3$ is given later in Figure \ref{figure:areas}. 

We first present a preparatory result that is used later to estimate of the spectrum of the Riemannian Hessian. 
\begin{lemma} The following estimate holds for all $\mathbf{z}\in \Sn$:
\begin{equation} \label{eq:lem-05}
 z_{0}^2 \leq \min_{{\bf v} \in \mathcal T_{\bf z}\mathcal M \cap \Sn} \sum_{k \in [n]} v_k^2z_k^2 \leq  \frac{n}{n-1}(1-z_{0}^2)z_{0}^2,
\end{equation}
where $ z_{0}^2 := \min_{k \in [n]} z_k^2$. 
\label{lem:large-0}
\end{lemma}
\begin{proof}
If there exists $k \in [n]$ with $z_k = 0$, then the optimal value of the latter optimization problem is $0$ and it is attained for ${\bf v} = e_k$. Next, let us assume $z_0^2 > 0$ and let $i \in [n]$ be given with $z_i^2 = z_0^2$. We define $\mathbf{v}\in\mathcal{T}_{\mathbf{z}}\mathcal{M}\cap\mathbb{S}^{n-1}$ via
\begin{align*}
v_{i}=-\frac{(n-1)c}{z_{1}},\quad v_{k}=\frac{c}{z_{k}}, \quad c=\left[\frac{(n-1)^{2}}{z_{i}^{2}}+{\sum}_{k\neq i}\frac{1}{z_{k}^{2}}\right]^{-1/2}, \quad k \neq i.
\end{align*}
Using the reverse H\"older inequality, we have $\sum_{k\neq i} z_k^{-2} \geq (n-1)^2 (\sum_{k \neq i} z_k^2)^{-1}$ and thus, we obtain
\begin{align*}
\sum_{k\in[n]}z_{k}^{2}v_{k}^{2}=n(n-1)c^{2} \leq \frac{n(n-1)}{\frac{(n-1)^{2}}{z_{0}^{2}}+\frac{(n-1)^2}{{1-z_{0}^{2}}}} \leq \frac{n}{n-1}(1-z_{0}^{2})z_{0}^{2}.
\end{align*}
This establishes the upper bound in \eqref{eq:lem-05}. The lower bound follows from $\|{\bf v}\| = 1$ and $z_k^2 \geq z_0^2$ for all $k \in [n]$.

\end{proof}


In the following lemma, we verify that the objective function has the strong convexity property on the region $\mathcal{R}_{1}$.
\begin{lemma}
Let $\gamma > 0$ be given and suppose that $\beta \geq 2(1+\gamma)\rho n$. Then, for all ${\bf z} \in \mathcal R_1$ and all $\mathbf{v}\in\mathcal{T}_{\mathbf{z}}\mathcal{M}\cap\mathbb{S}^{n-1}$ it holds that $H_{f}(\mathbf{z})[\mathbf{v}] \geq \gamma \rho$.
\label{lem:large-1}
\end{lemma}
\begin{proof}
Let us define $w_{k} := z_{k}^{2}-\frac{1}{n}$. Hence, due to $\sum_{k\in[n]} w_{k}=0$, we obtain:
\begin{align*}
\sum_{k\in[n]}z_{k}^{4}=\sum_{k\in[n]}\left[w_{k}+\frac{1}{n}\right]^{2}=\frac{1}{n}+\sum_{k\in[n]} w_{k}^{2} \leq \frac{5}{4n}.
\end{align*}
Next, by Lemma \ref{lem:large-0} and ${\bf z} \in \mathcal R_1$, we have
\begin{align*}
\min_{\mathbf{v}\in\mathcal{T}_{\mathbf{z}}\mathcal{M}\cap\mathbb{S}^{n-1}}\sum_{k\in[n]}z_{k}^{2}v_{k}^{2} \geq \min_{k \in [n]} z_k^2 \geq \frac{1}{2n}.
\end{align*}
Combining the last results and using $\mathbf{v}^{T}A\mathbf{v}-\mathbf{z}^{T}A\mathbf{z}\geq-\rho$, it follows
\begin{align*}
\min_{\mathbf{v}\in\mathcal{T}_{\mathbf{z}}\mathcal{M}\cap\mathbb{S}^{n-1}}H_{f}(\mathbf{z})[\mathbf{v}]&\geq-\rho+\frac{2\beta}{n} \left [ \frac32 - \frac54 \right] \geq \gamma \rho  
\end{align*}
for all $\mathbf{z}\in\mathcal{R}_{1}$.
\end{proof}

The next lemma shows that the norm of the Riemannian gradient is strictly larger than zero -- uniformly -- on the region $\mathcal{R}_{2}$.
\begin{lemma}
Let $\beta \geq 8(1+\frac{1}{n-1}) (1+\gamma) \rho n^{3/2}$ be given for some $\gamma > 0$. Then, for all $\mathbf{z}\in\mathcal{R}_{2}$, it holds that $\lVert\grad{f(\mathbf{z})}\rVert \geq\frac{\gamma}{\sqrt{2}}\rho$.
\label{lem:large-2}
\end{lemma}
\begin{proof}
First, we split the norm of the Riemannian gradient $\grad{f(\mathbf{z})}$ as follows
\begin{align*}
\lVert\grad{f(\mathbf{z})}\rVert & \geq2\beta\lVert\diag(|{\bf z}|^2){\bf z}- \|{\bf z}\|^4_4 \cdot \mathbf{z} \rVert - \lVert A{\bf z}-(\mathbf{z}^{T}A\mathbf{z})\cdot \mathbf{z}\rVert  \\  & = 2\beta \sqrt{\|{\bf z}\|_6^6 - \|{\bf z}\|_4^8} - \lVert A{\bf z}-(\mathbf{z}^TA\mathbf{z})\cdot \mathbf{z}\rVert .
\end{align*}
We now estimate the minuend and subtrahend in the latter expression separately. Let us set $\mathbf{z}=\sum_{k} \alpha_{k}\mathbf{p}_{k}$, where $\{\mathbf{p}_{k}\}_k$ is the set of orthogonal eigenvectors of the matrix $A$ with corresponding eigenvalues $\lambda_{1},...,\lambda_{n}$. Then, it holds that $\sum_{k}\alpha_{k}^{2}=1$ and
\begin{align*}
\lVert A{\bf z}-(\mathbf{z}^{T}A\mathbf{z})\cdot {\bf z}\rVert^{2}&=\mathbf{z}^{T}A^{2}\mathbf{z}-(\mathbf{z}^{T}A\mathbf{z})^{2}=\sum_{k\in[n]}\alpha_{k}^{2}\lambda_{k}^{2}-\left[{\sum}_{k\in[n]}\alpha_{k}^{2}\lambda_{k}\right]^{2} \\
&=\sum_{k\in[n]}\alpha_{k}^{2}\lambda_{k}^{2} -\sum_{k,\ell \in[n]}\alpha_{k}^{2} \alpha_\ell^2 \lambda_{k} \lambda_\ell \\
& =\frac{1}{2}\sum_{k, \ell \in[n]}\alpha_{k}^{2} \alpha_{\ell}^{2}(\lambda_{k}-\lambda_{\ell})^{2} \leq \frac{1}{2} \rho^2 \sum_{k, \ell \in [n]} \alpha_{k}^{2} \alpha_{\ell}^{2} = \frac{1}{2}\rho^{2}.
\end{align*}
We continue with bounding the term $\lVert\mathbf{z}\rVert_{6}^{6}-\lVert\mathbf{z}\rVert_{4}^{8}$. As before using the spherical constraint $\sum_k z_k^2 = 1$, we obtain:
\begin{align*}
\lVert\mathbf{z}\rVert_{6}^{6}-\lVert\mathbf{z}\rVert_{4}^{8} =  \sum_{k\in[n]} z_{k}^{6}-\left[ {\sum}_{k\in[n]}z_{k}^{4}\right]^{2} = \frac12 \sum_{k,\ell\in[n]} z_{k}^{2}z_{\ell}^{2}(z_{k}^{2}-z_{\ell}^{2})^{2}.
\end{align*}
In the following and without loss of generality, we assume that the components of ${\bf z}$ are ordered and satisfy $z_1^2 \leq z_2^2 \leq ... \leq z_n^2$. Notice that $\mathbf{z}\in\mathcal{R}_{2}$ implies $z_1^2 \neq z_n^2$ and let us choose $t \in [n-1]$ such that $z_t^2 \leq \frac12(z_1^2 + z_n^2) \leq z_{t+1}^2$.
%
Then, we have
\begin{align*}
1=\sum_{k=1}^{t}z_{k}^{2}+\sum_{k=t+1}^{n}z_{k}^{2}\geq tz_{1}^{2}+\frac{n-t-1}{2}(z_{1}^{2}+z_{n}^{2}) + z_n^2,
\end{align*}
which yields
\begin{align*}
t\geq\frac{(n-1)(z_{n}^{2}+z_{1}^{2}) + 2z_n^2 -2}{z_{n}^{2}-z_{1}^{2}} \geq \frac{(n-1)z_1^2 + z_n^2 - 1}{z_n^2 - z_1^2}.
\end{align*}
Consequently, we get
\begin{align*}
\sum_{k,\ell\in[n]}z_{k}^{2}z_{\ell}^{2}(z_{k}^{2}-z_{\ell}^{2})^{2} & \geq z_1^2 \sum_{k \in [n]} z_k^2 (z_k^2 - z_1^k)^2 + z_n^2 \sum_{k \in [n]} z_k^2 (z_k^2 - z_n^2)^2 \\  
&\geq \frac{1}{4}(z_{n}^{2}-z_{1}^{2})^{2} \left[ z_{1}^{2} \sum_{k=t+1}^{n} z_{k}^{2}+z_{n}^{2}\sum_{k=1}^{t}z_{k}^{2} \right] \\
&=\frac{1}{4}(z_{n}^{2}-z_{1}^{2})^{2}\left [z_{1}^{2}+(z_{n}^{2}-z_{1}^{2})\sum_{k=1}^{t}z_{k}^{2}\right]\\
&\geq\frac{1}{4}(z_{n}^{2}-z_{1}^{2})^{2}[z_{1}^{2}+(z_{n}^{2}-z_{1}^{2})t z_{1}^{2}] \geq \frac{n}{4} (z_n^2-z_1^2)^2 z_1^4.
\end{align*}
Using $z_{1}^{2}\geq\frac{n-1}{4n^{2}}$ and $\max_{k}\lvert z_{k}^{2}-\frac{1}{n}\rvert = \max\{|z_1^2 - \frac1n|,|z_n^2 - \frac1n|\}\geq\frac{1}{2n}$, it follows
\begin{align*}
2[\|{\bf z}\|_6^6 - \|{\bf z}\|_4^8] \geq \frac{n}{4}(z_{n}^{2}-z_{1}^{2})^{2}z_{1}^{4}\geq\frac{n}{4}\left(\frac{1}{2n}\right)^{2}\left(\frac{n-1}{4n^{2}}\right)^{2}=\left(\frac{n-1}{16n^{5/2}}\right)^{2}
\end{align*}
and finally, combining the last results, we obtain
\[ \lVert\grad{f(\mathbf{z})}\rVert 
\geq\frac{1}{\sqrt{2}}\left[ 2\beta\cdot\frac{n-1}{16n^{5/2}}-\rho\right ] \geq \frac{\gamma}{\sqrt{2}} \rho, \]
as desired.
\end{proof}

Finally, we show that we can always find a negative curvature direction if ${\bf z}$ belongs to the set $\mathcal{R}_{3}$.
\begin{lemma}
Let $\gamma > 0$ be arbitrary and let $\beta \geq 2(1+\gamma)\rho n$ be given. Then, for all $\mathbf{z}\in\mathcal{R}_{3}$, there exists $\mathbf{v}\in\mathcal{T}_{\mathbf{z}}\mathcal{M}\cap\mathbb{S}^{n-1}$ such that $H_{f}(\mathbf{z})[\mathbf{v}]\leq-\gamma\rho$.
\label{lem:large-3}
\end{lemma}
\begin{proof}
Using the bound $\mathbf{v}^{T}A\mathbf{v}-\mathbf{z}^{T}A\mathbf{z}\leq\rho$ for $\mathbf{z},\mathbf{v}\in\mathbb{S}^{n-1}$ and Lemma \ref{lem:large-0}, it follows
\begin{align*}
\min_{\mathbf{v}\in\mathcal{T}_{\mathbf{z}}\mathcal{M}\cap\mathbb{S}^{n-1}}H_{f}(\mathbf{z})[\mathbf{v}]&\leq\rho+2\beta\left[\frac{3n}{n-1}(1-z_0^2) z_{0}^{2}- \|{\bf z}\|_4^4\right],
\end{align*}
where $z_0^2 := \min_{k \in [n]} z_k^2$. Moreover, due to $\lVert\mathbf{z}\rVert_{4}^{4}\geq\frac{1}{n}\lVert\mathbf{z}\rVert^{4}=\frac{1}{n}$ and $\mathbf{z}\in\mathcal{R}_{3}$, we can infer
\[ \min_{\mathbf{v}\in\mathcal{T}_{\mathbf{z}}\mathcal{M}\cap\mathbb{S}^{n-1}}H_{f}(\mathbf{z})[\mathbf{v}] \leq \rho + 2\beta \left[ \frac{3n}{n-1}z_0^2 - \frac{1}{n} \right] \leq \rho - \frac{\beta}{2n} \leq - \gamma\rho, \]
which finishes the proof of Lemma \ref{lem:large-3}.
\end{proof}

Combining the last lemmata, we can derive the following strict-saddle property.
\begin{theorem}
Suppose that the coefficient $\beta$ satisfies $\beta \geq \frac{8n}{n-1}(1+\gamma)\rho n^{3/2}$ for some given $\gamma > 0$. Then, the function $f$ has the $(C_\gamma \rho,\frac{\gamma}{\sqrt{2}}\rho,C_\gamma \rho)$-strict-saddle property with $C_\gamma := \frac{4}{n-1}(1+\gamma)n^{3/2} - 1 $.
\label{theo:large}
\end{theorem}


As a consequence of the strict-saddle property, we can prove that each component of $\mathcal{R}_{1}$ contains exactly one local minimizer when $\beta$ is sufficiently large. In the next lemma, we first discuss uniqueness of local minimizer if the Riemannian Hessian is positive definite on a certain subset of the sphere. 
\begin{lemma} Let $\nu \in (0,1]$ and ${\bf z}_0 \in \Sn$ be given and let us define $\mathcal R_\nu := \{\mathbf{z}\in\mathbb{S}^{n-1}:\lVert\mathbf{z}-\mathbf{z}_{0}\rVert^{2}\leq \nu\}$. If the Riemannian Hessian $\Hess{f}$ of $f$ is positive definite on $\mathcal R_\nu$, i.e., if we have
\[ {\bf v}^T \Hess{f({\bf z})} {\bf v} = H_f({\bf z})[{\bf v}] > 0, \quad \forall~{\bf v} \in \mathcal T_{{\bf z}}\mathcal M \backslash \{0\}, \quad \forall~{\bf z} \in \mathcal R_\nu, \]
then the problem $\min_{{\bf z} \in \mathcal R_\nu} f({\bf z})$ has at most one local minimizer.
\label{lem:convex}
\end{lemma}
\begin{proof}
Suppose that there exist two different local minima $\mathbf{z}_{1},\mathbf{z}_{2}$ of $f$ in the set $\mathcal{R}_\nu$. Let us consider the geodesic curve $\ell : \R \to \Sn$ on the sphere connecting $\mathbf{z}_{1}$ and $\mathbf{z}_{2}$. As shown in \cite[Example 5.4.1]{Absil2009Optimization}, the curve $\ell$ can be represented as follows
\[ \ell(t) := {\bf z}_1 \cos(t)+{\bf v}\sin(t), \quad {\bf v}:=\frac{{\bf z}_{2}-({\bf z}_{2}^{T}{\bf z}_{1}){\bf z}_{1}}{\|{\bf z}_{2}-({\bf z}_{2}^{T}{\bf z}_{1}){\bf z}_{1}\|} \in \mathcal T_{{\bf z}_1}\mathcal M \cap \Sn, \]
where $T\in(0,2\pi)$ is chosen such that $\ell(T) = {\bf z}_2$. Multiplying $\ell(T)={\bf z}_{2}$ with ${\bf v}^{T}$ from the left yields $\sin(T)={\bf v}^{T}{\bf z}_{2}\geq0$ and hence, we have $T\in(0,\pi]$. Then, for all $t \in (0,T)$, it follows $\ell(t)^T \ell^\prime(t) = 0$, $\|\ell^\prime(t)\| = 1$, and 
%
%
\begin{align*}
\lVert\ell(t)-\mathbf{z}_{0}\rVert^{2}=2-2\left(\mathbf{z}_{0}^{T}{\bf z}_1\cdot\cos(t)+\mathbf{z}_{0}^{T}{\bf v}\cdot\sin(t)\right)=2-2M\cos(t+\theta),
\end{align*}
where $M=\sqrt{(\mathbf{z}_{0}^{T}{\bf z}_{1})^{2}+(\mathbf{z}_{0}^{T}{\bf v})^{2}}$ and $\theta=\arccos(\mathbf{z}_{0}^{T}{\bf z}_{1})$. Since $\ell(0)$ and $\ell(T)$ are both elements of $\mathcal{R}_\nu$, we know that
\begin{align*}
\lVert\mathbf{z}_{0}-\ell(0)\rVert^{2}\leq \nu\leq1,\quad\lVert\mathbf{z}_{0}-\ell(T)\rVert^{2}\leq \nu\leq1,
\end{align*}
which lead to $\cos(\theta),\cos(\theta+T)>0$. Due to $T\in(0,\pi]$, we have $\theta,\theta+T\in\left(-\frac{\pi}{2}+2k\pi,\frac{\pi}{2}+2k\pi\right)$ for some $k\in\mathbb{Z}$. The landscape of $\cos(t)$ in this range shows that the maximal value of $\lVert\mathbf{z}_{0}-\ell(t)\rVert^{2}$ is achieved by the endpoints $t=0$ or $t=T$. Hence, it holds that $\lVert\mathbf{z}_{0}-\ell(t)\rVert^{2}\leq\nu$ and $\ell(t)\in\mathcal{R}_\nu$ for all $t\in[0,T]$.

The special form of the curve $\ell$ yields $\ell^{\prime\prime}(t)=-\ell(t)$ for all $t \in \R$ and thus, the second-order derivative of the continuous function $g(t):=f(\ell(t))$ is given by
\begin{align*}
g^{\prime\prime}(t)&=\ell^{\prime\prime}(t)^{T}\nabla f(\ell(t))+\ell^{\prime}(t)^{T}\nabla^{2}f(\ell(t))\ell^\prime(t)\\
&=\ell^\prime(t)^{T}\left[\nabla^{2}f(\ell(t))-\ell(t)^{T}\nabla f(\ell(t))\cdot I\right]\ell^\prime(t)\\
&=\ell^\prime(t)^{T}\Hess{f(\ell(t))}\ell^\prime(t)>0
\end{align*}
for all $t\in(0,T)$, which implies that $g$ is strictly convex on $[0,T]$. Per assumption the points $\mathbf{z}_{1}$ and $\mathbf{z}_{2}$ are local minima of the problem $\min_{\mathbf{z}\in\mathcal{R}_\nu} f(\mathbf{z})$ and thus, also of the problem $\min_{t\in[0,T]}g(t)$. However, this contradicts the strict convexity of $g$. 
\end{proof}

\begin{corollary}
If $\beta > \rho n^{2}$, the problem \eqref{eqn:obj} has at least $2^{n}$ local minima. Furthermore, if $\beta>\frac{18n^{3}}{n-1}\rho$, then the problem \eqref{eqn:obj} has exactly $2^{n}$ local minima.
\label{cor:large}
\end{corollary}
\begin{proof}
Without loss of generality, we can assume $\lambda_n(A) = \lambda_n = 0$. We first prove that there exists at least one local minimizer in each component of the region $\mathcal{R}_{1}$ if $\beta\geq\rho n^{2}$. 
Notice that we have ${\bf z} \in \mathcal R_1$ if and only if $z_k^2 \in [\frac{1}{2n},\frac{3}{2n}]$ for all $k \in [n]$. 
Let $\sigma \in \{-1,+1\}^n$ be a given binary vector and let us define the sets $\mathcal B := \{\mathbf{z}\in\mathbb{S}^{n-1}: \| |{\bf z}|^2 - \frac{1}{n} {\mathds 1}\| \leq \frac1n \}$ and $\mathcal P_\sigma := \prod_{k \in [n]} \sigma_k \R_+$. We now show that for each possible choice of $\sigma$ the set $\mathcal B \cap \mathcal P_\sigma$ contains a local minimizer of $f$. Let us note that, for all ${\bf z} \in \Sn$, the condition $\||{\bf z}|^2 - \frac{1}{n}{\mathds 1}\| \leq \frac{1}{n}$ is equivalent to $\|{\bf z}\|_4^4 \leq \frac1n + \frac{1}{n^2}$. This observation can be used to establish $z_k \neq 0$ for all ${\bf z} \in \mathcal B \cap \mathcal P_\sigma$ and consequently, the sets $\mathcal B \cap \mathcal P_\sigma$ and $\mathcal B \cap \mathcal P_\nu$ do not intersect for different binary vectors $\sigma \neq \nu$. 
Now, for all ${\bf z} \in \mathcal D := \{{\bf z} \in \Sn: \|{\bf z}\|_4^4 = \frac1n + \frac{1}{n^2}\}$, we have $f({\bf z}) \geq \frac{\beta}{2}(n^{-1} + n^{-2})$ and 
%
setting $\mathbf{z}_{\sigma}=\sigma/\sqrt{n} \in \mathcal B \cap \mathcal P_\sigma$, we obtain
\begin{align}
\label{eq:loc-bound-sig}
f(\mathbf{z}_{\sigma})=\frac{1}{2}\mathbf{z}_{\sigma}^{T}A\mathbf{z}_{\sigma}+\frac{\beta}{2n}\leq\frac{\rho}{2}+\frac{\beta}{2n} < \frac{\beta}{2}\left[\frac{1}{n}+\frac{1}{n^{2}}\right].
\end{align}
Hence, the global minimizer ${\bf y}$ of the problem $\min_{{\bf z} \in \mathcal B \cap \mathcal P_\sigma} f({\bf z})$ satisfies $\|{\bf y}\|_4^4 < \frac1n + \frac{1}{n^2}$. This implies that the Lagrange multiplier associated with the constraint $\|{\bf z}\|_4^4 \leq \frac1n + \frac{1}{n^2}$ is zero, and the KKT conditions for ${\bf y}$ reduce to $\grad{f({\bf y})} = 0$. Due to $\mathcal B \cap \mathcal P_\sigma \subset \mathcal R_1 \cap \mathcal P_\sigma$, the Riemannian Hessian $\Hess{f({\bf y})}$ is positive definite on the tangent space $\mathcal T_{{\bf y}}\mathcal M \backslash \{0\}$ and thus, by Lemma \ref{lemma:son-sos}, the points {\bf y} is one of at least $2^n$ isolated local minimum of problem \eqref{eqn:obj}.
%

Next, we consider the case when $\beta>\frac{18n^{3}}{n-1}\rho$. We introduce the following refined versions of the $\mathcal{R}_{1}$ and $\mathcal{R}_{2}$:
\begin{align*}
&\bar{\mathcal{R}}_{1}=\bigcup_{\sigma \in \{\pm 1\}^n} \left\{\mathbf{z}\in\mathbb{S}^{n-1}: \| {\bf z} - {\textstyle \frac{1}{\sqrt{n}}}\sigma\| \leq {\textstyle \frac{2}{9\sqrt{n}}} \right \} ,\\
&\bar{\mathcal{R}}_{2}=\bigcap_{\sigma \in \{\pm 1\}^n} \left \{\mathbf{z}\in\mathbb{S}^{n-1}: \| {\bf z} - {\textstyle \frac{1}{\sqrt{n}}}\sigma\| \geq {\textstyle \frac{2}{9\sqrt{n}}},{\min}_{k\in[n]}z_{k}^{2}\geq {\textstyle\frac{n-1}{4n^{2}}} \right\}.
\end{align*}
It can be easily seen that the set $\bar{\mathcal{R}}_{1}$ consists of $2^{n}$ non-intersecting components and that the three regions $\bar{\mathcal{R}}_{1}$, $\bar{\mathcal{R}}_{2}$, and $\mathcal R_3$ cover the sphere $\mathbb{S}^{n-1}$. 
Now, let $\mathbf{z}\in \bar{\mathcal{R}}_{1}$ be arbitrary. Then, there exists $\sigma \in \{\pm 1\}^n$ such that
\begin{align*}
\left\lvert z_{k}-\frac{\sigma_{k}}{\sqrt{n}}\right\rvert\leq\frac{2}{9\sqrt{n}} \quad \text{and} \quad \left| z_k + \frac{\sigma_k}{\sqrt{n}}\right| \leq \frac{20}{9\sqrt{n}}, \quad\forall~k\in[n].
\end{align*}
Hence, it follows $\||{\bf z}|^2 - \frac{1}{n} {\mathds 1}\|_\infty \leq \frac{40}{81n} < \frac{1}{2n}$, 
which implies $\bar{\mathcal{R}}_{1} \subset \mathcal{R}_{1}$. Thus, the strong convexity property also holds on $\bar{\mathcal{R}}_{1}$. We now prove the Riemannian gradient is lower bounded on $\bar{\mathcal{R}}_{2}$. For every $\mathbf{z}\in\bar{\mathcal{R}_{2}}$, there exists $\sigma \in \{\pm1\}^n$ such that $\sigma_k z_k = |z_k|$ for all $k \in [n]$. Consequently, we obtain
\begin{align*}
\||{\bf z}|^2 - {n}^{-1}{\mathds 1}\|_\infty = \max_{k \in [n]} \left| z_k - \frac{\sigma_k}{\sqrt{n}} \right| \left| |z_k| + \frac{1}{\sqrt{n}} \right| \geq \frac{\|{\bf z} - \frac{\sigma}{\sqrt{n}} \|_\infty}{\sqrt{n}} \geq \frac{2}{9n\sqrt{n}} \end{align*}
and by mimicking the steps and estimates in the proof Lemma \ref{lem:large-2}, we get 
\begin{align*}
2[\|{\bf z}\|_6^6 - \|{\bf z}\|_4^8]  \geq\frac{n}{4}\left(\frac{2}{9n\sqrt{n}}\right)^{2}\left(\frac{n-1}{4n^{2}}\right)^{2}=\left(\frac{n-1}{36n^{3}}\right)^{2}.
\end{align*}
Thus, the norm of the Riemannian gradient is lower bounded by
\begin{align*}
\lVert\grad{f(\mathbf{z})}\rVert \geq\frac{1}{\sqrt{2}}\left[2\beta\sqrt{2[\|{\bf z}\|_6^6 - \|{\bf z}\|_4^8]}-\rho\right] \geq\frac{1}{\sqrt{2}}\left(\beta\cdot\frac{n-1}{18n^{3}}-\rho\right)>0
\end{align*}
and all local minima should need to be located in the set $\bar{\mathcal{R}}_{1}$. Applying Lemma \ref{lem:convex}, each connected component of $\bar{\mathcal{R}}_{1}$ contains at most one local minimizer and hence, problem \eqref{eqn:obj} has exactly $2^{n}$ local minima.
\end{proof}

Although the local minima of problem \eqref{eqn:obj} are not unique, all the local minima have similar objective function values. 
\begin{theorem}
Suppose that $\beta > \frac{18n^3}{n-1}\rho$. Then, it follows 
\begin{equation} \label{eq:loc-glob-diff} f(\mathbf{y})- \min_{{\bf z} \in \Sn} f(\mathbf{z}) \leq \frac{1}{18{n}} \cdot \min_{{\bf z} \in \Sn}f(\mathbf{z}),
\end{equation}
for all local minimizer ${\bf y} \in \Sn$ of problem \eqref{eqn:obj}.
\end{theorem}
\begin{proof} Without loss of generality, we can assume that the smallest eigenvalue of the matrix $A$ is zero. Then, for all ${\bf z} \in \Sn$, we have $f({\bf z}) \geq \frac{\beta}{2} \|{\bf z}\|_4^4 \geq \frac{\beta}{2n}$.
According to the analysis in Corollary \ref{cor:large}, each component ${\mathcal B} \cap \mathcal P_\sigma$, $\sigma \in \{\pm1\}^n$, contains exactly one local minimizer ${\bf y}$ of problem \eqref{eqn:obj} which is also the unique global minimizer of the restricted problem $\min_{{\bf z} \in \mathcal B \cap \mathcal P_\sigma} f({\bf z})$. Together with \eqref{eq:loc-bound-sig}, this yields
\[ f({\bf y}) \leq f({\bf z}_\sigma) \leq \frac{\rho}{2} + \frac{\beta}{2n} \leq \left[ 1 + \frac{n\rho}{\beta} \right] f({\bf z}), \quad \forall~{\bf z} \in \Sn. \]
%
Finally, the estimate \eqref{eq:loc-glob-diff} can be established via minimizing the latter expression with respect to ${\bf z}$ and using the bound on $\beta$.
\end{proof}

In the remainder of this section, we present an example demonstrating that the bound $\beta \approx C \rho n^{3/2}$ and the dependence on $n^{3/2}$ can not be further improved and that the strict-saddle property is violated whenever a smaller coefficient is chosen.

\begin{example} Let $C>0$ and $\epsilon > 0$ be given constants and suppose $\beta = C n^{3/2-\epsilon}$. In the following, we construct a specific matrix $A \in \R^{n \times n}$ and a point ${\bf z} \in \mathbb{S}^{n-1}$ such that the three conditions of the strict-saddle property do not hold at ${\bf z}$. We set
\[ z_1 = \left[\frac{1}{3 n- 2}\right]^\frac12, \quad z_k = \left[\frac{3}{3 n - 2}\right]^\frac12, \quad k \geq 2, \] 
${\bf u} = 2\beta \|{\bf z}\|_4^4 \cdot {\bf z} - 2\beta \diag(|{\bf z}|^2) {\bf z}$, and $A = \alpha {\bf w}{\bf w}^T + {\bf z}{\bf u}^T + {\bf u}{\bf z}^T$, where
\[ w_1 = - \left[ \frac{3n-3}{3n-2} \right]^\frac12, \quad w_k = \left[ \frac{1}{(3n-2)(n-1)} \right]^\frac12, \, k \geq 2, \quad \alpha =  - \frac{16\beta}{(3n-2)^2}. \]
Then, we have $\|{\bf z}\| = \|{\bf w}\| = 1$, $\|{\bf z}\|_4^4 = (9n-8) / (3n-2)^2$, and
\begin{align*} \|{\bf u}\|^2  & = 4\beta^2 [\|{\bf z}\|^6_6 - \|{\bf z}\|_4^8] = 4\beta^2 \left[ \frac{1+27(n-1)}{(3 n-2)^{3}}-\|{\bf z}\|_4^8 \right]  = \frac{48(n-1)}{(3 n-2)^{4}} \cdot \beta^2. \end{align*}
Furthermore, it holds that ${\bf u}^T {\bf z} = {\bf w}^T {\bf z}= 0$
%
%
which implies $A{\bf z} = {\bf u}$ and $\grad{f({\bf z})} = 0$. The eigenvalues of the matrices $\alpha {\bf w}{\bf w}^T$ and ${\bf z}{\bf u}^T + {\bf u}{\bf z}^T$ are $0$, ..., $0$, $\alpha$ and $\|{\bf u}\|$, $0$, ..., $0$, $- \|{\bf u}\|$, respectively. Thus, by Weyl's inequality, it follows
\[ \rho = \lambda_1(A) - \lambda_n(A) \leq 2 \|{\bf u}\| - \alpha = 8\beta \cdot \frac{2 + \sqrt{3n-3}}{(3n-2)^2} \leq \bar C n^{-\epsilon} = O(n^{-\epsilon}),  \]
for some constant $\bar C$. Let us now consider an arbitrary vector ${\bf v} \in \mathcal T_{\bf z}\mathcal M \cap \Sn$.  We have $\sum_{k=2}^n v_k^2 = 1 - v_1^2$, $v_1 + \sqrt{3} \sum_{k=2}^n v_k = 0$, and
\[  {\bf v}^T {\bf w} = v_1 w_1 - \frac{1}{\sqrt{3}} \left[ \frac{1}{(3n-2)(n-1)} \right]^\frac12 v_1 = -\left[ \frac{3n-2}{3n-3} \right]^{\frac12}  \cdot v_1.  \]
Hence, we obtain
\begin{align*} H_f({\bf z})[{\bf v}] & = \alpha ({\bf v}^T{\bf w})^2 + 6\beta {\bf v}^T \diag(|{\bf z}|^2) {\bf v} - 2\beta \|{\bf z}\|_4^4 \\ & = \frac{3n-2}{3n-3} \alpha v_1^2 - \frac{12\beta}{3n-2} v_1^2 + \frac{18\beta}{3n-2} - 2\beta \|{\bf z}\|_4^4. \end{align*}
Since $H_f({\bf z})[{\bf v}]$ is linear in $v_1^2$, its maximum and minimum value are attained at the boundary of the range interval $[\underline{\nu},\overline{\nu}]$ of $v_1^2$. Notice that we have $\underline{\nu} = 0$ and $\overline{\nu}$ can be found by discussing the optimization problem
\[ \min_{\bf v}~-v_1 \quad \st \quad v_1 + \sqrt{3} \sum_{k=2}^n v_k = 0, \quad \|{\bf v}\| = 1 \]
and its associated KKT conditions. In particular, it can be shown that $\overline{\nu} = (3n-3)/(3n-2)$. In the case $v_1^2 = 0$, we obtain $H_f({\bf z})[{\bf v}] > 0$ and for $v_1^2 = \overline{\nu}$, we have
\[ H_f({\bf z})[{\bf v}] = \alpha - \frac{36\beta (n-1)}{(3n-2)^2} + \frac{18\beta}{3n-2} - \frac{2\beta(9n-8)}{(3n-2)^2} = \alpha + \frac{16\beta}{(3n-2)^2} = 0. \]
Consequently, we can infer $\min_{{\bf v} \in \mathcal T_{\bf z}\mathcal M \cap \Sn} H_f({\bf z})[{\bf v}] = 0$ and $H_f({\bf z})[{\bf v}] \geq 0$ for all ${\bf v} \in \mathcal T_{\bf z}\mathcal M \cap \Sn$, which shows that none of the conditions of the strict-saddle property hold at ${\bf z}$. Thus, the order $n^{3/2}$ can not be improved in the deterministic case.
\end{example}

\begin{figure}[t]
\centering
\setlength{\belowcaptionskip}{-6pt}
\begin{tabular}{cc}
\subfloat[$\beta = 0.2$]{\includegraphics[trim={0 1.5cm 0 2cm},clip,width=5.5cm]{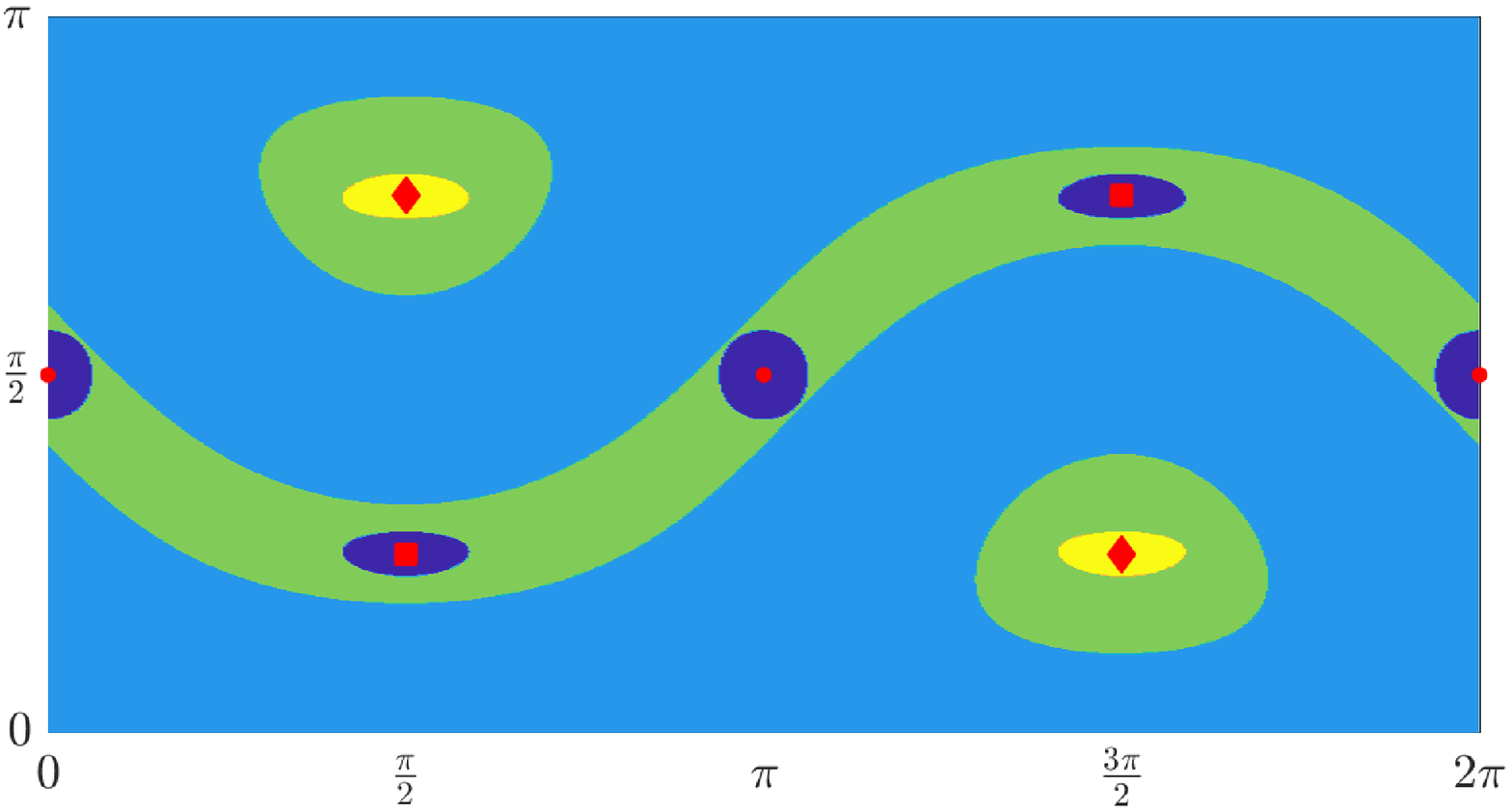}} &
\subfloat[$\beta = 3.75$]{\includegraphics[trim={0 1.5cm 0 2cm},clip,width=5.5cm]{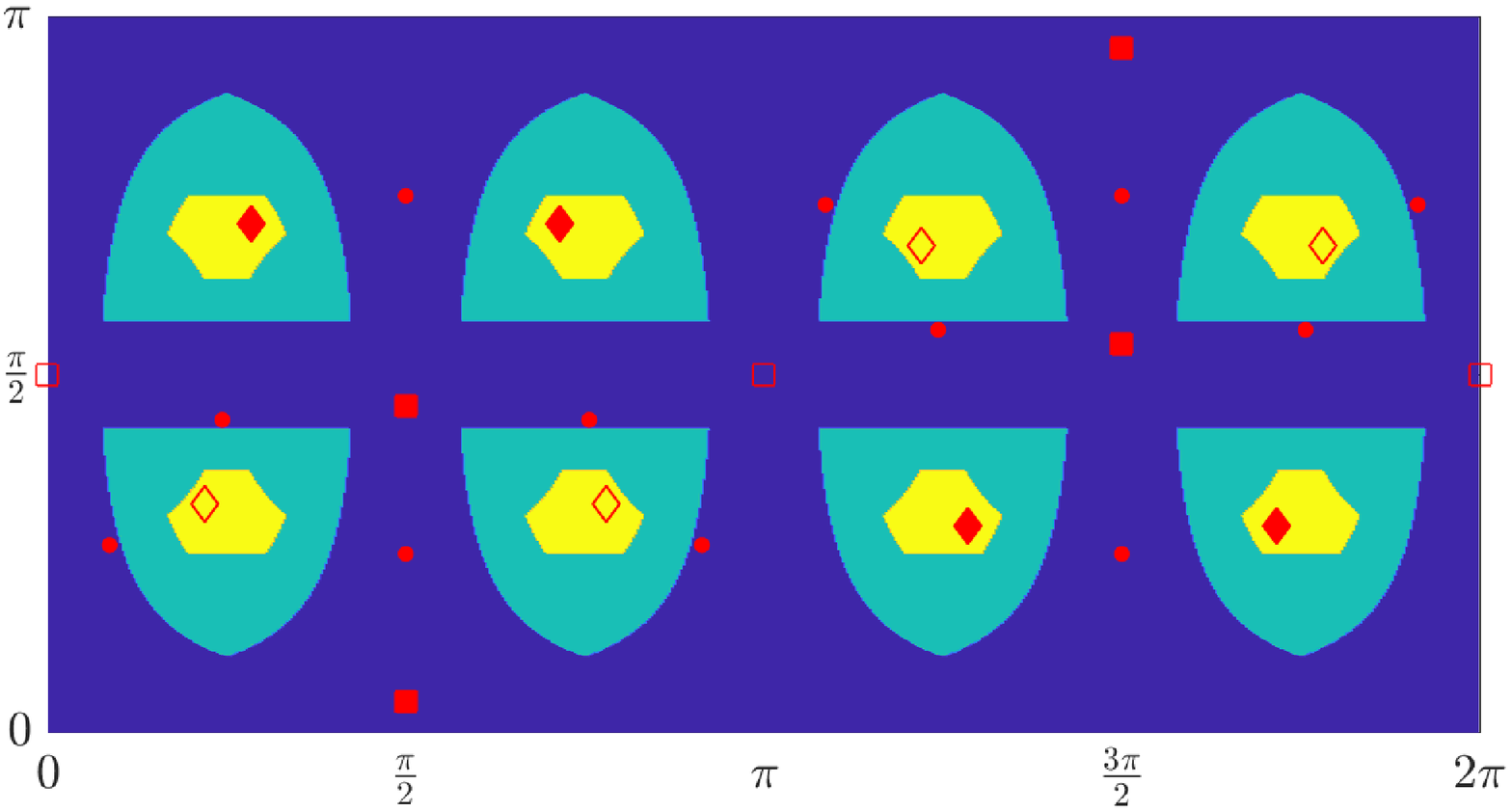}} 
\end{tabular}
\caption{Plot of the sets $\mathcal R_1$--$\mathcal R_3$ for the problem \eqref{eq:example-A}. Figure (a) depicts the regions introduced in section \ref{sec:small} for a small interaction coefficient $\beta = 0.2$. The overlap of the sets $\mathcal R_1$--$\mathcal R_2$ and $\mathcal R_2$--$\mathcal R_3$ is shown in green. The set $\mathcal R_1$ is the union of the yellow and the two surrounding green areas, while $\mathcal R_2$ is the union of all green and light blue areas. The region $\mathcal R_3$ coincides with the union of the dark blue sets and the enclosing green area. Figure (b) shows the regions introduced in section \ref{sec:large} for large $\beta$. Here, the (disjoint) yellow, turquoise, and dark blue areas directly correspond to the sets $\mathcal R_1$, $\mathcal R_2$, and $\mathcal R_3$, respectively. The red point marker again depicts the location of saddle points. Non-filled and filled diamond markers are used for local and global minima. Local and global maxima are marked by non-filled and filled squares.}
\label{figure:areas}
\end{figure}

\subsection{Small interaction coefficient}
\label{sec:small}

In the following, we discuss the geometric landscape of problem (\ref{eqn:obj}) for small interaction coefficients. We additionally assume that there is a positive spectral gap $\delta=\lambda_{n-1}-\lambda_{n}>0$ between the two smallest eigenvalues of the matrix $A$. As shown by Marvcenko \cite{marvcenko1967distribution}, this condition holds with probability $1$ when $A$ is a Gaussian random matrix.

Let us recall that the eigenvalue decomposition of $A$ is given by $A = P \Lambda P^T$, where $\Lambda = \diag(\lambda_1,\lambda_2,...,\lambda_n)$ and $P = (\mathbf{p}_{1},\mathbf{p}_{2},...,\mathbf{p}_{n})$ is an orthogonal matrix. Similar to the method used in section \ref{sec:large}, we now divide $\mathbb{S}^{n-1}$ into three sub-regions:
\begin{itemize}
  \item[1.](Strong convexity) $\mathcal{R}_{1}: =\{\mathbf{z} \in \mathbb{S}^{n-1}: {\bf z} =\sum_{k=1}^n \alpha_{k}\mathbf{p}_{k}, \, \alpha_{n}^{2}\geq\frac{(2+\gamma)\beta+\rho}{\delta+\rho}\}$.
  \item[2.](Large gradient) $\mathcal{R}_{2} :=\{\mathbf{z} \in\mathbb{S}^{n-1}: \|A{\bf z}\|^2 - ({\bf z}^T A{\bf z})^2 \geq (\frac23+\gamma)^2\beta^{2}\}$.
  \item[3.](Negative curvature) $\mathcal{R}_{3} := \{\mathbf{z} \in\mathbb{S}^{n-1}: {\bf z}=\sum_{k = 1}^n \alpha_{k}\mathbf{p}_{k}, \, \alpha_{n}^{2}\leq\frac{\delta-(4+\gamma)\beta}{\delta+\rho}\}$.
\end{itemize}

An exemplary illustration of the sets $\mathcal R_1$--$\mathcal R_3$ is given in Figure \ref{figure:areas}. We first show that the Riemannian Hessian is uniformly positive definite on $\mathcal{R}_{1}$.
\begin{lemma}
Suppose that the gap between the two smallest eigenvalues of the matrix $A$ satisfies $\delta:=\lambda_{n-1}-\lambda_{n}>0$ and let $\beta, \gamma > 0$ be given. Then, for all $\mathbf{z}\in\mathcal{R}_{1}$ and all $\mathbf{v}\in\mathcal{T}_{\mathbf{z}}\mathcal{M}\cap\mathbb{S}^{n-1}$, it follows $H_{f}(\mathbf{z})[\mathbf{v}]\geq\gamma\beta$.
\label{lem:small-1}
\end{lemma}
\begin{proof}
Let ${\bf v} \in \mathcal{T}_{\mathbf{z}}\mathcal{M}\cap\mathbb{S}^{n-1}$ be arbitrary with $\mathbf{v}=\sum_{k\in[n]} \nu_{k}\mathbf{p}_{k}$. Then, we have $\sum_{k\in[n]}\alpha_{k} \nu_{k}=0$ and $\sum_{k\in[n]}\nu_{k}^{2}=1$ and the Cauchy inequality implies
\begin{align*}
\vert \nu_{n}\alpha_{n}\rvert&=\left\lvert{\sum}_{k=1}^{n-1}\nu_{k}\alpha_{k}\right\rvert\leq \left[{\sum}_{k=1}^{n-1}\nu_{k}^{2}\right]^\frac12\left[{\sum}_{k=1}^{n-1}\alpha_{k}^{2}\right]^\frac12 = \sqrt{(1-\nu_{n}^{2})(1-\alpha_{n}^{2})}.
\end{align*}
Thus, it follows $\nu_{n}^{2}\alpha_{n}^{2}\leq(1-\nu_{n}^{2})(1-\alpha_{n}^{2})$ and $\alpha_{n}^{2}\leq1-\nu_{n}^{2}$ and, due to $\|{\bf z}\|_4^4 \leq 1$, we obtain
\begin{align*}
H_{f}(\mathbf{z})[\mathbf{v}]&=\mathbf{v}^{T}A\mathbf{v}-\mathbf{z}^{T}A\mathbf{z}+6\beta{\sum}_{k=1}^n z_{k}^{2}v_{k}^{2} - 2\beta \|{\bf z}\|^4_4 \\
&\geq {\sum}_{k=1}^n \lambda_{k}(\nu_{k}^{2}-\alpha_{k}^{2})-2\beta \\
&\geq{\sum}_{k=1}^{n-1}(\lambda_{n}+\delta)\nu_{k}^{2}+\lambda_{n}\nu_{n}^{2}-{\sum}_{k=1}^{n-1}(\lambda_{n}+\rho)\alpha_{k}^{2}-\lambda_{n}\alpha_{n}^{2}-2\beta\\
&=(1-\nu_{n}^{2})\delta-(1-\alpha_{n}^{2})\rho-2\beta \\ 
& \geq \alpha_{n}^{2}\delta-(1-\alpha_{n}^{2})\rho-2\beta = \alpha_{n}^{2}(\delta+\rho)-\rho-2\beta \geq \gamma\beta,
\end{align*}
as desired.
\end{proof}

Next, we prove that the norm of the Riemannian gradient is strictly larger than zero on the set $\mathcal{R}_{2}$.
\begin{lemma}
For all $\mathbf{z}\in\mathcal{R}_{2}$, it holds that $\lVert\grad{f(\mathbf{z})}\rVert_{2}\geq \gamma\beta$ .
\label{lem:small-2}
\end{lemma}
\begin{proof}
As in the proof of Lemma \ref{lem:large-2}, we have 
\begin{align*}
\left\lVert\left(\diag(|{\bf z}|^2)- \|{\bf z}\|_4^4 \cdot I_{n}\right)\mathbf{z}\right\rVert ^{2} = \|{\bf z}\|_6^6 - \|{\bf z}\|_4^8 \leq \|{\bf z}\|_4^6 (1 - \|{\bf z}\|_4^2) \leq \frac14 \left[ \frac34\right]^3,
\end{align*}
where the last estimate follows from the fact that the mapping $x \mapsto x^6-x^8$ attains its global maximum at $x = \pm \frac{\sqrt{3}}{2}$. 
Hence, we obtain
\begin{align*}
\lVert\grad{f(\mathbf{z})}\rVert_{2}&=\left\lVert(A-\mathbf{z}^{T}A\mathbf{z}\cdot I_{n})\mathbf{z}+2\beta\left(\diag(|{\bf z}|^2)- \|{\bf z}\|_4^4 \cdot I_{n}\right)\cdot\mathbf{z}\right\rVert\\
&\geq\lVert(A-\mathbf{z}^{T}A\mathbf{z}\cdot I_{n})\mathbf{z}\rVert-2\beta\left\lVert\left(\diag(|{\bf z}|^2)- \|{\bf z}\|_4^4 \cdot I_{n}\right)\mathbf{z}\right\rVert \\
&\geq\lVert(A-\mathbf{z}^{T}A\mathbf{z}\cdot I_{n})\mathbf{z}\rVert -\frac{3\sqrt{3}}{8} \beta= \left[ \frac23 - \frac{3\sqrt{3}}{8} \right ] \beta + \gamma\beta \geq \gamma\beta.
\end{align*}
\end{proof}

Finally, for points in the region $\mathcal{R}_{3}$, we construct directions along which the curvature of the objective function is strictly negative.
\begin{lemma}
Suppose that the gap between the two smallest eigenvalues of the matrix $A$ satisfies $\delta:=\lambda_{n-1}-\lambda_{n}>0$ and let $\gamma > 0$ be given. If $\beta \leq (4+\gamma)^{-1}\delta$, then for all $\mathbf{z}\in\mathcal{R}_{3}$ there exists $\mathbf{v}\in\mathcal{T}_{\mathbf{z}}\mathcal{M}\cap\mathbb{S}^{n-1}$ such that $H_{f}(\mathbf{z})[\mathbf{v}]\leq-\gamma\beta$.
\label{lem:small-3}
\end{lemma}
\begin{proof}
By the Cauchy's inequality and using the estimate $|3x-x^{2}| \leq 2$, $x \in [0,1]$, we have
\begin{align*}
\sum_{k\in[n]} 3z_{k}^{2}v_{k}^{2}-z_{k}^{4} \leq 3\lVert\mathbf{z}\rVert_{4}^{2}\lVert\mathbf{v}\rVert_{4}^{2}-\lVert\mathbf{z}\rVert_{4}^{4}\leq3\lVert\mathbf{z}\rVert_{4}^{2}-\lVert\mathbf{z}\rVert_{4}^{4}\leq2.
\end{align*}
Next, we choose a specific direction $\mathbf{v}=\sum_{k\in[n]} \nu_{k}\mathbf{p}_{k}$ that satisfies $H_{f}(\mathbf{z})[\mathbf{v}]\leq- \gamma\beta$, ${\bf z}^T {\bf v} = \sum_{k\in[n]}\alpha_{k}\nu_{k}=0$ and $\|{\bf v}\|^2 = \sum_{k \in [n]} \nu_k^2 = 1$.

\textit{Case 1. $\alpha_{n}=0$. Let us set $\nu_{n}=1$ and $\nu_{k}=0$ for all $ k\in [n-1]$.} Then, we obtain
\begin{align*}
H_{f}(\mathbf{z})[\mathbf{v}]&=\mathbf{v}^{T}A\mathbf{v}-\mathbf{z}^{T}A\mathbf{z}+2\beta{\sum}_{k=1}^n(3z_{k}^{2}v_{k}^{2}-z_{k}^{4}) \\
& \leq {\sum}_{k = 1}^n \lambda_{k}(\nu_{k}^{2}-\alpha_{k}^{2})+4\beta=\lambda_{n}-{\sum}_{k=1}^{n-1}\lambda_{k}\alpha_{k}^{2}+4\beta\\
&\leq\lambda_{n}- (\lambda_n + \delta ){\sum}_{k=1}^{n-1} \alpha_{k}^{2}+4\beta=-\delta+4\beta\leq- \gamma \beta,
\end{align*}
where the last inequality immediately follows from $\beta \leq \frac{\delta}{4+\gamma}$.

\textit{Case 2. $0< \alpha_{n}^{2}\leq\frac{\delta-(4+\gamma)\beta}{\delta+\rho}$.} In this situation, we set
\begin{align*}
\nu_{k} = \eta \alpha_{k},\quad \forall~k \in [n-1],\quad \nu_{n}=- \sqrt{1-\alpha_{n}^{2}}, \quad \eta = \frac{\alpha_{n}}{\sqrt{1-\alpha_{n}^{2}}}.
\end{align*}
With this choice, we have ${\bf z}^T{\bf v} = \eta(1-\alpha_n^2) - \alpha_n \sqrt{1-\alpha_n^2} = 0$ and $\|{\bf v}\|^2 = (\eta^2 + 1)(1-\alpha_n^2) = 1$, i.e., it holds that $\mathbf{v} \in\mathcal{T}_{\mathbf{z}}\mathcal{M}\cap\mathbb{S}^{n-1}$. Similar to the calculations in the proof of Lemma \ref{lem:small-1}, we now get
\begingroup
\allowdisplaybreaks
\begin{align*}
H_{f}(\mathbf{z})[\mathbf{v}]&=\mathbf{v}^{T}A\mathbf{v}-\mathbf{z}^{T}A\mathbf{z}+2\beta{\sum}_{k=1}^{n}(3z_{k}^{2}v_{k}^{2}-z_{k}^{4})\\
&\leq {\sum}_{k=1}^{n}\lambda_{k}(\nu_{k}^{2}-\alpha_{k}^{2})+4\beta\\
&\leq {\sum}_{k=1}^{n-1}((\lambda_{n}+\rho)\nu_{k}^{2}-(\lambda_{n}+\delta)\alpha_{k}^{2})+\lambda_{n}(\nu_{n}^{2}-\alpha_{n}^{2})+4\beta\\
&=\rho \cdot {\sum}_{k=1}^{n-1} \nu_{k}^{2}-\delta \cdot {\sum}_{k=1}^{n-1} \alpha_{k}^{2}+4\beta=\rho(1-\nu_{n}^{2})-\delta(1-\alpha_{n}^{2})+4\beta\\
&=\rho\eta^{2}(1-\alpha_{n}^{2})+\delta \alpha_{n}^{2}-\delta+4\beta=(\delta+\rho)\alpha_{n}^{2}-\delta+4\beta\leq-\gamma\beta.
\end{align*}
\endgroup
Combining the latter two cases, we can conclude the proof of Lemma \ref{lem:small-3}.
\end{proof}

We now verify that $f$ has the strict-saddle property whenever $\beta$ is chosen sufficiently small. 

\begin{theorem}
Suppose that the gap between the two smallest eigenvalues of the matrix $A$ satisfies $\delta:=\lambda_{n-1}-\lambda_{n}>0$ and let $\gamma > 0$ be given. If $\beta \leq [2(\frac73+\gamma) + (\frac23 + \gamma)\frac{\rho}{\delta}]^{-1}\delta =: b_\gamma$, then $f$ has the $(\gamma\beta,\gamma\beta,\gamma\beta)$-strict-saddle property.
\label{theo:small}
\end{theorem}
\begin{proof}
By Lemmata \ref{lem:small-1}-\ref{lem:small-3}, we know that the function $f$ satisfies the strong convexity, large gradient, and negative curvature property on the three different set $\mathcal R_1$, $\mathcal R_2$, and $\mathcal R_3$, respectively. To finish the proof, we need to show that those regions actually cover the whole sphere $\mathbb{S}^{n-1}$. Combining these observations, we can then conclude that $f$ has the $(\gamma\beta,\gamma\beta,\gamma\beta)$-strict-saddle property.

In order to prove $\mathcal{R}_{1}\cup\mathcal{R}_{2}\cup\mathcal{R}_{3}=\mathbb{S}^{n-1}$, we only need to verify that for all $\mathbf{z}=\sum_{k\in[n]} \alpha_{k}\mathbf{p}_{k}\in\mathbb{S}^{n-1}$ with 
\begin{equation} \label{eq:b-alpha} \frac{\delta-(4+\gamma)\beta}{\delta +\rho}\leq \alpha_{n}^{2}\leq\frac{(2+\gamma)\beta+\rho}{\delta+\rho}, \end{equation} 
we have $\|A{\bf z}\|^2 - ({\bf z}^T A {\bf z})^2 \geq (\frac{2}{3} + \gamma)^2\beta^{2}$. 
%
Using the bounds \eqref{eq:b-alpha}, it follows
\begin{align*}
\|A{\bf z}\|^2 - ({\bf z}^T A {\bf z})^2  
& = \frac{1}{2} \sum_{k,j\in[n]} \alpha_{k}^{2} \alpha_{j}^{2}(\lambda_{k}-\lambda_{j})^{2} \\ & \geq \sum_{k\in[n-1]} \alpha_{k}^{2}\alpha_{n}^{2}(\lambda_{k}-\lambda_{n})^{2}
 \geq\sum_{k\in[n-1]}\alpha_{k}^{2}\alpha_{n}^{2}\delta^{2}  \\ & =(1-\alpha_{n}^{2})\alpha_{n}^{2}\delta^{2} \geq\frac{(\delta-(4+\gamma)\beta)(\delta-(2+\gamma)\beta)}{(\delta+\rho)^{2}}\delta^{2}.
\end{align*}
where the first identity was established in the proof of Lemma \ref{lem:large-2}. Rearranging the terms in the latter estimate, we see that our claim is satisfied if 
\[ \ell(\beta) := \left[ ({\textstyle\frac23} + \gamma)^2(\delta+\rho)^2 - (4+\gamma)(2+\gamma) \right] \beta^2 + 2(3+\gamma) \delta^3 \beta - \delta^4 \leq 0 \]
for all $\beta \leq b_\gamma$. Since the unique nonnegative zero of the quadratic polynomial $\ell$ is given by
 \[  \bar \beta = \left[{3+\gamma + \sqrt{ 1+ ({\textstyle\frac23} + \gamma)^2(1+{\textstyle\frac{\rho}{\delta}})^2}} \right]^{-1}  \delta, \] 
 we can finish the proof by noticing $\bar \beta \geq b_\gamma$.

\end{proof}

Finally, as a counterpart of Corollary \ref{cor:large}, we can establish the uniqueness of local minima as a consequence of the strict-saddle property.
\begin{corollary}
Under the conditions of Theorem \ref{theo:small}, the problem \eqref{eqn:obj} has exactly two local minima which are also global minima.
\end{corollary}
\begin{proof}
Note that all local minima locate in $\mathcal{R}_{1}$ and that $\mathcal{R}_{1}$ consists of two symmetrical non-intersecting subsets. We now consider one of the subsets $\bar{\mathcal{R}}_{1}:=\{\mathbf{z}\in\Sn:\mathbf{z}=\sum_{k\in[n]}\alpha_{k}\mathbf{p}_{k},\alpha_{n}\geq\sqrt{\nu}\}$ where $ \nu := ({\delta+\rho})^{-1}[(2+\gamma)\beta+\rho]$. Using $\|{\bf z}-\mathbf{p}_{n}\|^{2}=2-2\alpha_{n}$, for ${\bf z} \in \bar {\mathcal R}_1$, an equivalent definition of this subset is given by
\begin{align*}
\bar{\mathcal{R}}_{1}=\left\{\mathbf{z}\in\Sn:\|\mathbf{z}-\mathbf{p}_{n}\|^{2}\leq2-2\sqrt{\nu}\right\}.
\end{align*}
In order to apply Lemma \ref{lem:convex}, we need to verify $2-2\sqrt{\nu}\leq1$ or, equivalently, $\nu \geq \frac{1}{4}$. However, due to $\gamma,\beta>0$ and $\rho\geq\delta>0$, we have $\nu \geq\frac{\rho}{\delta+\rho}\geq\frac{1}{2}>\frac{1}{4}$. Similary, the second subset can be represented by $\bar{\mathcal R}_2 := \{{\bf z} \in \Sn: \|{\bf z} + {\bf p}_n\|^2 \leq 2 - 2\sqrt{\nu}\}$. Then, by Lemma \ref{lem:convex}, there exist exactly two equivalent local minima which are also global minima of problem \eqref{eqn:obj}.
\end{proof}



\section{Estimation of the Kurdyka-\L ojasiewicz Exponent}
\label{sec:KLExp}

In this section, we estimate the Kurdyka-\L ojasiewicz (KL) exponent of problem \eqref{eqn:obj}. 
Specifically, we want to find the largest $\theta\in(0,\frac{1}{2}]$ such that for all stationary points $\mathbf{z}$ of problem (\ref{eqn:obj}), the \L ojasiewicz inequality,
\begin{equation}
\label{eq:kl}
\lvert f(\mathbf{y})-f(\mathbf{z})\rvert^{1-\theta}\leq\eta_{\mathbf{z}}\lVert\grad{f(\mathbf{y})}\rVert,\quad\forall~\mathbf{y}\in B(\mathbf{z},\delta_{\mathbf{z}})\cap\mathbb{CS}^{n-1},
\end{equation}
holds with some constants $\delta_{\mathbf{z}},\eta_{\mathbf{z}}>0$. The largest possible $\theta$ is called the KL exponent of problem \eqref{eqn:obj}.

As already mentioned, the {\L}ojasiewicz inequality \eqref{eq:kl} plays a fundamental role in nonconvex optimization and is frequently utilized to analyze the local convergence properties of nonconvex optimization methods, \cite{attouch2009convergence,AttBolRedSou10,AttBolSva13,bolte2014proximal,OchCheBroPoc14,BonLorPorPraReb17,liu2017quadratic,li2018calculus}. In \cite{attouch2009convergence}, Attouch and Bolte derived an abstract KL-framework based on the \L ojasiewicz inequality that allows to establish global convergence and local convergence rates for general optimization approaches satisfying certain function reduction and asymptotic step size safe-guard conditions. In particular, if $\theta\geq\frac{1}{2}$, then the corresponding iterates can be shown to converge linearly. Otherwise, if $\theta\in(0,\frac{1}{2})$, the iterates converge at a sublinear rate $O(t^{-\frac{\theta}{1-2\theta}})$. By introducing the auxiliary problem
\begin{align} 
\min_{{\bf z} \in \Cn}~\hat f({\bf z}), \quad
\hat{f}(\mathbf{z}) :=
\begin{cases}
f(\mathbf{z}) & \text{if } \mathbf{z}\in\CSn,\\
+\infty       & \text{otherwise},\\
\end{cases}
\label{eq:nonsmooth}
\end{align}
the original problem \eqref{eqn:obj} can be treated as the minimization of an extended real-valued, proper, and lower semicontinuous function which allows to apply existing results and the rich KL theory for nonsmooth problems, see, e.g., \cite{lojasiewicz1963propriete,Kur98,BolDanLew-MS-06,BolDanLew06,BolDanLewShi07}. 

In the nonsmooth setting, the Riemannian gradient, appearing in \eqref{eq:kl}, is typically substituted by the nonsmooth slope of $\hat f$ which is based on the Fr\'{e}chet and limiting subdifferential of $\hat f$.
In our case, if problem \eqref{eq:nonsmooth} is restricted to the real space $\Rn$, the limiting subdifferential and nonsmooth slope of $\hat{f}$ at $\mathbf{z}\in\Sn$ can be expressed as $\partial\hat{f}(\mathbf{z})= \{\nabla f(\mathbf{z})+\mu\mathbf{z}:\mu\in\mathbb{R}\}$ and 
\[ \argmin_{\mathbf{g}\in\partial\hat{f}(\mathbf{z})}\|\mathbf{g}\|= (I - {\bf z}{\bf z}^T)\nabla f({\bf z}) = \grad{f(\mathbf{z})}. \] 
Hence, the Riemannian-type {\L}ojasiewicz inequality \eqref{eq:kl} coincides with the standard notion and KL framework used in nonsmooth  optimization. Our goal is now to show that the KL exponent of \eqref{eqn:obj} is at least $\frac{1}{4}$ under suitable conditions.


Throughout this section, we assume that $\mathbf{z}\in\mathbb{CS}^{n-1}$ is a stationary point of problem \eqref{eqn:obj}. Furthermore, $\mathbf{y}\in\mathbb{CS}^{n-1}$ denotes a neighboring point of $\mathbf{z}$ and we set $\Delta=\mathbf{y}-\mathbf{z}$. 
We now collect and present some preparatory notations and computational results that will be used in the following derivations. Since $\mathbf{z}$ is a stationary point of problem (\ref{eqn:obj}), we have 
\begin{align}
H\mathbf{z} = A\mathbf{z}+2 \beta\diag(\lvert\mathbf{z}\rvert^{2})\mathbf{z}-2\lambda\mathbf{z} = 0, \quad 2\lambda = {\bf z}^*A{\bf z} + 2\beta\|{\bf z}\|_4^4
\label{eqn:6-1}
\end{align}
and as proved in \eqref{eq:objective_value}, it holds that
\begin{equation}
\begin{aligned}
f(\mathbf{y})-f(\mathbf{z})=\frac{1}{2}\mathbf{y}^{*}H\mathbf{y}+\frac{\beta}{2} \|\tau\|^2, \quad \tau_k := \lvert y_{k}\rvert^{2}-\lvert z_{k}\rvert^{2}, \quad k \in [n].
\end{aligned}
\label{eqn:6-2}
\end{equation}
%
The norm of the Riemannian gradient can be expressed as follows
%
\begingroup
\allowdisplaybreaks
\begin{align}
\nonumber \|\grad{f({\bf y})}\|^2 & = \half \| P^\bot_{\bf y}[H + 2\beta \diag(\tau)]{\bf y}\|^2 \\\nonumber & = \half \|[H + 2\beta \diag(\tau)] {\bf y}\|^2 - \half ({\bf y}^*[H + 2\beta \diag(\tau)]{\bf y})^2 \\ 
\nonumber &=\frac{1}{2}\mathbf{y}^{*}H^{2}\mathbf{y}-\frac{1}{2}(\mathbf{y}^{*}H\mathbf{y})^{2}+ 2\beta^{2} \mathbf{y}^{*}\diag(|\tau|^2)\mathbf{y}-2 \beta^{2}(\mathbf{y}^{*}\diag(\tau)\mathbf{y})^{2}\\
&\hspace{4ex}+2\beta\mathbf{y}^{*}H\diag(\tau)\mathbf{y}-2\beta(\mathbf{y}^{*}H\mathbf{y})(\mathbf{y}^{*}\diag(\tau)\mathbf{y}),
\label{eqn:6-4}
\end{align}
\endgroup
%
%
where $P_{\mathbf{y}}^{\perp}=I_{n}-\mathbf{yy}^{*}$ is the orthogonal projection onto the space $[\mathrm{span}\,\{\mathbf{y}\}]^\bot$ 
Finally, let us define the index sets
\begin{align*}
\mathcal{A}=\{k:z_{k}=0\},\quad\mathcal{I}=\{k:z_{k}\neq0\}
\end{align*}
and $r_{+}=\min_{k\in\mathcal{I}}r_{k}>0$. Notice that we have $\tau_{k}=t_{k}^{2}\geq0$ for all $k\in\mathcal{A}$. 

We first show that the \L ojasiewicz inequality holds with $\theta=\frac{1}{4}$ at those stationary points where $H=0$ (we also refer to the remark after this lemma).
\begin{lemma} \label{lemma:kl-slim}
Suppose $\mathbf{z}$ is an arbitrary point on $\mathbb{CS}^{n-1}$. Then, the inequality
\begin{align}
\lVert\tau\rVert^{\frac{3}{2}}\leq\eta_{\mathbf{z}}\lVert P_{\mathbf{y}}^{\perp}\diag(\tau)\mathbf{y}\rVert,\quad\forall~\mathbf{y}\in B(\mathbf{z},\delta_{\mathbf{z}})\cap\mathbb{CS}^{n-1}
\label{eqn:6-5}
\end{align}
holds for some constants $\eta_{\mathbf{z}},\delta_{\mathbf{z}}>0$. 
\label{lem:6-2}
\end{lemma}
\begin{proof}
In the case $\tau=0$, we have $\lVert\tau\rVert = \lVert P_{\mathbf{y}}^{\perp}\diag(\tau)\mathbf{y}\rVert = 0$ and consequently, the inequality \eqref{eqn:6-5} holds trivially with $\theta=\frac{1}{2}$. Next, let us assume $\mathcal{A}\neq\emptyset$ and let us introduce the polar coordinates $z_k =  r_k e^{i \theta_k}$, $y_k = t_k e^{i \phi_k}$ for $r_k, t_k \geq 0$, $\theta_k, \phi_k \in [0,2\pi]$, and all $k \in [n]$. A straightforward calculation yields $|\Delta_k|^2 = t_k^2 - 2 r_k t_k \cos(\theta_k - \phi_k) + r_k^2 \geq (t_k -r_k)^2$ and hence, it follows
\[  |\tau_k| = (t_k + r_k) |t_k - r_k| \leq (2r_k + |\Delta_k|)|\Delta_k|, \quad \forall~k \in \cI, \] 
%
and $\tau_k = t_k^2 = |\Delta_k|^2$ for all $k \in \cA$ . 
%
%
%
Using $\sum_{k \in \mathcal I} \tau_k = - \sum_{k \in \mathcal A} \tau_k = - \|\tau_\cA\|_1$, the estimates $\|\tau_\cA\|_1^2 \leq |\cA|\|\tau_\cA\|^2$ and
\begin{align} \nonumber \frac12 \sum_{k,j \in \mathcal I}(\tau_k- \tau_j)^2 & = |\mathcal I| \sum_{k \in \mathcal I} \tau_k^2 - \left[ {\sum}_{k \in \mathcal I} \tau_k \right]^2 \\ \label{eq:diff-norm}& = |\mathcal I| \|\tau_\cI\|^2 - \left[ {\sum}_{k \in \mathcal A} \tau_k \right]^2 =  |\mathcal I| \|\tau_\cI\|^2 - \|\tau_\cA\|_1^2, \end{align}
and setting $m := |\mathcal I | \geq 1$, we obtain 
\[  \|\tau\|^2  \leq \frac{1}{2m} \sum_{k,j \in \mathcal I}(\tau_k - \tau_j)^2 + \frac{n}{m} \|\tau_\cA\|^2 \leq \frac{n}{m} \left [ {\sum}_{k,j \in \cI} |\tau_k - \tau_j|^2 + \|\tau_\cA\|^2 \right]. \] 
Furthermore, it holds that 
\begingroup
\allowdisplaybreaks
\begin{align*}  2 \| P_{\bf y}^\bot\diag(\tau){\bf y}\|^2  & = 2 {\bf y}^* \diag(|\tau|^2){\bf y} - 2 ({\bf y}^*\diag(\tau){\bf y})^2  \\ &=  2 \left[ \sum t_k^2 \tau_k^2 - \left( \sum t_k^2 \tau_i \right)^2 \right]   = \sum_{k,j = 1}^n t_k^2 t_j^2 (\tau_k - \tau_j)^2 \\ & =  \sum_{k,j \in \mathcal A} \tau_k \tau_j (\tau_k - \tau_j)^2 +  \sum_{k,j \in \mathcal I} t_k^2 t_j^2 (\tau_k - \tau_j)^2 \\ & \hspace{4ex} + 2 \sum_{k \in \mathcal A} \sum_{j \in \mathcal I} \tau_k t_j^2 (\tau_k - \tau_j)^2. \end{align*}
\endgroup
Next, defining 
$ \delta_{\bf z} := \min \left\{1,\min\{r_+^2,1\}/6 \right\}$, it follows $\|\tau_\cI\|_\infty \leq \min\{r_+^2,1\} / 2$ for all ${\bf y} \in B_{\delta_{\bf z}}({\bf z})$. 
Moreover, the latter condition implies $|y_k|^2 = t_k^2 > {r^2_+}/{2}$ for all $k \in \mathcal I$ and thus, it holds that
\begin{align*} 2 \sum_{k \in \cA} \sum_{j \in \cI} \tau_k t_j^2 (\tau_k - \tau_j)^2 & \geq r_+^2 \sum_{k \in \cA} \sum_{j \in \cI} [\tau_k^3 - 2 \tau_k^2 \tau_j + \tau_k \tau_j^2] \\ & = r_+^2  \sum_{k \in \cA} \left [ m \tau_k^3 - 2 \tau_k^2 \left( {\sum}_{j \in \cI} \tau_j \right) + \tau_k \|\tau_\cI\|^2 \right] \\ & = r_+^2 m \|\tau_\cA\|_3^3 + r_+^2 \|\tau_\cA\|_1 [ \|\tau\|^2 + \|\tau_\cA\|^2] \geq r_+^2 m \|\tau_\cA\|_3^3. \end{align*}
Similarly, we can show $\sum_{k,j \in \cA} \tau_k\tau_j(\tau_k-\tau_j)^2 = 2 \|\tau_\cA\|_1 \|\tau_\cA\|_3^3 - 2 \|\tau_\cA\|^4 \geq 0$. Due to $\|\tau_\cI\|_\infty \leq 1/2$, we have $|\tau_k - \tau_j| \leq 1$ for all $k,j \in \cI$ and consequently, it follows $\sum_{k,j \in \cI} (\tau_k - \tau_j)^2 \geq  \sum_{k,j \in \cI} |\tau_k - \tau_j|^3$. Together and using the estimate $\|{\bf w}\|^3 \leq \sqrt{p} \|{\bf w}\|_3^3$, ${\bf w} \in \C^p$, we finally obtain
\begin{align*} \| P_{\bf y}^\bot\diag(\tau){\bf y}\|^2  & \geq \frac{r_+^2}{2} \min \left\{\frac{r_+^2}{4},m\right\} \left[ {\sum}_{k,j \in \mathcal I} |\tau_k - \tau_j|^3 + \|\tau_\cA\|_3^3 \right] \\ & \geq \underbracket{\begin{minipage}[t][5.5ex][t]{20ex}\centering$\displaystyle\frac{r^2_+\min \left\{{r_+^2}/{4},m\right\}}{2\sqrt{m^2 + n - m}} $ \end{minipage}}_{=: c} \left[ {\sum}_{k,j \in \mathcal I} |\tau_k - \tau_j|^2 + \|\tau_\cA\|^2 \right]^{\frac32} . \end{align*} 
Thus, the {\L}ojasiewicz-type inequality \eqref{eqn:6-5} is satisfied with $\eta_{\bf z} := (\frac{n}{m})^{1.5}c^{-1}$. 
\end{proof}

\begin{remark}
If $\mathbf{z}\in\mathbb{CS}^{n-1}$ is a stationary point of problem \eqref{eqn:obj} with $H=0$, then we have $f({\bf y}) - f({\bf z}) = (\beta/2) \|\tau\|^2$ and $\|\grad{f({\bf z})}\| = \sqrt{2}\beta \|P_{{\bf y}}^\bot \diag(\tau){\bf y}\|$ and thus, Lemma \ref{lemma:kl-slim} implies that the \L ojasiewicz inequality \eqref{eq:kl} holds with $\theta=\frac{1}{4}$. Moreover, this result can be used to show that the KL exponent can not be larger than $\frac14$ for general stationary points. 
%
%
\end{remark}
\begin{remark} The proof of Lemma \ref{lemma:kl-slim} implies that the exponent $\frac32$ can be improved to $1$ if the index set $\mathcal A$ is empty. More generally, it can be shown that the \L ojasiewicz-type inequality \eqref{eqn:6-5} holds with exponent $1$ along directions ${\bf y} \in B({\bf z},\delta_{\bf z}) \cap \CSn$ with $y_\cA = 0$.
\end{remark}

Next, we prove that the KL exponent is at least $\frac{1}{4}$ when the matrix $A$ is diagonal.
\begin{theorem}
Let $A = \diag({\bf a}) \in \C^{n \times n}$, ${\bf a} \in \Rn$, be a diagonal matrix. Then, the KL exponent of problem \eqref{eqn:obj} is $\frac{1}{4}$.
\end{theorem}
\begin{proof}

In the case $\tau=0$, we have $f({\bf y}) - f({\bf z}) = \half {\bf y}^* H {\bf y}$ and $2 \|\grad f({\bf y})\|^2 = {\bf y}^*H^2 {\bf y} - ({\bf y}^*H{\bf y})^2$. We now decompose ${\bf y}$ as follows 
\[ {\bf y} = {\bf u} + {\bf v}, \quad H{\bf u} = 0, \quad \|H{\bf v}\| \geq \sigma_{-}(H) \|{\bf v}\|, \quad \text{and} \quad {\bf u}^*{\bf v} = 0, \]
where $\sigma_{-}(H)$ denotes smallest positive singular value of $H$. Furthermore, let $\sigma_{+}(H) \geq 0$ be the maximum singular value of  $H$. It holds that
\[ f({\bf y}) - f({\bf z}) \leq \frac{\sigma_{+}(H) }{2} \|{\bf v}\|^2, \quad \|\grad f({\bf y})\|^2 \geq \frac{\sigma_-(H)^2}{2} \|{\bf v}\|^2 - \frac{\sigma_{+}(H)^2}{2} \|{\bf v}\|^4. \]
Thus, due to $\|{\bf v}\|^2 \leq \|\Delta\|^2$, the inequality \eqref{eq:kl} is satisfied with exponent $\theta = \frac12$.

Next, we consider the general case $\tau \neq 0$. In this situation, we have $A_{[\cA\cI]} = 0$, $H_{[\cI\cA]} = A_{[\cI\cA]} = 0$ and the KKT conditions imply 
\[ a_k + 2\beta |z_k|^2 - 2\lambda = 0 \quad \forall~k \in \cI \quad \implies \quad H_{[\cI\cI]} = 0 \]
 and hence, due to $\tau_k = |\Delta_k|^2 = |y_k|^2 = t_k^2$ for all $k \in \cA$, it follows 
 \begin{align*} {\bf y}^*\diag(\tau)H{\bf y} & = y_\cA^* \diag(\tau_\cA) H_{[\cA\cdot]}\Delta + y_\cI^* \diag(\tau_\cI)H_{[\cI\cdot]}\Delta \\ & = \Delta_\cA^* \diag(\tau_\cA) A_{[\cA\cA]} \Delta_\cA - 2\lambda \|\tau_\cA\|^2 = \sum_{k \in \cA} (a_k - 2\lambda) \tau_k^2.  \end{align*}
Using Young's inequality, $t_k^2 \leq 1$, and  $\tau_k = t_k^2$ for all $k \in \cA$, it follows
\[ {\bf y}^*H{\bf y} \cdot {\bf y}^*\diag(\tau){\bf y} = {\bf y}^*H{\bf y} \left [\|\tau_\cA\|^2 + \sum_{k \in \cI} t_k^2 \tau_k\right] \leq \frac{2|{\bf y}^* H {\bf y}|^q }{q} +  \frac{\|\tau_\cA\|^{2p} + \|\tau_\cI\|_1^p}{p}  \]
for $p > 1$ and $q = 1 + \frac{1}{p-1}$. 
%
%
%
%
%
%
%
 %
 %
 %
Let us now introduce the index set $\mathcal B := \{k \in \mathcal A: a_k - 2\lambda \neq 0\}$ and let us define $h_- := \min_{k \in \mathcal B} |a_k - 2\lambda|$ and $h_+ := \max_{k \in \mathcal B} |a_k - 2\lambda|$. Then, due to $\|\tau_\cB\|^2 = \|y_\cB\|^4_4$ and \eqref{eqn:6-4}, and applying the estimates derived in the proof of Lemma \ref{lemma:kl-slim}, we obtain  
\begin{align*} \|\grad f({\bf y})\|^2 & \\ & \hspace{-8ex} \geq \frac{h_-^2}{2} \|y_\cB\|^2 - \frac{h_+^2}{2} \|y_\cB\|^4 - \frac{4\beta}{q} h_+^q \|y_\cB\|^{2q} - \frac{2\beta}{p} [\|\tau_\cA\|^{2p} + \|\tau_\cI\|_1^p] \\ & \hspace{-5ex} - 2\beta h_+ \|\tau_\cB\|^2 + 2\beta^2 {\bf y}^* \diag(|\tau|^2){\bf  y} - 2\beta^2 ({\bf y}^*\diag(\tau){\bf y})^2 \\ & \hspace{-8ex} \geq \frac{h_-^2}{2} \|y_\cB\|^2 - o(\|y_\cB\|^2) + \frac{\beta^2r_+^2}{4} {\sum}_{k,j \in \cI} |\tau_k - \tau_j|^2 + \beta^2 r_+ m \|\tau_\cA\|_3^3 \\ & \hspace{-5ex} - 2\beta p^{-1} [\|\tau_\cA\|^{2p} + \|\tau_\cI\|_1^p]  \end{align*}
for $\|y_\cB\| \to 0$, $m := |\cI|$, and ${\bf y}$ sufficiently close to ${\bf z}$.  
%
%
By \eqref{eq:diff-norm}, we have
\begin{align*} \|\tau_\cI\|_1^4 \leq m^2 \|\tau_\cI\|^4 & \leq \left[ \half {\sum}_{k,j \in \cI} |\tau_k - \tau_j|^2 + |\cA| \|\tau_\cA\|^2 \right]^2 \\ & \leq \left[{\sum}_{k,j \in \cI} |\tau_k - \tau_j|^2\right ]^2 + 2(n-m)^2 \|\tau_\cA\|^4 \end{align*}
and consequently, setting $p = 4$, we can choose $\|\Delta\|$ (and thus $\|\tau\|$ and $\|y_\cB\|$) sufficiently small, such that
\begin{align*}\|\grad f({\bf y})\|^2 & \geq \eta_1 \left [ \|y_\cB\|^2 + {\sum}_{k,j \in \cI} |\tau_k - \tau_j|^2 + \|\tau_\cA\|_3^3 \right] \\ & \geq \eta_2 \left [ \|y_\cB\|^2 + {\sum}_{k,j \in \cI} |\tau_k - \tau_j|^2 + \|\tau_\cA\|^2 \right]^{\frac32} \end{align*}
and $f({\bf y}) - f({\bf z}) \leq \eta_3 [ \|y_\cB\|^2 + {\sum}_{k,j \in \cI} |\tau_k - \tau_j|^2 + \|\tau_\cA\|^2 ]$ for suitable $\eta_1, \eta_2, \eta_3 > 0$. This shows that the {\L}ojasiewicz inequality is satisfied with $\theta = \frac14$. 
\end{proof}

Finally, we derive the KL exponent in the real case for global minimizers characterized by the positive semidefiniteness condition in Theorem \ref{theorem:glob-suff}.

\begin{theorem}
Suppose $A \in \R^{n\times n}$ is a symmetric matrix and $\mathbf{z}$ is a stationary point of problem \eqref{eqn:obj} satisfying
\begin{align*}
H \succeq0,
\end{align*}
where $H$ is defined in \eqref{eq:glob-suff}. Then, the KL exponent of problem \eqref{eqn:obj} at $\mathbf{z}$ is at least $\frac{1}{4}$.
\end{theorem}


\begin{proof}
Without loss of generality we assume $\beta=1$. Let ${\bf y} \in \Sn$ be arbitrary and let us set $\Delta = {\bf y} - {\bf z}$, and 
\begin{align*}
&\gamma_{1} = \sum_{k\in \cI}z_{k}^{3}\Delta_{k},\quad \gamma_{2} = \sum_{k\in\cI}  z_{k}^2\Delta_{k}^2, \quad \gamma_{3}=\sum_{k\in \cI} z_{k}\Delta_{k}^{3},\quad \gamma_{4}= \|\Delta\|_4^4.
\end{align*}
Based on the representation $\grad{f({\bf y})} = P_{\bf y}^\bot [H + 2 \diag(\tau)]{\bf y}$, we now introduce the following decompositions
\begin{align*}
& 2\diag(\tau){\bf y} = {\bf w} + c_{1} {\bf y}, \quad {\bf w} = 2 P_{\bf y}^\bot \diag(\tau){\bf y} , \quad c_{1} = 2 {\bf y}^T \diag(\tau){\bf y} \\
& {\Delta} = {\bf u} + {\bf v}, \quad H{\bf u} = 0, \quad {\bf u}^T {\bf v} = 0, \quad \|H{\bf v}\| \geq \sigma_-(H)\|{\bf v}\|,
\end{align*}
where $\sigma_-(H)$ denotes the smallest positive singular value of $H$. Using the latter decomposition and $H{\bf y} = H\Delta$, we can express the norm of the Riemannian gradient as follows 
\begin{equation} \label{eq:grad-calc-kl} \|\grad{f({\bf y})}\|^2 = \|H{\Delta} + {\bf w}\|^2 - ({\Delta}^T H{\Delta})^2 =  \|H{\bf v} + {\bf w}\|^2 - ({\bf v}^T H{\bf v})^2 \end{equation}  
Let $\lambda_+(H)$ be the largest eigenvalue of $H$. Then, by definition of ${\bf v}$, we obtain
%
\begin{align}
|{\bf v}^{T}H{\bf v}| & \leq \lambda_+(H) \|{\bf v}\|^2 \leq \lambda_+(H){\sigma_{-}(H)}^{-2} \|H{\bf v}\|^2 =: \bar \sigma^{-1} \cdot {\bf v}^{T} H^{2} {\bf v}. 
\label{eqn:6-10}
\end{align}
%
Moreover, Lemma \ref{lem:6-2} yields
\begin{align}
\|\tau\|^\frac32 \leq \eta_1 \|{\bf w}\| 
\label{eqn:6-11}
\end{align}
for some constant $\eta_{1}>0$ and for all ${\bf y} \in \Sn$ sufficiently close to ${\bf z}$. Let $\epsilon > 0$ be an arbitrary small positive constant. We now discuss three different cases.

\textit{Case 1.} $\lVert\mathbf{w}\rVert \geq (1+\epsilon) \lVert H{\bf v}\rVert$ or $\lVert\mathbf{w}\rVert \leq (1-\epsilon)\lVert H{\bf v}\rVert$. In this case, we have
\begin{align*}
\lVert H{\bf v} + {\bf w}\rVert^{2} \geq \frac{\epsilon^{2}}{2} \left(\lVert{\bf w}\rVert^{2}+\lVert H{\bf v}\rVert^{2} \right).
\end{align*}
Let us set $\delta_{\bf z} > 0$ sufficiently small such that $\|{\bf v}\|^2 \leq \epsilon^2 \bar \sigma / (4\lambda_+(H))$ and $|{\bf v}^TH{\bf v}| \leq 1$. Using \eqref{eqn:6-10}, \eqref{eqn:6-11}, and the estimates $({\bf v}^{T}H{\bf v})^{2}\leq\lambda_{+}^{2}(H)\lVert{\bf v}\rVert^{4}$ and $[\frac 12 |a + b|]^{3/2} \leq [|a|^{3/2} + |b|^{3/2}] / \sqrt{2}$, $a,b \in \R$, it follows
\begin{align*}
\lVert\grad{f(\mathbf{y})}\rVert^{2}&\geq\frac{\epsilon^{2}}{2}(\lVert\mathbf{w}\rVert^{2}+\lVert H{\bf v}\rVert^{2})-\lambda_{+}^{2}(H)\lVert{\bf v}\rVert^{4}\geq\frac{\epsilon^{2}}{4}\left(2\lVert\mathbf{w}\rVert^{2}+\lVert H{\bf v}\rVert^{2}\right)\\
&\geq \frac{\epsilon^2}{2} \min\left \{\frac{1}{\eta_{1}},\frac{\bar\sigma}{2} \right \} \left[ \|\tau\|^3 +  |{\bf v}^T H {\bf v}|^\frac32 \right ]\\
&\geq \frac{\epsilon^2\sqrt{2}}{2} \min\left \{\frac{1}{\eta_{1}},\frac{\bar\sigma}{2} \right \}  \cdot | f({\bf y}) - f({\bf z}) |^\frac32.
\end{align*}
Thus, we can infer that the KL exponent of problem \eqref{eqn:obj} at ${\bf z}$ is $\frac{1}{4}$.
%

\textit{Case 2.} $(2-{\epsilon})r_{+}^{2}\lVert\Delta_{\mathcal{I}}\rVert\geq\lVert\Delta\rVert^2$ or $\gamma_1 \leq (2-\epsilon)r_+^2 \|\Delta_\cI\|^2 \|\Delta\|^{-2}$. 
First, due to $\Delta^T {\bf y} = \half \|\Delta\|^2$, we have $P_{\mathbf{y}}^{\perp}\Delta=\Delta-\frac{1}{2}\lVert\Delta\rVert_{2}^{2}\mathbf{y}$. Defining $t_\Delta := \Delta^T P_{\mathbf{y}}^{\perp}\Delta = \|\Delta\|^2 - \frac14 \|\Delta\|^4$, we will work with the following decompositions
\begin{align*}
& H\Delta = \Delta^{T}H\Delta \cdot \mathbf{y}+c_{2} P_{\mathbf{y}}^{\perp}\Delta+\mathbf{w}_{1}, \quad c_{2}=\frac{1-\frac{1}{2}\lVert\Delta\rVert^{2}}{t_\Delta}\Delta^{T}H\Delta,\\
& \mathbf{w}=c_{3} P_{\mathbf{y}}^{\perp}\Delta+\mathbf{w}_{2}, \quad c_{3}= \frac{1}{t_\Delta}[2\Delta^T\diag(\tau){\bf y} - \|\Delta\|^2 {\bf y}^T\diag(\tau){\bf y}] 
\end{align*}
where $\mathbf{w}_{1}$ and $\mathbf{w}_{2}$ are vectors orthogonal to $\mathbf{y}$ and $\Delta$. Hence, by \eqref{eq:grad-calc-kl}, it holds that
\begin{equation*}
\lVert\grad{f(\mathbf{y})}\rVert^{2} = (c_{2}+c_{3})^{2} \cdot t_\Delta +\lVert\mathbf{w}_{1}+\mathbf{w}_{2}\rVert^{2} \geq (c_{2}+c_{3})^{2} \cdot t_\Delta.
\end{equation*}
Using the definitions introduced in the first case, we can express $t_\Delta \cdot c_{3}$ via
\[ t_\Delta \cdot c_{3} = (2-\|\Delta\|^2) \gamma_4 + (6 - 4 \|\Delta\|^2) \gamma_3 + (4-5\|\Delta\|^2) \gamma_2 - 2 \|\Delta\|^2 \gamma_1. \]
Now, if $(2-{\epsilon})r_{+}^{2}\lVert\Delta_{\mathcal{I}}\rVert \geq\lVert\Delta\rVert^{2}$, we readily obtain $(4-5\|\Delta\|^2) \gamma_2 - 2\|\Delta\|^2 \gamma_1 \geq (2\epsilon - 5\|\Delta\|^2) r_+^2 \|\Delta_\cI\|^2$ and thus, it follows
\begin{equation} \label{eq:low-tdelta} t_\Delta \cdot c_{3} \geq \gamma_4 + \epsilon r_+^2 \|\Delta_\cI\|^2 + o(\|\Delta_\cI\|^2) \geq \|\Delta\|_4^4 + \frac{\epsilon r_+^2}{2} \|\Delta_\cI\|^2 \end{equation}
for $\Delta$ sufficiently small. 
Otherwise, if $\gamma_1 \leq (2-\epsilon)r_+^2 \|\Delta_\cI\|^2 \|\Delta\|^{-2}$, 
then we also have $(4-5\|\Delta\|^2) \gamma_2 - 2 \|\Delta\|^2 \gamma_1 \geq (2\epsilon - 5\|\Delta\|^2) r_+^2 \|\Delta_\cI\|^2$ and thus, \eqref{eq:low-tdelta} holds in both sub-cases. 
Consequently, due to the positive semidefiniteness of $H$ and $\|\Delta\|^2 = - 2 \Delta^T {\bf z} = -2\Delta_\cI^T z_\cI \leq 2 \|\Delta_\cI\|$, we can infer
\begin{align*}
\lVert\grad{f(\mathbf{y})}\rVert & \geq|c_{2}+c_{3}| \sqrt{t_\Delta} = \left[\left(1-\half \|\Delta\|^2\right) \Delta^{T}H\Delta + t_\Delta c_{3}\right] t_\Delta^{-1/2} \\
& \geq [\Delta^{T}H\Delta + 2 \|\Delta\|_4^4 +\epsilon r_{+}^{2}\lVert\Delta_{\mathcal{I}}\rVert^{2}] \cdot (2\|\Delta\|)^{-1} \\
& \geq  [\Delta^{T}H\Delta + 2 \|\Delta\|_4^4 +\epsilon r_{+}^{2}\lVert\Delta_{\mathcal{I}}\rVert^{2}]^{\frac34} \cdot \frac{\epsilon^{\frac14}\sqrt{r_{+}\lVert\Delta_{\mathcal{I}}\rVert}}{2\|\Delta\|} \\ 
& \geq \eta_2 [\Delta^{T}H\Delta + 2 \|\Delta\|_4^4 +\epsilon r_{+}^{2}\lVert\Delta_{\mathcal{I}}\rVert^{2}]^{\frac34}, 
\end{align*}
where $\eta_2 := \epsilon^{\frac14}\sqrt{r_{+}} / 2\sqrt{2}$.
Next, utilizing \eqref{eqn:6-2} and $|\tau_k| \leq 2 |\Delta_k|$ for all $k \in \cI$, we finally obtain
\begin{align*}
\lvert f(\mathbf{y})-f(\mathbf{z})\rvert&=\frac{1}{2}[ \Delta^{T}H\Delta+ \|\tau\|^2]  \\ & \leq \frac{1}{2}\left[\Delta^{T}H\Delta+ \|\Delta_\cA\|_4^4 + 4\lVert\Delta_{\mathcal{I}}\rVert^{2} \right] \leq \eta_{3} \lVert\grad{f(\mathbf{y})}\rVert^{\frac43}
\end{align*}
for some constant $\eta_{3}>0$ and for all ${\bf y}$ sufficiently close to ${\bf z}$. Hence, the KL exponent is $\frac{1}{4}$ in this case.

\textit{Case 3.} $(1-\epsilon)\lVert H{\bf v}\rVert\leq\lVert\mathbf{w}\rVert\leq(1+\epsilon)\lVert H{\bf v}\rVert$, $\gamma_{1}\geq(2-\epsilon)r_{+}^{2}\lVert\Delta_{\mathcal{I}}\rVert^{2}\|\Delta\|^{-2}$ and $\left(2-\frac{\epsilon}{2}\right)r_{+}^{2}\lVert\Delta_{\mathcal{I}}\rVert\leq\lVert\Delta\rVert^{2}\leq2\lVert\Delta_{\mathcal{I}}\rVert$. In this case, we have
\begin{align}
\gamma_{1} = \Theta(\|\Delta\|^2) ,\quad \gamma_{2} = \Theta( \|\Delta\|^4), \quad \gamma_{3}=O(\|\Delta\|^{6}),\quad \gamma_{4}=\Theta(\|\Delta\|^{4}).
\label{eq:order-gamma}
\end{align}
Let us set $m := |\cI|$ and define $\sigma_{k}= \Delta_{k} z_{k}+\frac{1}{2m}\|\Delta\|^{2}$ for all $k\in \cI$ and $\sigma_k = 0$ for all $k \in \cA$. 
Then, it follows 
\begin{align*}
\sum_{k \in [n]} \sigma_k = 0, \quad \|\sigma\|_1 \geq \sum_{k\in \cI} z_{k}^{2}\sigma_{k} = \gamma_1 + \frac12 \|\Delta\|^2,
\end{align*}
and $\|\sigma\| = \Theta(\|\Delta\|^2)$.
We now express $\|\mathbf{w}\|^{2}$ in terms of $\gamma_{1}$, $\gamma_2$, $\gamma_3$, and $\gamma_{4}$. Specifically, by utilizing \eqref{eq:order-gamma}, we obtain
\begin{align*}
\frac14 \|\mathbf{w}\|^{2} & = {\bf y}^T \diag(|\tau|^2){\bf y} - ({\bf y}\diag(\tau){\bf y})^2 \\ & = \sum_{k \in [n]} \Delta_k^2 [\Delta_k^4 + 6\Delta_k^3z_k + 13\Delta_k^2z_k^2 + 12 \Delta_kz_k^3 + 4z_k^4] - 4\gamma_1^2  - 20\gamma_1\gamma_2 \\
& \hspace{2ex} - 25\gamma_2^2 - 8\gamma_3(5\gamma_2+2\gamma_1) - 16 \gamma_3^2 - 2 \|\Delta\|_4^4(4\gamma_3+5\gamma_2+2\gamma_1) - \|\Delta\|_4^8 \\ & = 4 \sum_{k \in [n]} \Delta_k^2 z_k^4 - 4\gamma_1^2 + O(\|\Delta\|^6) = 2 \sum_{k, j \in \cI} z_k^2 z_j^2 (\sigma_k - \sigma_j)^2 + O(\|\Delta\|^6) \\ & = \Theta(\|\sigma\|^2) + O(\|\Delta\|^6),
\end{align*}
which implies $\|\mathbf{w}\|=\Theta(\|\Delta\|^{2})$ and $\| H{\bf v}\|=\Theta(\|\Delta\|^{2})$. As a consequence, we get ${\bf v}^{T}H{\bf v}=\Theta(\|\Delta\|^{4})$ and
\begin{align}\label{eqn:case3-diff}
2 \lvert f(\mathbf{y})-f(\mathbf{z})\rvert=\lvert{\bf v}^{T}H{\bf v} + \|\Delta\|_4^4 +4\gamma_{3}+4\gamma_{2}\rvert = \Theta(\|\Delta\|^{4}).
\end{align}
For some index sets $\mathcal K, \mathcal J \subset [n]$, let $H_{\mathcal K\mathcal J} \in \R^{|\mathcal K| \times |\mathcal J|}$ denote the submatrix $H_{\mathcal K\mathcal J} = (H_{kj})_{k \in \mathcal K, j \in \mathcal J}$. Due to the positive semidefiniteness of $H$, we have $H_{\mathcal{I}\mathcal{I}}\succeq0$ and $H_{\mathcal{A}\mathcal{A}}\succeq0$. Furthermore, due to \eqref{eq:grad-calc-kl} and \eqref{eq:order-gamma}, we obtain
%
\begin{align}
\nonumber \|\grad{f({\bf y})}\|^2 & = \|H\Delta + {\bf w} \|^2 - \Theta(\|\Delta\|^8) \\ \nonumber
& = \|H_{\cI\cA}^T\Delta_\cI + H_{\cA\cA} \Delta_\cA + 2\diag(|\Delta_\cA|^2)\Delta_\cA - c_1 \Delta_\cA \|^2   \\ \label{eq:grad-split}& \hspace{4ex} + \| H_{\cI\cI} \Delta_\cI + H_{\cI \cA} \Delta_\cA + {w}_\cI\|^2 - \Theta(\|\Delta\|^8).
\end{align}
We define
\begin{align}
\nonumber {\bf g}_1 & := H_{\cI\cI} \Delta_\cI + H_{\cI \cA} \Delta_\cA + {w}_\cI, \\
{\bf g}_2 &:= H_{\cI\cA}^T\Delta_\cI + H_{\cA\cA} \Delta_\cA + 2\diag(|\Delta_\cA|^2)\Delta_\cA - c_1 \Delta_\cA 
\end{align}
%

Next, let $\eta_{4},\eta_{5}>0$ and $\mu \in (0,\frac{1}{2})$ be given constants and suppose $\Delta_{\mathcal{I}}^{T}(H_{\mathcal{I}\mathcal{I}}\Delta_{\mathcal{I}}+H_{\mathcal{I}\mathcal{A}}\Delta_{\mathcal{A}})\geq-\eta_{4}\|\Delta\|^{4+\mu}$. Then, using $\half c_1 = \|\Delta\|_4^4 + 4\gamma_3 + 5\gamma_2 + 2\gamma_1$, $\Delta_\cI^T y_\cI = - \half \|\Delta\|^2$, and $\Delta_\cI^T w_\cI = 2\|\Delta\|_4^4 + 6 \gamma_3 + 4\gamma_2 - c_1 \Delta_\cI^T y_\cI$, it follows
\begin{align*}
\Delta_{\mathcal{I}}^{T}{\bf g}_{1}=\Delta_{\mathcal{I}}^{T}(H_{\mathcal{I}\mathcal{I}}\Delta_{\mathcal{I}}+H_{\mathcal{I}\mathcal{A}}\Delta_{\mathcal{A}})+ \Delta_\cI^T w_\cI  \geq \Theta(\|\Delta\|^{4}).
\end{align*}
%
Similarly, in the case $\Delta_{\mathcal{I}}^{T}(H_{\mathcal{I}\mathcal{I}}\Delta_{\mathcal{I}}+H_{\mathcal{I}\mathcal{A}}\Delta_{\mathcal{A}})\leq-\eta_{5}\|\Delta\|^{4-\mu}$ and if $\|\Delta\|$ is sufficiently small, we get
\[ \Delta_{\mathcal{I}}^{T}{\bf g}_{1} \leq -\frac{\eta_{5}}{2}\|\Delta\|^{4-\mu}, \]
Combining both cases, we can infer $\|\Delta_{\mathcal{I}}\|\|{\bf g}_{1}\|\geq|\Delta_{\mathcal{I}}^{T}{\bf g}_{1}| \geq \eta_6 \|\Delta\|^4$ for some $\eta_6 > 0$ and for all ${\bf y}$ sufficiently close to ${\bf z}$. This implies $\|{\bf g}_{1}\| \geq [(2-\frac{\epsilon}{2})r_+^2]^{-1}\eta_6 \|\Delta\|^2$ and hence, by \eqref{eq:grad-split}, we obtain $\|\grad f({\bf y})\| \geq \Theta(\|\Delta\|^2)$. 
Considering equation \eqref{eqn:case3-diff}, the KL exponent has to be $\theta = \frac{1}{2}$ in these two cases.

Now, let us suppose $-\eta_{5}\|\Delta\|^{4-\mu}\leq\Delta_{\mathcal{I}}^{T}(H_{\mathcal{I}\mathcal{I}}\Delta_{\mathcal{I}}+H_{\mathcal{I}\mathcal{A}}\Delta_{\mathcal{A}})\leq-\eta_{4}\|\Delta\|^{4+\mu}$. Due to $\Delta_{\mathcal{I}}^{T}H_{\mathcal{I}\mathcal{I}}\Delta_{\mathcal{I}}\geq0$, this directly yields $-\eta_{5}\|\Delta\|^{4-\mu}\leq\Delta_{\mathcal{I}}^{T}H_{\mathcal{I}\mathcal{A}}\Delta_{\mathcal{A}}<0$ and 
\begin{align*}
\Delta_{\mathcal{A}}^{T}H_{\mathcal{A}\mathcal{A}}\Delta_{\mathcal{A}}\geq\Delta_{\mathcal{A}}^{T}(H_{\mathcal{I}\mathcal{A}}^{T}\Delta_{\mathcal{I}}+H_{\mathcal{A}\mathcal{A}}\Delta_{\mathcal{A}})\geq\Delta^{T}H\Delta=\Theta(\|\Delta\|^{4}).
\end{align*}
If we have $\Delta_{\mathcal{A}}^{T}H_{\mathcal{A}\mathcal{A}}\Delta_{\mathcal{A}}\geq\eta_{7}\|\Delta\|^{4-2\mu}$ for some constant $\eta_{7}>0$, then it holds that
\begin{align*}
\Delta_{\mathcal{A}}^{T}{\bf g}_{2}&=\Delta_{\mathcal{A}}^{T}(H_{\mathcal{I}\mathcal{A}}^{T}\Delta_{\mathcal{I}}+H_{\mathcal{A}\mathcal{A}}\Delta_{\mathcal{A}})+2\|\Delta_{\mathcal{A}}\|_{4}^{4}-c_1\|\Delta_{\mathcal{A}}\|^{2}\\
&\geq\eta_{7}\|\Delta\|^{4-2\mu}-\eta_{5}\|\Delta\|^{4-\mu}+O(\|\Delta\|^{4})\geq\frac{\eta_{7}}{2}\|\Delta\|^{4-2\mu}
\end{align*}
if $\|\Delta\|$ is sufficiently close to zero. As before, this estimate can be utilized to show  $\|\grad{f({\bf y})}\|  \geq \Theta(\|\Delta\|^{3-2\mu})$ and consequently, the KL exponent is $\frac{1+2\mu}{4}$. 

Finally, we consider $\eta_{8}\|\Delta\|^{4}\leq\Delta_{\mathcal{A}}^{T}H_{\mathcal{A}\mathcal{A}}\Delta_{\mathcal{A}}\leq\eta_{7}\|\Delta\|^{4-2\mu}$, where $\eta_{8} > 0$ is chosen such that $\Delta^{T}H\Delta\geq\eta_{8}\|\Delta\|^{4}$. Let us define the decompositions
\begin{align*}
\Delta_{\mathcal{A}}=\psi_{1}+\xi_{1},\quad\psi_{1}\in\mathrm{null}~H_{\mathcal{A}\mathcal{A}},\quad\xi_{1} \in [\mathrm{null}~H_{\mathcal{A}\mathcal{A}}]^\bot ,\\
\diag(|\Delta_\cA|^2)\Delta_{\mathcal{A}}=\psi_{2}+\xi_{2},\quad\psi_{2}\in\mathrm{null}~H_{\mathcal{A}\mathcal{A}}, \quad \xi_{2} \in [\mathrm{null}~H_{\mathcal{A}\mathcal{A}}]^\bot,
\end{align*}
where $\mathrm{null}~M$ is the null space of matrix $M$. We then have $\|\xi_{1}\|=\Theta(\|H_{\mathcal{A}\mathcal{A}}\xi_{1}\|)=\Theta(\sqrt{\xi_{1}^{T}H_{\mathcal{A}\mathcal{A}}\xi_{1}})=O(\|\Delta\|^{2-\mu})$ and $\lVert\psi_{1}\rVert=O(\|\Delta\|)$. Notice that such decompositions exist due to the symmetry of $H_{\cA\cA}$.

Since $H$ is positive semidefinite, we can show that $\mathrm{null}\,H_{\mathcal{A}\mathcal{A}}\subset\mathrm{null}\,H_{\mathcal{I}\mathcal{A}}$. If $H_{\mathcal{I}\mathcal{A}}=0$, then this claim is certainly true. Otherwise, if we assume that the statement is false, the set $S = \mathrm{null}\,H_{\mathcal{A}\mathcal{A}}\cap [\mathrm{null}\,H_{\mathcal{I}\mathcal{A}}]^{\bot}$ is nonempty and there exists $\psi \in S$ and $\xi \in [\mathrm{null}\,H_{\mathcal{I}\mathcal{A}}^{T}]^{\bot}$. Then it holds
\begin{align*}
\begin{bmatrix}\xi \\ \nu\psi \end{bmatrix} \begin{bmatrix} H_{\cI\cI} & H_{\cI\cA} \\ H_{\cI\cA}^T & H_{\cA\cA} \end{bmatrix} \begin{bmatrix}\xi &\nu\psi\end{bmatrix}=\xi^{T}H_{\mathcal{I}\mathcal{I}}\xi +2\nu \psi^{T}H_{\mathcal{I}\mathcal{A}}^{T}\xi \geq0,\quad\forall~\nu\in\mathbb{R}.
\end{align*}
But since $\psi^{T}H_{\mathcal{I}\mathcal{A}}^{T}\xi \neq0$, we can choose $\nu$ such that $\xi^{T}H_{\mathcal{I}\mathcal{I}}\xi+2\nu\psi^{T}H_{\mathcal{I}\mathcal{A}}^{T}\xi<0$, which is a contradiction. Hence, due to $H_{\mathcal{I}\mathcal{A}}^{T}\Delta_{\mathcal{I}} \in \mathrm{ran}\,H_{\cI\cA}^T = [\mathrm{null}\,H_{\mathcal{I}\mathcal{A}}]^\bot$, we can infer $H_{\mathcal{I}\mathcal{A}}^{T}\Delta_{\mathcal{I}} + H_{\mathcal{A}\mathcal{A}}\Delta_{\mathcal{A}}\in [\mathrm{null}\,H_{\mathcal{A}\mathcal{A}}]^{\perp}$. Consequently, ${\bf g}_2$ can be written as ${\bf g}_2 = {\bf g}_3 + {\bf d}$ where ${\bf g}_3  \in [\mathrm{null}\,H_{\mathcal{A}\mathcal{A}}]^{\perp}$ and ${\bf d}=2\psi_{2}-c_{1}\psi_{1} \in \mathrm{null}\,H_{\mathcal{A}\mathcal{A}}$. 
If $\|{\bf d}\|\geq\frac{\eta_{8}}{2}\|\Delta\|^{3}$, then we obtain $\|\grad{f({\bf y})}\|\geq\|{\bf g}_{2}\|\geq\|{\bf d}\|\geq\frac{\eta_{8}}{2}\|\Delta\|^{3}$ and, by \eqref{eqn:case3-diff}, the KL exponent is $\frac{1}{4}$. Otherwise, if $\|{\bf d}\|\leq\frac{\eta_{8}}{2}\|\Delta\|^{3}$, it follows
\begin{align*}
2\lVert\Delta_{\mathcal{A}}\rVert_{4}^{4}-c_{1}\lVert\Delta_{\mathcal{A}}\rVert^{2} 
&= 2\psi_{1}^{T}\psi_{2}-c_{1}\lVert\psi_{1}\rVert^{2}+2\xi_{1}^{T}\xi_{2}-c_{1}\lVert\xi_{1}\rVert^{2}\\
&=\psi_{1}^{T}\mathbf{d}+2\xi_{1}^{T}\xi_{2}-c_{1}\lVert\xi_{1}\rVert^{2}
\geq-\frac{\eta_{8}}{2}\|\Delta\|^{4}+O(\lVert\Delta\rVert^{5-\mu})
\end{align*}
and
\begin{align*}
\Delta_{\mathcal{A}}^{T}{\bf g}_{2}&=\Delta_{\mathcal{A}}^{T}(H_{\mathcal{I}\mathcal{A}}^{T}\Delta_{\mathcal{I}}+H_{\mathcal{A}\mathcal{A}}\Delta_{\mathcal{A}})+2\lVert\Delta_{\mathcal{A}}\rVert_{4}^{4}-c_{1}\lVert\Delta_{\mathcal{A}}\rVert^{2}\\
&\geq\Delta^{T}H\Delta-\frac{\eta_{8}}{2}\|\Delta\|^{4}+O(\lVert\Delta\rVert^{5-\mu}) \\
&\geq\frac{\eta_{8}}{2}\|\Delta\|^{4}+\Theta(\lVert\Delta\rVert^{5-\mu})\geq\frac{\eta_{8}}{4}\|\Delta\|^{4}.
\end{align*}
Thus, we have $\|\grad{f({\bf y})}\|\geq \Theta(\|\Delta\|^{3})$ and the KL exponent is $\frac{1}{4}$.
\end{proof}

\section{Conclusions}
\label{sec:conclusion}

In this paper, we analyze the geometric properties of a class of
quartic-quadratic optimization problems under a single spherical constraint.
When the matrix $A$ in the quadratic form is diagonal, the stationary points and
local minima can be fully characterized and we show that the minimization
problem does not possess any spurious local minima. Furthermore, a closed-form
expression for global minimizer is available which is based on the projection
onto the $n$-simplex. If $A$ is a rank-one matrix, a similar analysis can be
performed and we derive characteristic properties of associated local minima and
uniqueness of global minima up to a certain phase shift. We verify that the
problem satisfies a Riemannian-type strict-saddle property in the real space
when the interaction coefficient is at least of order $O(n^{3/2})$ or
sufficiently small which corresponds to the case where either the quartic or the
quadratic part is the leading term of the objective function. Finally, we
estimate the Kurdyka-\L ojasiewicz exponent $\theta$ of problem \eqref{eqn:obj}
and show that $\theta$ is $\frac{1}{4}$ for all stationary points ${\bf z}$ if
$A$ is diagonal or if the problem is restricted to the real space and ${\bf z}$
fulfills a certain global optimality condition.

\bibliographystyle{spmpsci}
\bibliography{references}

\end{document}